\documentclass[12pt]{article}
\usepackage{amsmath, latexsym, amsfonts, amssymb, amsthm, amscd}
\usepackage[left=1.5cm, right=1.5cm, top=2.5cm, bottom=2.5cm]{geometry}
\usepackage{upquote}
\usepackage{color}
\usepackage{setspace}
\usepackage[labelfont=bf]{caption}

\usepackage[utf8]{inputenc}
\usepackage[T1]{fontenc}
\usepackage[english]{babel}
\usepackage{dsfont}
\usepackage{mathtools} 
\usepackage{enumitem}
\usepackage{makecell}
\usepackage{mathrsfs}
\usepackage{dsfont}
\usepackage{multirow}
\usepackage{xcolor}
\definecolor{Maroon}{HTML}{ad2231}
\definecolor{webgreen}{HTML}{008000}

\usepackage{hyperref}
\hypersetup{colorlinks, breaklinks, urlcolor=Maroon, linkcolor=Maroon, citecolor=webgreen} % Set link colors

\allowdisplaybreaks

\usepackage{bookmark}


\newtheorem{corollary}{Corollary}
\newtheorem{proposition}{Proposition}
\newtheorem{lemma}{Lemma}
\newtheorem{remark}{Remark}

\newtheorem{theorem}{Theorem}

\newtheorem{example}{Example}

\newtheorem{conj}{Conjecture}
\theoremstyle{definition}

\newtheorem*{general*}{General assumption}

\usepackage{todonotes}
\newcommand{\Gabriel}[1]{\todo[color=orange!70,inline]{Gabriel: #1}}
\begin{document}
\title{Branching processes with pairwise interactions}
\author{Gabriel Berzunza Ojeda\footnote{ {\sc Department of Mathematical Sciences, University of Liverpool, United Kingdom.} E-mail: gabriel.berzunza-ojeda@liverpool.ac.uk}\, \, and \, \, Juan Carlos Pardo \footnote{ {\sc Centro de Investigaci\'on en Matem\'aticas A.C., Mexico.} E-mail: jcpardo@cimat.mx}}

\maketitle

\vspace{0.1in}

\begin{abstract} 
In this manuscript, we are interested in the long-term behaviour of branching processes  with pairwise interactions (BPI-processes).  A process in this class behaves as a pure branching process with the difference that competition and cooperation events between pairs of individuals are also allowed. Here, we provide a series of integral tests that explain how  competition and cooperation regulate the long-term behaviour of BPI-processes. In particular, such integral tests describe the events of explosion and  extinction; and provide conditions under which the process comes down from infinity. Moreover, we also determine whether the process admits, or not,  a stationary distribution. Our arguments use a random time change representation in terms of a modified branching process with immigration and moment duality. The moment dual of BPI-processes  turns out to be a family of diffusions taking values on $[0,1]$ which are interesting in their own right  and that we introduce as generalised Wright-Fisher diffusions. \\

\noindent {\sc Key words and phrases}: Branching processes with interactions, moment duality, generalised Wright-Fisher diffusions, Lamperti transform, explosion, extinction, stability, coming down from infinity.

\bigskip

\noindent MSC 2020 subject classifications: 60J80, 60K35.
\end{abstract}

%%%%%%%%%%%%%%%%%%%%%%%%%%%%%%%%%%%%%%%%%%%%%%%%%%%
\section{Introduction and main results}
%%%%%%%%%%%%%%%%%%%%%%%%%%%%%%%%%%%%%%%%%%%%%%%%%%%

A branching process or Bienaym\'e-Galton-Watson process (BGW) is a continuous-time and discrete state-space Markov chain where particles or individuals (of the same nature) give birth independently at some constant rate to a random number of offspring or die independently at constant rate  (see for e.g., the monographs of  Harris \cite{Harris2002} or Athreya and Ney \cite{Ath2004} for further details). Besides their elegance and tractability, it is  well-known  that branching processes are offset by the difficulty of predicting that the population will, with probability one, either die out or grow without bound. On the event of dying out, the process is absorbed at $\{0\}$ in finite time. On the other hand, if the process does not die out, it can grow indefinitely or be absorbed at $\{\infty\}$ in finite time. The latter event is known as explosion and occurs typically when the process performs infinitely many large jumps in  finite time with positive probability.
 
In order to circumvent various unrealistic properties of pure branching processes, as their degenerate long-term behaviour, several authors have introduced generalisations of these processes by incorporating  interactions between individuals. Following the studies of Kalinkin \cite{Ka82, Ka99, Ka01,Ka02, Ka02-1}, Lambert \cite{La2005} (see also Jagers \cite{Jagers1994}) and the recent manuscript of Gonz\'alez Casanova et al.\ \cite{Pa2017}, one approach consists in generalise the birth and death rates by considering polynomial rates as functions of the population size  that can be interpreted as a different type of interactions between individuals. For example, {\sl competition} pressure  and {\sl cooperation} which are ubiquitous fundamental mechanisms in biology  in various space and time scales, and thus deserve special attention. It is important to mention that continuous-time BGW processes with polynomial rates similar to the one we focus on here (but not the same) can also be found in Chen \cite{Chen1997} and Pakes \cite{Pakes2007}. The authors in these papers make stronger integrability assumptions on the offspring distributions than the ones we consider here. Moreover, the approach we develop and the results we obtain are quite different.

Formally speaking, we study continuous-time and discrete state-space Markov chains, in which {\sl natural births} and {\sl natural deaths} both occur at rates proportional to  current population size (as in the continuous-time BGW process) but  also  we consider additional birth and death events, due to cooperation and competition, that occur at  rates proportional to the square of the population size. We call these Markov chains branching processes with pairwise interaction  (or BPI-processes for short). More precisely, consider the parameters $c, d \geq 0,  (b_{i})_{i \geq 1}$ and $(\pi_{i})_{i \geq 1}  $ such that  $b_{i}, \pi_{i} \geq 0 $, for $i\ge 1$, and
\begin{align} \label{eq7}
b \coloneqq \sum_{i \geq 1} b_{i}  < \infty \qquad \text{and} \qquad \rho \coloneqq \sum_{i \geq 1} \pi_{i}  < \infty.
\end{align}

\noindent The BPI-process $Z=(Z_{t}, t \geq 0)$ is a continuous-time Markov process whose dynamics are described through its  infinitesimal generator  matrix $Q = (q_{i,j})_{i,j \in \mathbb{N}_{0}}$ which is given by
\begin{align} \label{BPIQmatrix}
q_{i,j} = \left\{ \begin{array}{ll}
              i \pi_{j-i} + i(i-1)b_{j-i} & \mbox{  if }  i \geq 1 \, \, \text{and} \, \, j >i, \\
              di +c i (i-1)  & \mbox{  if }  i \geq 1 \, \, \text{and} \, \, j = i-1, \\
              -i(d + \rho + (b + c) (i-1)  )  & \mbox{  if }  i \geq 1 \, \, \text{and} \, \, j =i,  \\
              0 &  \text{otherwise}. \\
              \end{array}
    \right.
\end{align}
\noindent Here $\mathbb{N}_{0} \coloneqq \mathbb{N} \cup \{0\}$. Note that a BPI-process can also be understood as the counting process of a particle system where particles may coalesce (i.e.\ competition event) or fragment  with or without collisions (i.e.\ cooperation and branching events,  respectively). We denote by $\mathbb{P}_{z}$ the law of $Z$ when issued from $z \in\mathbb{N}_{0}$. 

Observe  that the total natural birth rate, given that the size of the population is $z \in \mathbb{N}$, is thus $\rho z$ and the natural death rate $dz$. In particular,  when $\rho>0$, the parameter $\pi_{i}/\rho$ represents the probability of having $i$ new individuals born at each reproduction event. Then, the function
\begin{align} \label{eq9}
\Psi_{d,\rho}(u) \coloneqq d - (\rho +d)u + \sum_{i \geq 1} \pi_{i}u^{i+1}, \hspace*{5mm} u \in [0,1],
\end{align}
\noindent characterises completely the underlying branching mechanism of the BPI-process, i.e.\ when $c = b=0$, the BPI-process becomes a continuous-time BGW process whose branching mechanism is given by $\Psi_{d, \rho}$. In the sequel, we refer to  $\Psi_{d,\rho}$  defined in  \eqref{eq9} as the  {\it branching mechanism}.

On the other hand, the extra death rate $cz(z - 1)$ models competition pressure, i.e., deaths may occur if one of the $z$ individuals selects another from the remaining $z-1$ ones at constant rate $c$ and then kills it. The extra birth rate $bz(z-1)$ corresponds to the  interaction coming from cooperation, i.e.\ when $b>0$,  each pair of individuals interact and produce $i$ new individuals with probability $b_{i}/b$. These interactions motivate the introduction of the function 
\begin{align} \label{eq15}
\Phi_{c,b}(u) \coloneqq  c-(c+b)u + \sum_{i \geq 1} b_{i} u^{i+1}, \hspace*{5mm} u \in [0,1],
\end{align}
\noindent which characterises the cooperation and competition mechanisms of the BPI-process. In the sequel, we refer to $\Phi_{c,b}$ as the {\it interaction mechanism}. 

An interesting and special case of the already introduced family of processes is the so-called logistic branching process (or LB-process for short), i.e., when $c >0$ and $b=0$. The LB-process was introduced and studied deeply by Lambert in \cite{La2005} under the following log-moment condition 
\begin{align}\label{logmoment}
\sum_{i \geq 1} \pi_{i}  \ln (i) < \infty.
\end{align} 
Another important example is  the corresponding counting block process of the so-called Kingman's coalescent \cite{Kingman1982}  which corresponds to the absence of natural deaths, branching and cooperation  events in our model (i.e., $d=\rho = b=0$). 

Models with cooperation have appeared in the literature before, for instance we mention  Sturm and Zwart \cite{Sturm2015} and more recently  Gonz\'alez Casanova et al.\ \cite{Pa2017}, where  a more general model is considered which includes   other types of interactions  besides competition and cooperation, namely {\it annihilation} (which is a pairwise interaction event) and {\it catastrophes} (which is a multiple interaction event). Despite the fact that  annihilation is a pairwise interaction event, we exclude it from our model due to its complexity (it produces negative jumps of size two) and it cannot be covered with the techniques we develop here.    

According to the infinitesimal generator $Q$ in \eqref{BPIQmatrix}, we have the following classification of states of BPI-processes.
\begin{table}[h]
\centering
\begin{tabular}{ |c|c|c|c|} 
 \hline
 $\{0 \}$ & $\{1\}$ & If $\rho >0$ & If $\rho = 0$ and $b>0$  \\ 
 is absorbing & is absorbing & $Z$ is irreducible  &  $Z$ is irreducible  \\ 
 & when  $d=\rho=0$ & on $\mathbb{N}$ whenever &  on $\mathbb{N} \setminus \{1\}$ whenever  \\  
   & & $d>0$ or $c>0$ & $d > 0$ or $c >0$\\
 \hline
\end{tabular}
\caption{Classification of  states of  BPI-processes.}
\end{table}

\noindent Let $\zeta_{\infty}$ be the (first) explosion time of a BPI-process $Z$, i.e $\zeta_{\infty} \coloneqq \inf \{t  \geq 0: Z_{t} = \infty \}$, with the usual convention that $\inf \varnothing = \infty$. In this work, we shall consider that the BPI-process is sent to a cemetery point, let us  say $\infty$, after explosion and remains there forever (i.e.\ $Z_{t} = \infty$ if $t \geq \zeta_{\infty}$); we refer to Chapters 2 and 3 of  the monograph of Norris \cite{Norris1998} and references therein for background on continuous-time Markov chains.

The aim of this manuscript is to study the long-term behaviour of BPI-processes with  general branching and interaction mechanisms; and the precise understanding of how  competition and cooperation regulates their growth. In other words, we are interested in  studying the nature of the absorbing states $\{0 \}$ (extinction of the population), $\{1\}$ and $\{\infty\}$ (explosion of the population) as well as whether or not the processes have a stationary distribution.  In this direction, we define 
\begin{align*} 
\zeta_{0} \coloneqq \inf \{t  \geq 0: Z_{t} = 0 \} \qquad \text{and} \qquad \zeta_{1} \coloneqq \inf \{t  \geq 0: Z_{t} = 1 \}, 
\end{align*}
\noindent the first hitting times of the states $\{0\}$ and $\{1\}$, respectively, with the usual convention that $\inf \varnothing = \infty$. %We also recall that $\zeta_{\infty}$ denotes the first hitting time of the state $\{\infty\}$. 
We say that the BPI-process $Z$ explodes with positive probability  if and only if $\mathbb{P}_{z}(\zeta_{\infty} < \infty) > 0$ (otherwise, we say it is conservative) and that it becomes extinct  with positive probability if and only if $\mathbb{P}_{z}(\zeta_{0} < \infty) > 0$. A rather surprising phenomenon is that  competition does not always prevent explosion while cooperation can also not prevent extinction. Nevertheless, the competition pressure as the population size increases could cause a compensation near $\{\infty\}$ in order to allow the population to survive forever, that is,  the process will never reach the absorbing states $\{0\}$ and $\{\infty\}$.  

In the case of  LB-processes (i.e., $c>0$ and $b=0$),  Lambert \cite{La2005}  found that  the log-moment condition \eqref{logmoment}  is sufficient (but not necessary) to guarantee that the process never explodes; see Theorems 2.2 and  2.3 in \cite{La2005}. Moreover,  under \eqref{logmoment}, it is observed in \cite{La2005}   that the competition alone has no impact on the extinction of the population. Indeed,  when $d=0$, the process admits a stationary distribution and, when $d>0$, the process gets absorbed at $\{0\}$ almost surely. The event of explosion of LB-processes was not formally studied in \cite{La2005} and one of our aims is to describe such behaviour.

Recently, Foucart \cite{Cle2017}  studied the event of extinction for the continuous-state  branching process with competition (or logistic CSBP) improving the results of Lambert \cite{La2005}  in the continuous-time and continuous-state setting. Furthermore, Foucart \cite{Cle2017} established a criterion for  explosion using the fact that the underlying process in the Lamperti transform of the logistic CSBP is a generalised Ornstein-Uhlenbeck process. The explicit knowledge of the   Laplace transform of generalised Ornstein-Uhlenbeck processes  is a fundamental tool to determine the event of explosion of the logistic CSBP. However, the LB-process cannot be seen as a logistic CSBP. Specifically, competition in \cite{Cle2017} (see Definition 2.1 there) appears as a negative non-linear drift while in our case competition corresponds to an extra negative jump which makes the  study of the asymptotic behaviour of LB-processes more involved. In fact, such negative jump implies that the generating function of the underlying process of the LB-process, in the Lamperti transform,  cannot be computed  explicitly except for two particular cases, see Proposition \ref{Pro1} below. In the general scenario, i.e.\ when there is cooperation, Gonz\'alez Casanova et al.\ \cite{Pa2017} (see Theorems 1 and 3) have proven similar results under the rather strong assumption that the mean of the offspring size at each birth time is finite, i.e.\ $\sum_{i \geq 1} i \pi_{i} < \infty$. In this case,  the study of the non-explosion probability  is based on a Foster-Lyapunov criteria developed by Meyn and Tweedie \cite{MT93} where the finiteness of the mean  of the offspring distribution plays a crucial role. We emphasize that the case considered  in \cite{Pa2017} includes  \textit{annihilation} and {\it catastrophes} events that we cannot cover  with the techniques developed in this manuscript. Nevertheless,  our results  provide a straightforward  stochastic bounds (via a coupling argument) for  the model considered in \cite{Pa2017} in  some of the events that we are interested on.\\

In this manuscript,  we  provide  integral tests under which a BPI-process explodes in finite time,   becomes extinct or absorbed at $\{1\}$. 
Moreover, we  also give an integral test for  the existence of a stationary distribution and a 
condition under which the process comes down from infinity.  Throughout this work and unless we specify otherwise, we always assume that 
\begin{equation*}  \label{main}
c > 0 \hspace{5mm} \text{and} \hspace*{5mm} \varsigma \coloneqq-c+\sum_{i \geq 1} ib_{i} \leq 0. \tag{${\bf H}$}
\end{equation*}
Observe that when $\varsigma<0$, we necessarily have that $c>0$. When $\varsigma=0$ and $c=0$, the process $Z$ is nothing but a BGW-process whose long-term behaviour has been well studied. In other words, under assumption (\ref{main}), we always assume that there is competition and that the mean of the offspring at each birth event given by cooperation is bounded by the competition. 

Following the terminology in Gonz\'alez Casanova et al.\ \cite{Pa2017}, we say that  a BPI-process is in the  {\it supercritical, critical }or {\it subcritical cooperative }regime accordingly as  $\varsigma>0$, $\varsigma=0$  or $\varsigma<0$.  Here, the supercritical cooperative regime is not considered  since the approach we develop basically  cannot be applied in this case as we will see later. Nevertheless, Theorem 1 in \cite{Pa2017} shows that in this regime, and when $\sum_{i \geq i} i \pi_{i} < \infty$,  the process may explode in finite time with positive probability. On the other hand,  the long-term behaviour in the critical cooperative regime has not been treated in \cite{Pa2017} but during the preparation of this manuscript we knew that some new developments have been done in the general setting (i.e.\ including annihilation and catastrophes as in \cite{Pa2017}) by Gonz\'alez Casanova et al.\ \cite{Pa20192} under the conditions $\sum_{i \geq 1} i \pi_{i} < \infty$ and $\sum_{i \geq 1} i^2 b_{i} < \infty$.

Before we present our main results, we introduce some notation. For  $x,y\in(0,1)$, we define
\begin{eqnarray*}
\mathcal{Q}_{d,\rho}^{c,b}(y;x)\coloneqq \int_{y}^{x} \frac{\Psi_{d,\rho}(w)}{w\Phi_{c,b}(w)} {\rm d} w, \qquad \mathcal{J}_{d,\rho}^{c,b}(y;x)\coloneqq \exp \left( \mathcal{Q}_{d,\rho}^{c,b}(y;x) \right), \qquad \mathcal{S}_{d,\rho}^{c,b}(y;x)\coloneqq \int_y^x \frac{{\rm d} u}{\mathcal{J}_{d,\rho}^{c,b}(y;u)} ,
\end{eqnarray*}
\begin{eqnarray*}
\mathcal{R}_{d,\rho}^{c,b}(y; x)\coloneqq \int_{y}^{x} \frac{1}{u\Phi_{c,b}(u)} \mathcal{J}_{d,\rho}^{c,b}(y;u){\rm d} u, \qquad \mathcal{E}_{d,\rho}^{c,b} (y; x)\coloneqq\int_{y}^{x}  \mathcal{R}_{d,\rho}^{c,b}(u; x) {\rm d} u,
\end{eqnarray*}

\noindent and 
\begin{eqnarray*}
\mathcal{I}_{d, \rho}^{c, b}(y; x) \coloneqq  \int_{y}^{x} \frac{1}{\mathcal{J}_{d, \rho}^{c, b}(y; u)} \mathcal{R}_{d, \rho}^{c, b}(y; u) {\rm d}u. 
\end{eqnarray*}

\noindent The above quantities are always finite whenever $x,y\in(0,1)$. For $\theta \in (0,1)$ and $i=0,1$, we also define 
\begin{eqnarray*}
\mathcal{Q}_{d,\rho}^{c,b}(\theta;i)\coloneqq \lim_{x \rightarrow i} \mathcal{Q}_{d,\rho}^{c,b}(\theta;x), \qquad \mathcal{J}_{d,\rho}^{c,b}(\theta;i)\coloneqq \lim_{x \rightarrow i} \mathcal{J}_{d,\rho}^{c,b}(\theta;x), \qquad \mathcal{S}_{d,\rho}^{c,b}(\theta;i)\coloneqq \lim_{x \rightarrow i} \mathcal{S}_{d,\rho}^{c,b}(\theta;x),
\end{eqnarray*}
\begin{eqnarray*}
\mathcal{R}_{d,\rho}^{c,b}(\theta;i)\coloneqq \lim_{x \rightarrow i} \mathcal{R}_{d,\rho}^{c,b}(\theta;x), \qquad \mathcal{E}_{d,\rho}^{c,b}(\theta;i)\coloneqq \lim_{x \rightarrow i} \mathcal{E}_{d,\rho}^{c,b}(\theta;x) \qquad \text{and} \qquad \mathcal{I}_{d,\rho}^{c,b}(\theta;i)\coloneqq \lim_{x \rightarrow i} \mathcal{I}_{d,\rho}^{c,b}(\theta;x).
\end{eqnarray*}

These last quantities may be finite or infinite. In particular, if they are finite (resp.\ infinite) for some $\theta \in (0,1)$ then they are finite (resp.\ infinite) for all $\theta \in (0,1)$. Most of the integral test that we present below will be in terms of the above quantities. 

%%%%%%%%%%%%%%%%%%%%%%%%%%%%%%%%%%%%%%%%%%%%%%
%EXPLOSION
\subsection{Explosion}
%%%%%%%%%%%%%%%%%%%%%%%%%%%%%%%%%%%%%

Before we start our exposition, we remark that the question of explosion for  continuous-time Markov chains is quite important, specially when we are interested in uniqueness. For instance, when there are no explosions, the solutions of the forward and backward equations are unique. There is a well-known necessary and sufficient condition for conservativeness in terms of the embedded jump chain which is usually hard to check. The integral tests that we present here are in terms of the parameters of the process and are quite sharp specially in the subcritical cooperative regime. We will present some illustrative examples, see  Section \ref{Sectionex} below, which show the applicability  of  our results  as well as the complexity of the event of explosion specially in the critical cooperative case, i.e.\ when  $\varsigma=0$.
    
Our main first result considers subcritical cooperative BPI-processes, i.e. when $\varsigma<0$, and provides a necessary and sufficient condition for the probability of explosion in the case when $d=0$. For the case $d>0$, we provide conditions under which the process is conservative or it explodes with positive probability.
\begin{theorem}\label{Theo1}
Let $z \in \mathbb{N}$, $\rho>0$,  $b\ge0$ and  $\varsigma< 0$. 
\begin{itemize}
\item[(i)]  For $d=0$, the associated BPI-process explodes with positive probability (i.e.\ $\mathbb{P}_{z}(\zeta_{\infty} < \infty) >0$) or it is conservative accordingly as $\mathcal{R}_{0,\rho}^{c,b}(\theta; 1)$, for some $\theta \in (0,1)$, is finite or infinite.

\item[(ii)] Set $\tilde{c}=c\lor d$. The associated BPI-process explodes with positive probability (i.e.\ $\mathbb{P}_{z}(\zeta_{\infty} < \infty) >0$) if $\mathcal{R}_{\tilde{c},\rho}^{\tilde{c},0} (\theta; 1)< \infty$, for some $\theta \in (0,1)$.

\item[(iii)] If $d >0$ and $\mathcal{R}_{d,\rho}^{c,b}(\theta; 1) = \infty$, for some $\theta \in (0,1)$, then the associated BPI-process is conservative.
\end{itemize}
\end{theorem}

 The proof of the above result relies in a Lamperti-type argument similarly as in Foucart \cite{Cle2017} and Leman and Pardo \cite{Leman}. Unfortunately here the situation is very different, and in some sense more complicated, than in \cite{Cle2017} or in \cite{Leman} since   the moment generating function of the underlying Markov chain, in the Lamperti-type time change, can only be computed explicitly in the following two cases: when $d=0$ or $c=d$ and $b=0$; see  Proposition \ref{Pro1}. In other words, with our methods, we can only get  a necessary and sufficient condition when $d=0$ and when $c=d$.  For the case $d>0$,  we use a coupling argument which depends on the values of the parameters  $c$ and $d$; and  explains the differences between parts (ii) and (iii). 

We also remark that  Theorem  \ref{Theo1} also  includes the case of  LB-processes, i.e.\ when $b=0$, which is the model studied in \cite{La2005}. Up to our knowledge, the extinction event of LB-processes has not been studied yet. In this particular case, the previous result, when combined with Lemma \ref{Pro3} below, can be simplified and we get a necessary and sufficient condition for the case when the natural death rate  $d$ is smaller or equal to the competition rate $c$. We state such a result as a Corollary for future reference.

\begin{corollary}[LB-processes] \label{newcoro}
Suppose that $z \in \mathbb{N}$, $c>0$, $b=0$ and $\rho >0$. Then, 
\begin{itemize}
\item[(i)] if $c\ge d \geq 0$, the associated BPI-process explodes with positive probability (i.e.\ $\mathbb{P}_{z}(\zeta_{\infty} < \infty) > 0$)  if and only if $\mathcal{R}_{d,\rho}^{c,0}(\theta; 1) < \infty$, for some $\theta \in (0,1)$.
\item[(ii)] if $c<d$, the associated BPI-process explodes with positive probability  if $\mathcal{R}_{d, \rho}^{d,0}(\theta; 1)< \infty$, for some $\theta \in (0,1)$, or it is conservative if $\mathcal{R}_{d,\rho}^{c,0}(\theta; 1) = \infty$, for some $\theta \in (0,1)$.
\end{itemize}
\end{corollary}

Moreover, we can determine in some cases when the BPI-process explodes almost surely.

\begin{proposition}\label{lemma3}
Suppose that $z \in \mathbb{N}$, $\rho>0$ and $\varsigma< 0$. Then, we have the following:
\begin{itemize}
\item[(i)] Suppose that $d=0$ and that $\mathcal{R}_{0,\rho}^{c,b}(\theta; 1) < \infty$,  for some $\theta \in (0,1)$. Then, the associated BPI-process explodes almost surely (i.e.\ $\mathbb{P}_{z}(\zeta_{\infty} < \infty) =1$).

\item[(ii)] Suppose that $b=0$ (i.e., the LB-process case), $c \geq d \geq 0$ and that $\mathcal{R}_{d,\rho}^{c,0} (\theta; 1)< \infty$, for some $\theta \in (0,1)$. Then, the associated BPI-process explodes almost surely if and only if $d=0$.
\end{itemize}
\end{proposition}

Unlike the subcritical cooperative case,  the explosion event in the critical cooperative case (\ $\varsigma=0$)   seems to be rather complicated.   In fact a new condition appears that depends on the value of $\mathcal{E}_{\cdot,\rho}^{\cdot,\cdot}(\theta; 1)$. For technical reasons, we cannot capture completely the event of explosion as in the subcritical cooperative regime. More precisely, our next result may exclude some simple cases, as we will explain below, and cannot cover completely Theorem 1 in \cite{Pa2017} where it is assumed that $\sum_{i \geq i} i \pi_{i} < \infty$.  Nonetheless, it includes some interesting cases where $\sum_{i \geq i} i \pi_{i} $ may be  infinite, see for instance Examples \ref{example2new} and \ref{example3new} below. In some sense, our next result complements  Theorem 1 in \cite{Pa2017}. We also point out that in this case the event of explosion depends on the values of  $\mathcal{R}_{\cdot,\rho}^{\cdot,\cdot} (\theta; 1)$ and $\mathcal{E}_{\cdot,\rho}^{\cdot,\cdot}(\theta; 1)$.

\begin{theorem} \label{Theo3}
Let $z \in \mathbb{N}$, $\rho, b>0$ and  $\varsigma=0$.  
\begin{itemize}
\item[(i)]  For $d=0$, the associated BPI-process explodes almost surely (i.e.\ $\mathbb{P}_{z}(\zeta_{\infty} < \infty) =1$)  if $\mathcal{R}_{0,\rho}^{c,b}(\theta; 1) < \infty$ and $\mathcal{E}_{0,\rho}^{c,b}(\theta; 1) < \infty$, for some $\theta \in (0,1)$. 

\item[(ii)] Set $\tilde{c}=c\lor d$. The associated BPI-process explodes with positive probability if $\mathcal{R}_{\tilde{c},\rho}^{\tilde{c},0} (\theta; 1)< \infty$, for some $\theta \in (0,1)$.

\item[(iii)] If $d \ge 0$ and $\mathcal{R}_{0,\rho}^{c,b}(\theta; 1) = \infty$, for some $\theta \in (0,1)$, then the associated BPI-process is conservative. 
\end{itemize}
\end{theorem}
Similarly as in the proof of Theorem \ref{Theo1}, our arguments relies in a Lamperti-type technique and a coupling argument. We observe that in (ii) and (iii), the value of $\mathcal{E}_{0,\rho}^{c,b}(\theta; 1)$ do not appear but it is implicit there as we will explain below.  In (ii), the critical cooperative BPI-process is stochastically lower bounded by a subcritical cooperative BPI-process with $b=0$. Then  according to Lemma \ref{NEWlemma1} below, the condition $\mathcal{R}_{\tilde{c},\rho}^{\tilde{c},0} (\theta; 1)< \infty$ implies that $\mathcal{E}_{\tilde{c},\rho}^{\tilde{c},0} (\theta; 1)< \infty$ and thus, in this case, only the finiteness of $\mathcal{R}_{\tilde{c},\rho}^{\tilde{c},0} (\theta; 1)$ is necessary. In (iii), we use Remark \ref{remark8}  below which says that  whenever $\mathcal{R}_{0,\rho}^{c,b}(\theta; 1)=\infty$, we necessarily have  $\mathcal{E}_{0,\rho}^{c,b} (\theta; 1)=\infty$. In other words, we have that $\mathcal{R}_{0,\rho}^{c,b}(\theta; 1)=\infty$ is enough to get that the associated BPI-process is conservative.  Indeed, in the subcritical cooperative regime, the value of $\mathcal{E}_{\cdot,\cdot}^{\cdot,\cdot} (\theta; 1)$ also plays a role in Theorem \ref{Theo1}, but Lemma \ref{NEWlemma1} and Remark \ref{remark8} imply that  its value is always in accordance with the value of $\mathcal{R}_{\cdot,\cdot}^{\cdot,\cdot} (\theta; 1)$.

%In particular, the BPI-processes will always be conservative in the following cases.

%Moreover, we can determine when BPI-processes explode almost surely and when they will always be  conservative as outlined in the next two Propositions.

\begin{proposition} \label{Pro5}
Suppose that $z \in \mathbb{N}$ and $\rho = 0$, then the associated BPI-process is conservative. In particular, 
\begin{itemize}
\item[(i)] if  $d >0$, then  $\mathbb{P}_{z}(\zeta_{0} < \infty) = 1$, 
\item[(ii)] if  $d = 0$, then  $\mathbb{P}_{z}(\zeta_{0} < \infty) = 0$ and $\mathbb{P}_{z}(\zeta_{1} < \infty) = 1$. 
\end{itemize}
\end{proposition}

The following corollary is useful and follows from Theorems \ref{Theo1} and \ref{Theo3} together with Lemmas \ref{lemma13} and \ref{Pro3} which appear below. It provides a simpler condition for a subcritical cooperative BPI-process  to be conservative.  For simplicity of exposition, we defer its proof to the Appendix.
\begin{corollary}  \label{corollary4} 
\begin{itemize}
\item[(i)] If $\varsigma<0$ and $\mathcal{Q}_{d, \rho}^{c,b}(\theta;1) \in (-\infty, \infty)$, for some $\theta \in (0,1)$,  then the associated BPI-process issued from $z \in \mathbb{N}$ is conservative.
\item[(ii)]  If $\varsigma=0$  and $\mathcal{Q}_{0, \rho}^{c,b}(\theta;1) \in (-\infty, \infty)$, for some $\theta \in (0,1)$, then the associated BPI-process issued from $z \in \mathbb{N}$ is conservative.
\end{itemize}
\end{corollary}

Note that  a classical coupling argument between continuous-time Markov chains (see for e.g., \cite[Chapters IV and V]{Lindvall1992}) shows that, under our main hypothesis (\ref{main}), Theorems \ref{Theo1} and \ref{Theo3} may also imply similar results
%the result established in Theorem 1  in \cite{Pa2017} 
for more general classes of subcritical cooperative branching processes with interactions (i.e., including annihilation and catastrophes) as those studied in  \cite{Pa2017}. \\

We conclude our exposition about the explosion event with the following remark in the subcritical regime. The integral tests in Theorem \ref{Theo1}, and in Corollary \ref{newcoro}, can be simplified. Indeed for $0 < \theta < x < 1$, we define
\begin{align*}
\widetilde{\mathcal{Q}}_{\rho}^{c, b}(\theta; x) = \frac{1}{\varsigma}\int_{\theta}^{x} \sum_{i \geq 1} \overline{\pi}_i w^{i-1} {\rm d} w, 
\end{align*}
\noindent which is always finite. For $\theta \in (0,1)$, we also define $\widetilde{\mathcal{Q}}_{\rho}^{c, b}(\theta; 1) = \lim_{x \uparrow 1}\widetilde{\mathcal{Q}}_{\rho}^{c, b}(\theta; x)$ which may be finite or infinite and observe that  it is equivalent to $\mathcal{Q}_{\rho}^{c, b}(\theta; 1)$.
\begin{lemma} \label{lemma1Red}
Suppose that $\varsigma<0$, then, for some $\theta \in (0,1)$, $\mathcal{Q}_{d, \rho}^{c, b}(\theta; 1) \in  (- \infty, \infty)$ if and only if $\widetilde{\mathcal{Q}}_{\rho}^{c, b}(\theta; 1) \in (-\infty, \infty)$. Moreover, there exists $\theta^{\prime} \in (0,1)$ such that for $\theta^{\prime} < \theta$, $\mathcal{Q}_{d, \rho}^{c, b}(\theta; x) \sim \widetilde{\mathcal{Q}}_{\rho}^{c, b}(\theta; x)$, as $x \uparrow 1$. 
\end{lemma}

In particular, by Corollary~\ref{corollary4} part (i), a subcritical cooperative BPI process is conservative if $\widetilde{\mathcal{Q}}_{\rho}^{c, b}(\theta; 1) \in (-\infty, \infty)$, for some $\theta \in (0,1)$. Indeed, the finiteness of $\widetilde{\mathcal{Q}}_{\rho}^{c, b}(\theta; 1)$ does not depend on the interaction parameters and the natural death parameter, but only on the natural birth parameter. On the other hand, 
for $0 < \theta < x < 1$, define
\begin{align*}
\widetilde{\mathcal{R}}_{\rho}^{c}(\theta; x)=\int_\theta^x \frac{1}{1-u}\exp\left(-\frac{1}{c}\sum_{i\ge 1}\frac{\overline{\pi}_i u^i}{i}\right){\rm d} u, 
\end{align*}
\noindent which is always finite. For $\theta \in (0,1)$, define also $\widetilde{\mathcal{R}}_{\rho}^{c}(\theta; 1) = \lim_{x \uparrow 1} \widetilde{\mathcal{R}}_{\rho}^{c}(\theta; x)$, which may be finite or infinite and observe that  it is equivalent to $\mathcal{R}_{\rho}^{c, b}(\theta; 1)$.

\begin{lemma} \label{lemma2Red}
$\mathcal{R}_{d, \rho}^{c, 0}(\theta; 1) < \infty$, for some $\theta \in (0,1)$, if and only if $\widetilde{\mathcal{R}}_{\rho}^{c}(\theta; 1) < \infty$, for some $\theta \in (0,1)$. 
\end{lemma}

In particular, the conditions in Corollary~\ref{newcoro} can be simplified. The proofs of Lemmas \ref{lemma1Red} and \ref{lemma2Red} are given in Appendix. These last two lemmas suggest that the death rate $d$ must be irrelevant in the subcritical cooperative regime. Moreover, we believe that the following conjecture must hold.
\begin{conj}  
A subcritical cooperative BPI-process explodes with positive probability or it is conservative accordingly as $\widetilde{\mathcal{R}}_{\rho}^{c,b}(\theta; 1)$, for some $\theta \in (0,1)$, is finite or infinite, where
\begin{align*}
\widetilde{\mathcal{R}}_{\rho}^{c,b}(\theta; 1)=\int_\theta^1 \frac{1}{1-u}\exp\left(\frac{1}{\varsigma}\sum_{i\ge 1}\frac{\overline{\pi}_i u^i}{i}\right){\rm d} u.
\end{align*}
\end{conj}

For the ease of readability,we will present some illustrative examples of the applicability of the previous results in Section \ref{Sectionex}. We then continue with the presentation of our results on the probability of extinction, the stationary distribution and the phenomena of coming down from infinity.\\

In what follows, we will assume that  one of the following conditions is satisfied 
\begin{equation*}  \label{mainE1}
 b=0  \quad\text{and} \quad \mathcal{R}_{d,\rho}^{c,0}(\theta;1) =  \infty  \quad \text{or} \quad
 b >0  \quad\text{and}\quad  \mathcal{R}_{0,\rho}^{c,b}(\theta;1) =  \infty. \tag{${\bf E}$}
\end{equation*}
\noindent That is, according to Theorems \ref{Theo1} and \ref{Theo3}, assumption \eqref{mainE1} implies that the associated BPI-process is always conservative. Therefore, under assumption \eqref{mainE1}, it is natural to study the event of extinction and the existence of a stationary distribution.

%%%%%%%%%%%%%%%%%%%%%%%%%%%%%%%%%%%%%%%%%%%%%
% EXTINCTION
\subsection{Extinction}
%%%%%%%%%%%%%%%%%%%%%%%%%%%%%%%%%%%%%%%%%%%%%

The main result of this section identifies the conditions under which subcritical and critical cooperative BPI-processes become extinct almost surely.

%For our next result, we assume that  $d >0$ and  write $\mathcal{R}_{d,\rho}^{c,b}(0;1) \coloneqq \mathcal{R}_{d,\rho}^{c,b}(\theta;1) - \mathcal{R}_{d,\rho}^{c,b}(\theta;0)$, for $\theta \in (0,1)$. In particular, $\mathcal{R}_{d,\rho}^{c,b}(0;1)$ is finite whenever  $\mathcal{R}_{d,\rho}^{c,b}(\theta;1)$ is finite.
\begin{theorem}[Extinction] \label{lemma4}
\noindent Suppose that $d >0$ and $z \in \mathbb{N}$. 
\begin{itemize}
\item[(i)] If $\mathcal{S}_{d,\rho}^{c,b}(\theta,1)< \infty$, for some $\theta \in (0,1)$, then $\mathbb{P}_{z} (\zeta_{0} < \infty)= 1$. 
\item[(ii)] Otherwise, $\mathbb{P}_{z} (\zeta_{0} < \infty)=0$. 
\end{itemize}
\end{theorem}
%Note that the quantity on the right-hand side of the probability in Theorem \ref{lemma4} (ii) is well-defined since the integral is less than or equal to $\mathcal{R}_{d,\rho}^{c,b}(0;1)$. 

Next, following Lambert \cite{La2005}, we provide a deeper insight into the law of the extinction time $\zeta_{0}$. With this purpose in mind, we will show how to extend the results established in \cite{La2005} to our case. The proofs are based in a slight adaptation to those of \cite{La2005}, and thus, we leave the details to the interested reader to keep this article of a reasonable size; see Section \ref{sec4} below. Nevertheless, we believe that these results are important in their own right and must be reported for future reference.

For what follows, we assume that $\mathcal{Q}_{d,\rho}^{c,b}(\theta; 1)  \in (-\infty, \infty)$, for some $\theta \in (0,1)$, which clearly implies that $\mathcal{S}_{d,\rho}^{c,b}(\theta; 1) < \infty$. Hence if $d>0$, from Theorem \ref{lemma4}, we have that $Z$ gets absorbed at $0$ in finite time almost surely. This allows us to define the function $m(u) \coloneqq \mathcal{Q}_{d,\rho}^{c,b}(u; 1)$, for $u \in (0,1]$. We also define  the nonnegative decreasing function $\tau : (0,1] \rightarrow [0, \xi)$ by
\[
\tau(u) \coloneqq \int_{u}^{1}e^{m(v)} {\rm d}v, \quad \textrm{for} \quad u \in (0,1], 
\]
where $\xi \coloneqq \int_{0}^{1}e^{m(v)} {\rm d}v \in (0, \infty]$. The mapping $\tau$ is a bijection whose inverse on $[0, \xi)$ is denoted  by $\varphi$. In particular $\varphi^{\prime}(u) = - e^{-m(\varphi(u))}$, for $u \in [0, \xi)$. Finally, we introduce the function
\begin{eqnarray*}
H_{q,z}(u) \coloneqq \int_{0}^{\infty} e^{-qt} \mathbb{E}_{z}[u^{Z_{t}}] {\rm d}t, \hspace*{4mm} \text{for} \hspace*{2mm} z \in \mathbb{N}_{0}, \hspace*{2mm} u \in [0,1] \hspace*{2mm} \text{and} \hspace*{2mm} q >0. 
\end{eqnarray*}

\noindent If $u = e^{-\lambda}$, for $\lambda \geq 0$, then the function $H_{q,z}$ is just the Laplace transform of the $q$-resolvent measure of the BPI-process $Z$ started at $z$. But since $Z$ is an integer-valued process, it is more convenient to work with the current definition of $H_{q,z}$. The next result corresponds to Lemmas 2.1, 3.6 and 3.7 in \cite{La2005}. 

\begin{lemma} \label{lemma11}
Suppose that $\mathcal{Q}_{d,\rho}^{c,b}(\theta; 1)  \in (-\infty, \infty)$, for some $\theta \in (0,1)$, $d > 0$,  $z \in \mathbb{N}_{0}$ and $q >0$. 
\begin{itemize}
\item[(i)] As a function of $u \in [0,1)$, $H_{q,z} \in C^{2}([0,1), \mathbb{R})$ and it solves the second-order linear differential equation 
\begin{equation} \label{eq1}
- u \Phi_{c,b}(u) y^{\prime \prime}(u) - \Psi_{d,\rho}(u)y^{\prime}(u) + q y(u) = u^{z}. 
\end{equation} 
\end{itemize}

\noindent Let $(E_{q}^{h})$ be the homogeneous differential equation associated to (\ref{eq1}). Then,
\begin{itemize}
\item[(ii)] for any solution $(I, f_{q})$ of $(E_{q}^{h})$, for any open interval $I_{0} \subset I$ on which $f_{q}(\cdot)$ does not vanish, the function $g_{q}(u)  \coloneqq f^{\prime}_{q}(u)/f_{q}(u)$, for $u \in I_{0}$, solves on $I_{0}$ the Riccati differential equation 
\begin{equation} \label{eq3}
y^{\prime}(u) + y^{2}(u) + \frac{\Psi_{d,\rho}(u)}{u \Phi_{c,b}(u)} y(u) = \frac{q}{u \Phi_{c,b}(u)}. 
\end{equation} 

\item[(iii)] For any solution $(J, g_{q})$ of (\ref{eq3}), the function $h_{q}(u)  \coloneqq \exp(-m(\varphi(u)))g_{q}(\varphi(u))$, for $u \in \tau(J)$, solves on $\tau(J)$ the Riccati differential equation 
\begin{equation} \label{eq14}
y^{\prime}(u) - y^{2}(u)  = -q r^{2}(u), 
\end{equation} 

\noindent where $r(u) = |\varphi^{\prime}(u)| (\varphi(u) \Phi_{c,b}(\varphi(u)) )^{-1/2}$, for $u \in (0,\xi)$. %notice that $\Phi_{c,b}$ is nonegative by (\ref{main}).

\item[(iv)]  The equation (\ref{eq14}) has a unique nonnegative solution $w_{q}(\cdot)$ defined on $(0, \xi^-)$ and vanishing at $\xi^-$. In addition, $w_{q}$ is positive on $(0, \xi)$, and for any $u$ sufficiently small or large, $w_{q}(u) < \sqrt{q} r(u)$. As a consequence, $\int_{0}^{\xi}w_{q}(w) {\rm d} w$ converges, and $w_{q}$ decreases initially and ultimately. 
\end{itemize}
\end{lemma}

The next result corresponds to Theorem 3.9 in \cite{La2005} and it determines the law of $\zeta_0$.

\begin{theorem}\label{Lamtype}
Suppose that  $\mathcal{Q}_{d,\rho}^{c,b}(\theta; 1)  \in (-\infty, \infty)$, for some $\theta \in (0,1)$, $d > 0$, $z \in \mathbb{N}_{0}$ and $q >0$. Recall the function $w_{q}$ defined in Lemma \ref{lemma11} part (iv) and the notation $\xi \coloneqq \int_{0}^{1}e^{m(v)} {\rm d}v$. We have that 
\begin{eqnarray*}
q H_{q,z}(u) = 1-\int_{0}^{\tau(u)} e^{-\int_{t}^{\tau(u)}w_{q}(x) {\rm d}x }\left( \int_{t}^{\xi} qr^{2}(y)(1-\varphi^{z}(y)) e^{-\int_{t}^{y}w_{q}(x) {\rm d}x } {\rm d}y \right) {\rm d}t,
\end{eqnarray*}

\noindent for $u \in [0,1]$. In particular, we have
\begin{eqnarray*}
\mathbb{E}_{z}[e^{-q \zeta_{0}}] = 1-\int_{0}^{\xi} e^{-\int_{t}^{\xi}w_{q}(x) {\rm d}x }\left( \int_{t}^{\xi} qr^{2}(y)(1-\varphi^{z}(y)) e^{-\int_{t}^{y}w_{q}(x) {\rm d}x } {\rm d}y \right) {\rm d}t.
\end{eqnarray*}

\noindent Furthermore, the expectation of $\zeta_{0}$ is finite and is equal to
\begin{eqnarray*}
\mathbb{E}_{z}[\zeta_{0}] = \int_{0}^{\xi} x r^{2}(x)(1-\varphi^{z}(x)) {\rm d}x = \int_{0}^{1} \frac{1-y^{z}}{y \Phi_{c,b}(y)} e^{-m(y)} \left( \int_{y}^{1} e^{m(x)} {\rm d}x \right) {\rm d}y.
\end{eqnarray*}
\end{theorem}

%%%%%%%%%%%%%%%%%%%%%%%%%%%%%%%%%%%%%%%%%%
%STATIONARITY
%%%%%%%%%%%%%%%%%%%%%%%%%%%%%%%%%%%%%%%%%%%%
\subsection{Stationary distribution}

As we mentioned earlier, according to Theorem 2.2 in \cite{La2005},  when $d=0$ and the log-moment condition \eqref{logmoment} is satisfied, the LB-process is positive recurrent in $\mathbb{N}$ and it admits a stationary distribution. Thus, it is natural to ask under which conditions a (sub)critical cooperative BPI-process is positive recurrent in $\mathbb{N}$ and if it also admits a stationary distribution.

Our next result deals with the existence of a stationary distribution for (sub)critical cooperative BPI-processes. We first note that when $d = 0$, we necessarily have  $|\mathcal{S}_{0,\rho}^{c,b}(\theta; 0)| < \infty$. To see this, it is enough to observe that $d=0$ implies $\Psi_{0,\rho}(u) \leq 0$, for all $u \in [0,1]$. Let us assume that $d = 0$ and we write $\mathcal{S}_{0,\rho}^{c,b}(0;1) \coloneqq \mathcal{S}_{0,\rho}^{c,b}(\theta;1) - \mathcal{S}_{0,\rho}^{c,b}(\theta;0)$, for some $\theta \in (0,1)$, which is well-defined and may be infinite.

\begin{theorem}\label{lemma6}
Suppose that $d = 0$ and $z \in \mathbb{N}$. 
\begin{itemize}
\item[(i)] If $\mathcal{S}_{0,\rho}^{c,b}(\theta; 1)  < \infty$, for some $\theta \in (0,1)$, then $Z$ converges in law, as $t \rightarrow \infty$, towards a random variable $Z_{\infty}$ on $\mathbb{N}$ whose probability generating function is given by
\begin{eqnarray*}
\mathbb{E}\left[u^{Z_{\infty}} \right] \coloneqq  \frac{\mathcal{S}_{0,\rho}^{c,b}(0;u) }{\mathcal{S}_{0,\rho}^{c,b}(0;1)}, \hspace*{5mm} \quad \textrm{for}\quad u \in [0,1]. 
\end{eqnarray*}

In particular, if $\rho >0$, then the support of $Z_{\infty}$ is $\mathbb{N}$. If $\rho =0$, then $Z_{\infty} \equiv 1$. 

\item[(ii)] If $\mathcal{S}_{0,\rho}^{c,b}  (\theta, 1)=  \infty$ for some (and then for all) $\theta \in (0,1)$, then $\mathbb{P}_{z} \left(\lim_{t \rightarrow \infty} Z_{t} = \infty \right)= 1$. 
\end{itemize}
\end{theorem}

The result in Theorem \ref{lemma6} improves those in Theorem 2.2 of \cite{La2005} and Theorem 3 of \cite{Pa2017}, which only deal with the subcritical cooperative case. Indeed, both results make strong assumptions on the offspring distribution, the former assumes the log-moment condition \eqref{logmoment} and the latter a first-moment condition. In  both cases, the imposed assumptions imply  $\mathcal{S}_{0,\rho}^{c,b}(\theta, 1)  < \infty$, for some $\theta \in (0,1)$. To see this, recall that $|\mathcal{Q}_{0,\rho}^{c,b}(\theta, 1)|  < \infty$, for some $\theta \in (0,1)$, whenever $\varsigma <0$ which clearly implies that \eqref{logmoment} is satisfied. 

Moreover, one can see from the probability generating function of the stationary distribution in Theorem \ref{lemma6} that $d=0$ and the log-moment condition \eqref{logmoment} are necessary and sufficient conditions for the stationary distribution to admit a first moment in the subcritical cooperative regime. In the critical cooperative regime, the latter assertion does not always hold. For instance, consider the BPI-process with parameters $d=0$, $c=1$, $\pi_{1} = \rho\in (0,1)$, $b_{1} = b =1$ and $\pi_{i}=b_{i} = 0$ for all $i \geq 2$. By Example \ref{Example4} below, assumption \eqref{mainE1} is satisfied, i.e.\ the associated BPI-process is conservative. Moreover  the stationary distribution is such that 
\begin{align}
\mathbb{P}(Z_{\infty}= k)=\frac{\Gamma(\rho -1+k)}{\Gamma(\rho -1)k!}, \qquad \textrm{for}\quad k \in \mathbb{N}, 
\end{align}
\noindent and whose probability generating function is given by $\mathbb{E}[u^{Z_{\infty}}]=1-(1-u)^{1-\rho}$, for $u \in [0,1]$. In other words,  $Z_\infty$ is distributed as a Sibuya distribution with parameter $1-\rho$. From the form of the distribution of $Z_\infty$, it is clear that it does not admit a first moment. This particular BPI-process appears in \cite{PaMi2019} in connection with the   Wright-Fisher diffusion with efficiency which is described as the unique strong solution, taking values on $[0,1]$, of the following SDE
\begin{eqnarray*}
{\rm d} U_t= -\rho U_t(1-U_t){\rm d}t+\sqrt{2U_t(1-U_t)^2}{\rm d}B_t, \quad \mathrm{with } \quad U_0=x,
\end{eqnarray*}
where $B=(B_t, t\ge 0)$ is a standard Brownian motion and $x \in [0,1]$; see also Theorem \ref{Theo2} in Section \ref{sec3} below for a precise statement. 

%%%%%%%%%%%%%%%%%%%%%%%%%%%%%%%%%%%%%%%%%%%%%%%%
%  Coming-down from infinity
%%%%%%%%%%%%%%%%%%%%%%%%%%%%%%%%%%%%%%%%%%%%%%%%%
\subsection{Coming-down from infinity}

Finally, our last main result provides a condition for the (sub)critical cooperative BPI-process to come-down from infinity. We consider the usual definition of (instantaneously) coming-down from infinity, which means that the state $\{\infty\}$ is an entrance boundary for the process $Z$; see e.g., Definition 2.1 in \cite{Ban2016}. We say that a BPI-branching process (instantaneously) comes down from infinity if there exists  $t > 0$ such that $\lim_{a \rightarrow \infty} \lim_{z \rightarrow \infty } \mathbb{P}_{z}(\zeta_{a} < t) > 0$, where $\zeta_{a}  \coloneqq \inf \{t \geq 0: Z_{t} = a\}$ for $a \in \mathbb{N}_{0}$.
\begin{theorem}[Coming-down from infinity] \label{lemma9}
Suppose that $\mathcal{I}_{d,\rho}^{c,b}(\theta; 1)  < \infty$, for some $\theta \in (0,1)$. Then a (sub)critical cooperative BPI-process (instantaneously) comes down from infinity. 
\end{theorem}

It is important to note that in the subcritical cooperative regime (i.e., $\varsigma <0$), we always have that $\mathcal{I}_{d,\rho}^{c,b}(\theta; 1)  < \infty$, for some $\theta \in (0,1)$; see proof of Lemma \ref{lemma15} in the Appendix. 

In Section \ref{sec5}, we will construct, under the assumptions given in Theorem \ref{lemma9}, the law of the (sub)critical cooperative BPI-process starting from infinity, say $\mathbb{P}_{\infty}$. In particular, if $\mathcal{Q}_{d,\rho}^{c,b}(\theta; 1)  \in (-\infty, \infty)$, for some $\theta \in (0,1)$, Theorem \ref{Lamtype} implies that
\begin{eqnarray*}
\mathbb{E}_{\infty}[e^{-q \zeta_{0} }] = \exp \left( -\int_{0}^{\xi} w_{q}(x) {\rm d} x \right), \hspace*{4mm} q >0,
\end{eqnarray*}
\noindent where $w_{q}(\cdot)$ is defined in Lemma \ref{lemma11} part (iv); this follows by using a similar identity as for (14) in \cite{La2005}.\\

The previous results  leave some open questions that we are planning to answer in the near future. A natural and interesting question in this setting is the speed at which the process comes down from infinity; see for instance Aldous \cite{Aldous1999} for the case of the counting block of the Kingman coalescent or Bansaye et al.\ \cite{Ban2016} for general birth-death processes. On the other hand, Kyprianou, et al.\ \cite{Kyprianou2017} studied the excursions from infinity of the counting block of a ``fast'' fragmentation-coalescence process that is a version of the BPI-process with $d=\rho=b=0$, $c>0$ and an extra-birth rate that sends the process to $\{\infty\}$. Their goal was to analyse how the extreme form of fragmentation ``competes'' against the Kingman's coalescent rate (i.e., the competition rate in the BPI-process). Following \cite{Kyprianou2017}, it would be of interest to study the excursions of the BPI-process (in the general setting) from infinity in order to see how the cooperation mechanism acts against the Kingman's coalescent rate.\\

In this work, we follow a different route than in Lambert \cite{La2005} and in Gonz\'alez Casanova et al. \cite{Pa2017} to deduce our results and instead we study directly the semigroup of the BPI-process. This route is closer to the approach used by Foucart \cite{Cle2017} and  Leman and Pardo \cite{Leman},  where the  CSBP with logistic competition with fixed environment and in a Brownian environment are  studied, respectively, but with the difference that competition is now represented as an extra jump downwards, and besides that our process possesses a cooperation mechanism, events that are not considered in \cite{Cle2017}. On the other hand, note that  BPI-processes do not satisfy the branching property and thus most of the traditional arguments for branching processes are not applicable. 

Our arguments are based on a random time-change in Lamperti's fashion. Indeed we show  that any BPI-process is a time-change  of a modified branching process with migration, which is introduced and studied in this work. The motivation for introducing this process comes from  Caballero et al.\ \cite{MAC09} (CSBP case),  Lambert \cite{La2005} (CSBP with logistic competition) and Lambert \cite{Lamb2008} where a particular case of the underlying  process has appeared (without cooperation, i.e.\ $b=0$); note that no formal proof or further properties of such process have been reported to the best of our knowledge.

We deeply analyse the behaviour of this modified branching process with migration in order to understand the behaviour of the associated BPI-process. Then, we deduce Theorems \ref{Theo1} and  \ref{Theo3};  and  Propositions \ref{lemma3} and \ref{Pro5}.  Our second tool to analyse the long term behaviour of BPI-processes is a remarkable duality connection with a generalisation of the Wright-Fisher diffusion. This duality relationship has been established in \cite[Theorem 2]{Pa2017} for the subcritical BPI-process to study similar problems. A more precise definition and construction of this dual will be given later in this work. Next, we show that this duality is valid even in the critical regime and prove some new properties of such generalised Wright-Fisher diffusions which may be of independent interest. We then use these results to show Theorems \ref{lemma4}, \ref{lemma6} and \ref{lemma9}.  

 The remainder of this manuscript  is organised as follows. In Section \ref{Sectionex}, we present some examples that show the applicability of our results for the event of explosion. In Section \ref{sec1}, we introduce the modified branching process with migration and prove that it is the underlying process in the Lamperti-type  representation of BPI-processes. In this section, we also show some useful properties of this modified branching process with migration. Section \ref{sec2} is devoted to the proofs of Theorem \ref{Theo1} and  \ref{Theo3}, and Propositions \ref{lemma3} and \ref{Pro5} which are derived from the aforementioned time-change. Section \ref{sec3} deals with the duality relationship between  BPI-processes and generalisations of the Wright-Fisher diffusion as well as some path properties of the latter which will be useful for the sequel. In Section \ref{sec4}, we prove Theorems \ref{lemma4} and  \ref{lemma6}, and we underline the main ideas of the proofs of Lemma \ref{lemma11} and Theorem \ref{Lamtype}. Finally, we establish Theorem \ref{lemma9} in Section \ref{sec5}. 

%%%%%%%%%%%%%%%%%%%%%%%%%%%%%%%%%%%%%
\section{Examples}\label{Sectionex}
%%%%%%%%%%%%%%%%%%%%%%%%%%%%%%%%%%%%%
Before we present the proofs of our main results, we present some examples that show the applicability of Theorems \ref{Theo1} and \ref{Theo3}. The first example  concerns to subcritical cooperative BPI-processes.

\begin{example}
Subcritical cooperative BPI-processes (including LB-processes) never explodes under the log-moment condition \eqref{logmoment}. We introduce the notation, 
\begin{equation}\label{cebolla0}
\overline{\pi}_{k} := \sum_{i \geq k} \pi_{i} \qquad\quad \textrm{and}\quad \qquad \overline{b}_{k} := \sum_{i \geq k} b_{i}, \qquad \textrm{for}\quad k \in \mathbb{N}.
\end{equation} 
We also note that the branching  and the interaction mechanisms can be rewritten as follows 
 \begin{equation} \label{cebolla}
 \Psi_{d,\rho}(u) = (1-u)\left(d- \sum_{i \geq 1} \overline{\pi}_{i} u^{i}\right) \quad \textrm{ and } \quad\Phi_{c,b}(u) = (1-u)\left(c- \sum_{i \geq 1}  \overline{b}_{i} u^{i}\right),
 \end{equation}
\noindent for $u \in [0,1]$. From (\ref{cebolla}), we see that
\begin{eqnarray} \label{eq02}
-\varsigma(1-u) \leq \Phi_{c,b}(u) \leq c(1-u), \hspace*{5mm} \text{for}  \hspace*{3mm} u \in [0,1].
\end{eqnarray}
Note that \eqref{logmoment} is equivalent to $\sum_{k \geq 1} k^{-1}\overline{\pi}_{k} < \infty$. If $\varsigma < 0$, then \eqref{eq02} implies that $|\mathcal{Q}_{d, \rho}^{c,b}(\theta;1)| < \infty$, for $\theta \in (0,1)$, and by Corollary \ref{corollary4}, subcritical cooperative BPI-processes do not explode whenever \eqref{logmoment} is satisfied. In particular, the stronger condition $\sum_{i \geq 1} i \pi_{i}  < \infty$ made in \cite{Pa2017} is clearly a sufficient condition for nonexplosion of subcritical cooperative BPI-processes. Indeed, Lemma \ref{lemma1Red} allows us to directly deduce the above conclusion.
\end{example}

The next examples deal with the critical cooperative regime and consider cases where $\sum_{i\ge 1}i\pi_i$ may be finite or infinite.
\begin{example} \label{example2new}
Assume that $d=0$, $\pi_{1} = 0 $ and $\pi_i=\frac{1}{i(i-1)}$, for $i=2, 3, \ldots$. In this case, we have that $\rho=1$ and  $\sum_{i\ge 1}i\pi_i=\infty$. Moreover, by  Taylor  expansion of $\log(1-u)$, we deduce  
\begin{align*}
\Psi_{0,1}(u)=-u(1-u)\left(1-\log(1-u)\right),\qquad \textrm{for} \quad u\in[0,1].
\end{align*}
\noindent Suppose that $c=\sum_{i\ge 1}ib_i$ (i.e., the critical cooperative regime) and observe  that $\sum_{i\ge 1}ib_i = \sum_{i \geq 1} \overline{b}_{i}$, where $\overline{b}_{i}$ is defined in (\ref{cebolla0}). Then, it follows from  (\ref{cebolla}) that $\Phi_{c,b}$ can be rewritten as 
\begin{align} \label{manzana1}
\Phi_{c,b}(u)= (1-u) \sum_{i \geq 1} (1-u^{i}) \overline{b}_{i} = (1-u)^2 \sum_{i\ge 0}b_{i}^{\ast} u^i, \qquad \textrm{for} \quad u\in[0,1],
\end{align}
where $b_{k}^{\ast}=\sum_{j>k} \overline{b}_j$, for $k \in \mathbb{N}_{0}$. If we assume that
\begin{align} \label{manzana2}
b_{n} \sim \frac{C\alpha (\alpha -1)}{\Gamma(2-\alpha)} n^{-\alpha-1}, \qquad \textrm{as} \quad n\to \infty,
\end{align}
\noindent for $\alpha\in(1,2)$ and $C>0$ a constant, then, by applying \cite[Lemma 4.4] {Steffanson2015} twice, we deduce 
\begin{align*}
b_{n}^{\ast} \sim Cn^{1-\alpha}/\Gamma(2-\alpha), \qquad \textrm{as}\quad n\to \infty;
\end{align*}
\noindent where $\Gamma(\cdot)$ denotes the so-called gamma function. On the other hand, \cite[Theorem XIII.5.5]{Feller1966} implies that
\begin{align*}
\sum_{i \ge 0} b_{i}^{\ast} u^{i} \sim C(1-u)^{\alpha-2}, \qquad \textrm{as} \quad u\to 1,
\end{align*}
\noindent and thus, from (\ref{manzana1}), we have 
\begin{align*}
\Phi_{c,b}(u)\sim C(1-u)^\alpha, \qquad \textrm{as} \quad u\to 1.
\end{align*}
\noindent Therefore, for all $\varepsilon \in (0,1)$, there exists $\theta \in (0,1)$ such that, for $u \in (\theta, 1]$,
\begin{align}
(1- \varepsilon) C (1-u)^{\alpha} & \leq  \Phi_{c,b}(u) \leq (1+\varepsilon) C (1-u)^{\alpha}  \label{pera2}. 
\end{align}
\noindent In particular, 
\begin{align*}
|\mathcal{Q}_{0, 1}^{c,b}(\theta;1)|\le \frac{1}{(1-\epsilon)C}\int_{\theta}^1 \frac{1-\log(1-u)}{(1-u)^{\alpha-1}}{\rm d} u<\infty,
\end{align*}
\noindent and by  Corollary \ref{corollary4}, the critical cooperative BPI-process with the above parameters is conservative.
\end{example}

\begin{example} \label{example3new}
Assume again that $d=0$ and  $\rho =1$. Furthermore, we assume that 
$\pi_{n} \sim \frac{C^{\prime}(1-\beta)}{\Gamma(\beta)}n^{\beta-2}$, as $n \rightarrow \infty$, for
$\beta\in(0,1)$ and a constant $C^{\prime}>0$. In this case, $\sum_{i\ge 1}i\pi_i=\infty$. On the one hand, it follows from \cite[Lemma 4.4] {Steffanson2015} 
\[
\overline{\pi}_{n} \sim \frac{C^{\prime}}{\Gamma(\beta)}n^{\beta-1},\qquad  \textrm{as}\qquad n \rightarrow \infty,
\]
 where $\overline{\pi}_{n}$ is defined in (\ref{cebolla0}). On the other hand, by \eqref{cebolla} and \cite[Theorem XIII.5.5]{Feller1966}, we have
\begin{align*}
\Psi_{0,1}(u)\sim - C^{\prime}u(1-u)^{1-\beta},\qquad \textrm{as}\quad u\to 1.
\end{align*}
\noindent As in Example \ref{example2new}, we suppose that $c=\sum_{i\ge 1}ib_i$ (i.e., the critical cooperative regime) and that \eqref{manzana2} is satisfied. Thus, we have that  
$\Phi_{c,b}(u)\sim C(1-u)^\alpha$, as $u\to 1$, where $\alpha \in (1,2)$. 

We split our analysis in three different cases depending on the values of $\alpha$ and $\beta$. In order to do so, we observe that for all $\varepsilon \in (0,1)$, there exists $\theta \in (0,1)$ such that, for $u \in (\theta, 1]$,
\begin{align} \label{pera}
 - (1+ \varepsilon) C^{\prime} u (1-u)^{1-\beta} & \leq  \Psi_{0,1}(u)  \leq - (1- \varepsilon) C^{\prime} u (1-u)^{1-\beta},
\end{align}
\noindent and \eqref{pera2} holds. Then,  it follows from \eqref{pera2} and \eqref{pera} that for every $\varepsilon \in (0,1)$ there exists $\theta \in (0,1)$ such that 
\begin{align} \label{zanahoria}
-\frac{1+\varepsilon}{1-\varepsilon}\frac{C^{\prime}}{C}\int_{\theta}^{x} (1-u)^{1-\beta-\alpha}{\rm d} u \leq \mathcal{Q}_{0, 1}^{c,b}(\theta;x) \leq -\frac{1-\varepsilon}{1+\varepsilon}\frac{C^{\prime}}{C}\int_{\theta}^{x} (1-u)^{1-\beta-\alpha}{\rm d} u, \quad \text{for} \quad x \in (\theta, 1). 
\end{align}

{\bf Case 1.} If $2-\alpha>\beta$, it follows from \eqref{zanahoria} that there exists $\theta \in (0,1)$ such that $|\mathcal{Q}_{0, 1}^{c,b}(\theta;1)|< \infty$.

\noindent Then, by  Corollary \ref{corollary4}, we deduce that the associated BPI-process is conservative.

{\bf Case 2.} Suppose that $2-\alpha=\beta$. Then, it follows from \eqref{zanahoria} that $\mathcal{Q}_{0, 1}^{c,b}(\theta;1)=-\infty$, for some $\theta \in (0,1)$. 
On the other hand, for every $\varepsilon \in (0,1)$ there exists $\theta \in (0,1)$ such that, by \eqref{pera2} and \eqref{zanahoria}, 
\begin{align*}
\mathcal{R}_{0, 1}^{c,b}(\theta ;x) &\ge  \frac{1}{(1+\epsilon)C}\int_{\theta}^{x}\frac{1}{u(1-u)^\alpha}\exp\left(-\left(\frac{1+\epsilon}{1-\epsilon}\right)\frac{C^\prime}{C}\log\left(\frac{1-\theta}{1-u}\right)\right){\rm d} u\\
&=\frac{(1-\theta)^{-\left(\frac{1+\epsilon}{1-\epsilon}\right)\frac{C^\prime}{C}}}{(1+\epsilon)C} \int_{\theta}^{x} u^{-1}(1-u)^{\left(\frac{1+\epsilon}{1-\epsilon}\right)\frac{C^\prime}{C}-\alpha} {\rm d} u,
\end{align*}
\noindent for $x \in (\theta, 1)$. Thus, if $1+\left(\frac{1+\epsilon}{1-\epsilon}\right)\frac{C^\prime}{C}\le \alpha$, we have $\mathcal{R}_{0, 1}^{c,b}(\theta;1)=\infty$. In particular, if $2-\alpha=\beta$ and $1+\frac{C^\prime}{C}< \alpha$, from the previous inequalities and  Theorem \ref{Theo3} we deduce that  the associated BPI process is conservative.

Similarly, for every $\varepsilon \in (0,1)$ there exists $\theta \in (0,1)$ such that, by \eqref{pera2} and \eqref{zanahoria}, 
\begin{align*}
\mathcal{R}_{0, 1}^{c,b}(\theta ;x) & \le  \frac{1}{(1-\epsilon)C}\int_{\theta}^{x}\frac{1}{u(1-u)^\alpha}\exp\left(-\left(\frac{1-\epsilon}{1+\epsilon}\right)\frac{C^\prime}{C}\log\left(\frac{1-\theta}{1-u}\right)\right){\rm d} u\\
 & = \frac{(1-\theta)^{-\left(\frac{1-\epsilon}{1+\epsilon}\right)\frac{C^\prime}{C}}}{(1-\epsilon)C} \int_{\theta}^{x} u^{-1}(1-u)^{\left(\frac{1-\epsilon}{1+\epsilon}\right)\frac{C^\prime}{C}-\alpha}{\rm d} u,
\end{align*}
\noindent for $x \in (\theta, 1)$. If $1+\left(\frac{1-\epsilon}{1+\epsilon}\right)\frac{C^\prime}{C}> \alpha$, we then have $\mathcal{R}_{0, 1}^{c,b}(\theta;1)<\infty$.   On the other hand, a similar computation shows that, for every $\varepsilon \in (0,1)$ there exists $\theta \in (0,1)$ such that
\begin{align*}
\mathcal{R}_{0, 1}^{c,b}(u;x) \leq \frac{(1-u)^{-\left(\frac{1-\epsilon}{1+\epsilon}\right)\frac{C^\prime}{C}}}{(1-\epsilon)C}   \frac{1}{1+\left(\frac{1-\epsilon}{1+\epsilon}\right)\frac{C^\prime}{C}-\alpha} u^{-1}\left( (1-u)^{1+\left(\frac{1-\epsilon}{1+\epsilon}\right)\frac{C^\prime}{C}-\alpha } - (1-x)^{1+\left(\frac{1-\epsilon}{1+\epsilon}\right)\frac{C^\prime}{C}-\alpha} \right),
\end{align*}
\noindent for $\theta < u \leq x < 1$, which implies that 
\begin{align*}
\mathcal{E}_{0, 1}^{c,b}(\theta;1)\le \frac{1}{(1-\epsilon)C(1+\left(\frac{1-\epsilon}{1+\epsilon}\right)\frac{C^\prime}{C}-\alpha)} \frac{C_{\alpha,\epsilon}}{\theta}\int_{\theta}^{1} u^{-1} (1-u)^{1-\alpha}{\rm d} u < \infty.
\end{align*}
\noindent In other words, if $2-\alpha=\beta$ and $1+\frac{C^\prime}{C}>\alpha$, we have that $\mathcal{E}_{0, 1}^{c,b}(\theta;1)$ is finite and by Theorem \ref{Theo3} the associated BPI-process explodes almost surely. 

{\bf Case 3.} Finally, we consider the case $2-\alpha<\beta$. In this case, similarly as before, for every $\varepsilon \in (0,1)$ there exists $\theta \in (0,1)$ such that, by \eqref{pera2} and \eqref{zanahoria},
\begin{align*}
\mathcal{R}_{0, 1}^{c,b}(\theta;x) & \leq \frac{1}{(1-\epsilon)C} \int_{\theta}^{x}\frac{1}{u(1-u)^\alpha}\exp\left(-  \frac{1-\varepsilon}{1+\varepsilon}\frac{C^{\prime}}{C(2-\beta-\alpha)}\left((1-\theta)^{2-\beta-\alpha}-(1-u)^{2-\beta-\alpha}\right)\right) {\rm d} u \nonumber \\
& = C(\varepsilon, \theta, \alpha, \beta)\int_{\theta}^{x}\frac{1}{(1-u)^\alpha}\exp\left(-  \frac{1-\varepsilon}{1+\varepsilon}\frac{C^{\prime} (1-u)^{2-\beta-\alpha}}{C(\beta+\alpha-2)} \right) {\rm d} u, \qquad \textrm{for}\quad x \in (\theta,1),
\end{align*}
where
\[
C(\varepsilon, \theta, \alpha, \beta)=\frac{1}{(1-\epsilon)\theta C} \exp\left(-  \frac{1-\varepsilon}{1+\varepsilon}\frac{C^{\prime}}{C(2-\beta-\alpha)}(1-\theta)^{2-\beta-\alpha}\right).
\]
By making the change of variables $y=\frac{1-\varepsilon}{1+\varepsilon}\frac{C^{\prime} (1-u)^{2-\beta-\alpha}}{C(\beta+\alpha-2)}$, we obtain that
\begin{align*}
\mathcal{R}_{0, 1}^{c,b}(\theta;1) & \leq C(\varepsilon, \theta, \alpha, \beta)(\beta+\alpha-2)^{\frac{\beta-1}{2-\beta-\alpha}} \left(\frac{1+\varepsilon}{1-\varepsilon} \frac{C}{C^{\prime}} \right)^{\frac{1-\alpha}{2-\beta-\alpha}} \int_{\frac{1-\varepsilon}{1+\varepsilon}\frac{C^{\prime} (1-\theta)^{2-\beta-\alpha}}{C(\beta+\alpha-2)}}^{\infty} y^{\frac{\alpha-1}{\beta+\alpha-2}-1}e^{-y}{\rm d} y \nonumber \\
 & = C(\varepsilon, \theta, \alpha, \beta)(\beta+\alpha-2)^{\frac{\beta-1}{2-\beta-\alpha}} \left(\frac{1+\varepsilon}{1-\varepsilon} \frac{C}{C^{\prime}} \right)^{\frac{1-\alpha}{2-\beta-\alpha}}  \Gamma\left(\frac{\alpha-1}{\beta+\alpha-2};\frac{1-\varepsilon}{1+\varepsilon}\frac{C^{\prime} (1-\theta)^{2-\beta-\alpha}}{C(\beta+\alpha-2)}\right),
\end{align*}
\noindent where $\Gamma(a;z)=\int_z^\infty u^{a-1}e^{-u}{\rm d}u$, for $a,z \geq 0$, 
is the so-called incomplete Gamma function. In particular, $\mathcal{R}_{0, 1}^{c,b}(\theta;1)$ is finite. 

On the other hand, a similar computation shows that, for every $\varepsilon \in (0,1)$ there exists $\theta \in (0,1)$ such that, by \eqref{pera2} and \eqref{zanahoria}, 
\begin{align*}
\mathcal{R}_{0,1}^{c,b}(u;1) & \geq \frac{e^{\frac{1+\varepsilon}{1-\varepsilon}\frac{C^{\prime} (1-u)^{2-\beta-\alpha}}{C(\beta+\alpha-2)}} }{(1+\epsilon)C} \int_{u}^{1}\frac{1}{w(1-w)^\alpha}\exp\left(-  \frac{1+\varepsilon}{1-\varepsilon}\frac{C^{\prime} (1-w)^{2-\beta-\alpha}}{C(\beta+\alpha-2)} \right) {\rm d} w \nonumber \\
& \geq e^{\frac{1+\varepsilon}{1-\varepsilon}\frac{C^{\prime} (1-u)^{2-\beta-\alpha}}{C(\beta+\alpha-2)}}K(\varepsilon, \alpha, \beta) \Gamma\left(\frac{\alpha-1}{\beta+\alpha-2};\frac{1+\varepsilon}{1-\varepsilon}\frac{C^{\prime} (1-u)^{2-\beta-\alpha}}{C(\beta+\alpha-2)}\right),
\end{align*}
\noindent for $u \in (\theta, 1)$ and where
\[
K(\varepsilon, \alpha, \beta)=\frac{ (\beta+\alpha-2)^{\frac{\beta-1}{2-\beta-\alpha}}}{(1+\epsilon)C} \left(\frac{1-\varepsilon}{1+\varepsilon} \frac{C}{C^{\prime}} \right)^{\frac{1-\alpha}{2-\beta-\alpha}}. 
\]
The previous inequality implies that 
\begin{align*}
\mathcal{E}_{0, 1}^{c,b}(\theta;1)\ge  K(\varepsilon, \alpha, \beta) \int_{\theta}^{1} e^{\frac{1+\varepsilon}{1-\varepsilon}\frac{C^{\prime} (1-u)^{2-\beta-\alpha}}{C(\beta+\alpha-2)}} \Gamma\left(\frac{\alpha-1}{\beta+\alpha-2};\frac{1+\varepsilon}{1-\varepsilon}\frac{C^{\prime} (1-u)^{2-\beta-\alpha}}{C(\beta+\alpha-2)}\right) {\rm d} u.
\end{align*}
\noindent By making the change of variables $y=\frac{1+\varepsilon}{1-\varepsilon}\frac{C^{\prime} (1-u)^{2-\beta-\alpha}}{C(\beta+\alpha-2)}$, we obtain that
\begin{align*}
\mathcal{E}_{0, 1}^{c,b}(\theta;1)\ge K(\varepsilon, \alpha, \beta)  \left(\frac{1-\varepsilon}{1+\varepsilon} \frac{C (\beta+\alpha-2)}{C^{\prime}} \right)^{\frac{\beta+\alpha-1}{2-\beta-\alpha}} \int_{\frac{1+\varepsilon}{1-\varepsilon}\frac{C^{\prime} (1-\theta)^{2-\beta-\alpha}}{C(\beta+\alpha-2)}}^{\infty} y^{-\frac{\beta+\alpha-1}{\beta+\alpha-2}}e^{y}\Gamma\left(\frac{\alpha-1}{\beta+\alpha-2};y\right) {\rm d} y.
\end{align*}
On the other hand, since $\Gamma(a;y)\sim y^{a-1}e^{-y}$, as $y \rightarrow \infty$, for $a >0$ (see for instance \cite{Temme75}), the above inequality implies that $\mathcal{E}_{0, 1}^{c,b}(\theta;1)=\infty$ and we cannot conclude from our results if the associated BPI-process is conservative or explodes with positive probability.
\end{example}

\begin{example}\label{Example4}
Assume that $d=0$ and $\pi_{1} = \rho$. Then, it is clear that $\Psi_{0,1}(u)=-\rho u(1-u)$, for  $u\in[0,1]$. Suppose also that $c=1$ and $b=b_1=1$ (i.e., the associated BPI-process is critical cooperative). Then, $\Phi_{1,1}(u)= (1-u)^2$, for $u\in[0,1]$. Therefore, for $\theta \in (0,1)$,
%\begin{align*}
%\mathcal{Q}_{0, \rho}^{1,1}(\theta;x)=- \rho \int_{\theta}^x \frac{1}{(1-u)}{\rm d} u=-\rho \log\left(\frac{1-\theta}{1-x}\right)
%\end{align*}
%\noindent and
\begin{align*}
\mathcal{R}_{0, \rho}^{1,1}(\theta;1)=\int_\theta^1 \frac{1}{u(1-u)^2}\exp\left(\rho \log\left(\frac{1-u}{1-\theta}\right)\right){\rm d} u=\frac{1}{(1-\theta)^{\rho}}\int_\theta^{1} u^{-1}(1-u)^{\rho-2}{\rm d}u.
\end{align*}
\noindent If $\rho \leq 1$, then $\mathcal{R}_{0, \rho}^{1,1}(\theta;1)=\infty$ and by Theorem \ref{Theo3}, we conclude that the associated BPI process is conservative. On the other hand, if $\rho>1$ we have that $\mathcal{R}_{0, \rho}^{1,1}(\theta;1)$ is finite, for any $\theta\in(0,1)$. Similarly, we observe $\mathcal{E}_{0, \rho}^{1,1}(\theta;1) \geq (\rho-1)^{-1} \int_\theta^1 (1-u)^{-1} {\rm d}u=\infty$, and we cannot conclude from our results. Nonetheless from Theorem 1 in \cite{Pa2017}, we know that the associated BPI-process is conservative.
\end{example}

\begin{example}
Assume that $ \sum_{i \geq 1} i \pi_{i} < \infty$ and $d=0$. As in Example \ref{example2new}, we also suppose that $c=\sum_{i\ge 1}ib_i$ (i.e., the critical cooperative regime) and that \eqref{manzana2} is satisfied. Thus, we have that  
$\Phi_{c,b}(u)\sim C(1-u)^\alpha$, as $u\to 1$, where $\alpha \in (1,2)$. In this case, since $\mathbf{m}:=\sum_{i \geq 1} i \pi_{i}= \sum_{i \geq 1} \overline{\pi}_{i}$ and $\Psi_{0,\rho}(u)\sim -\mathbf{m}(1-u)$, as $u\to 1$ (see the first identity in  \eqref{cebolla}), we have that
\begin{align}
\lim_{u \uparrow 1} \frac{(1-u)^{-1+\alpha}\Psi_{0,\rho}(u)}{u\Phi_{c,b}(u)} = -\frac{\mathbf{m}}{C}.
\end{align}
\noindent In particular, the above implies that $| \mathcal{Q}_{0, \rho}^{c,b}(\theta;1)| < \infty$, for some $\theta \in (0,1)$. Then, by  Corollary \ref{corollary4}, we deduce that the associated BPI-process is conservative.
\end{example}

\begin{example}
Once again, suppose that $c=\sum_{i\ge 1}ib_i$ (i.e., the critical cooperative regime). By \eqref{eq02}, we observe that, for $0<\theta \leq x < 1$,
\begin{align*} 
\mathcal{Q}_{0,\rho}^{c,b}(\theta; x) & = -\int_{\theta}^{x} \frac{d(1-w)}{w \Phi_{c,b}(w)} {\rm d} w  + \mathcal{Q}_{d,\rho}^{c,b}(\theta; x) \geq -\int_{\theta}^{x} \frac{d}{cw} {\rm d} w  + \mathcal{Q}_{d,\rho}^{c,b}(\theta; x)   \geq -\frac{d}{c} \ln(1/\theta) + \mathcal{Q}_{d,\rho}^{c,b}(\theta; x).
\end{align*}
\noindent In particular, it follows that there exists a constant $0 < C_{\theta}<\infty$ such that
\begin{align}  \label{cactus1}
\mathcal{R}_{0,\rho}^{c,b}(\theta; x) \geq C_{\theta} \int_{\theta}^{x} \frac{1}{u \Phi_{c,b}(u)} \exp \Big( \int_{\theta}^{u} \frac{\Psi_{d,\rho}(w)}{w \Phi_{c,b}(w)} {\rm d} w \Big) {\rm d} u.
\end{align}
\noindent Suppose first that $\mathbf{m}_d:=-d + \sum_{i \geq 1} i \pi_{i}  \leq 0$. Then, $\Psi_{d,\rho}(u) \geq 0$, for all $u \in [0,1]$. Moreover, since $c=\sum_{i\ge 1}ib_i$, we also have that $\Phi_{c,b}(u) \geq 0$, for all $u \in [0,1]$. Then,
\begin{align*} 
\mathcal{R}_{0,\rho}^{c,b}(\theta; 1) \geq C_{\theta} \int_{\theta}^{x} \frac{{\rm d} u}{u \Phi_{c,b}(u)}.
\end{align*}
\noindent In particular, the right-hand side of \eqref{eq02} implies that $\mathcal{R}_{0,\rho}^{c,b}(\theta; 1) = \infty$, for some $\theta \in (0,1)$, whenever $\mathbf{m}_d \leq 0$. That is, by Theorem \ref{Theo3}, the associated BPI-process is conservative.

Define $\sigma_{c,b}^{2} = c+ \sum_{i \geq 1} i^{2} b_{i}$ and suppose that 
\begin{align} \label{cactus2}
0 < \mathbf{m}_d < \frac{\sigma_{c,b}^{2}}{2} < \infty,
\end{align}
\noindent Recall that $b_{k}^{\ast}=\sum_{j>k} \overline{b}_j$, for $k \in \mathbb{N}_{0}$. Since $\mathbf{m}_d < \infty$ and $\sum_{i \geq 0}b_{i}^{\ast} = \sigma_{c,b}^{2}/2$, it follows from \eqref{cebolla} and \eqref{manzana1} that
\begin{align*}
\lim_{u \uparrow 1} \frac{(1-u)^{2}}{u\Phi_{c,b}(u)} = \frac{2}{\sigma_{c,b}^{2}} \qquad \text{and} \qquad \lim_{u \uparrow 1} \frac{(1-u)\Psi_{d,\rho}(u)}{u\Phi_{c,b}(u)} = -\frac{2 \mathbf{m}_d}{\sigma_{c,b}^{2}}.
\end{align*}
\noindent In particular, for all $\varepsilon \in (0, \min(2/\sigma_{c,b}^{2}, 1-2\mathbf{m}_d/\sigma_{c,b}^{2}))$, there exists $\theta^{\prime} \in (0,1)$  such that, for $u \in (\theta^{\prime},1)$,
\begin{align*}
\frac{\Psi_{d,\rho}(u)}{u\Phi_{c,b}(u)} \geq \Big(- \varepsilon -\frac{2 \mathbf{m}_d}{\sigma_{c,b}^{2}}\Big) \frac{1}{1-u} \qquad \text{and} \qquad  \frac{1}{u\Phi_{c,b}(u)} \geq \Big(- \varepsilon +\frac{2 }{\sigma_{c,b}^{2}}\Big) \frac{1}{(1-u)^{2}}.
\end{align*}
\noindent Then, by \eqref{cactus1}, there exists a constant $0 < C_{\varepsilon, \theta}<\infty$ such that, for $0 < \theta^{\prime} < \theta \leq x < 1$,
\begin{align*} 
\mathcal{R}_{0,\rho}^{c,b}(\theta; x) \geq C_{\varepsilon, \theta} \int_{\theta}^{x} (1-u)^{-2+\varepsilon+ (2 \mathbf{m}_d)/\sigma_{c,b}^{2}}  {\rm d} u.
\end{align*}
\noindent Therefore, we conclude that $\mathcal{R}_{0,\rho}^{c,b}(\theta; 1) = \infty$, for some $\theta \in (0,1)$, whenever \eqref{cactus2} is satisfied, and by Theorem \ref{Theo3}, the associated BPI-process is conservative.
\end{example}

%%%%%%%%%%%%%%%%%%%%%%%%%%%%%%%%%%%%%%%%%%%%%%%%%%%
\section{A modified branching process with migration and the Lamperti transform} \label{sec1}
%%%%%%%%%%%%%%%%%%%%%%%%%%%%%%%%%%%%%%%%%%%%%%%%%%%%%%%

In this section, we introduce a modified branching process with migration that plays an important role in our analysis. We  shall study its behaviour and establish some useful properties. In particular, we prove that BPI-processes can be constructed from the aforementioned family of processes via a Lamperti-type transformation. Let $X = (X_{t}, t \geq 0)$ be a continuous-time Markov chain whose infinitesimal generator matrix $A= (a_{i,j} )_{i,j \in \mathbb{N}_{0}}$ is given by
\begin{eqnarray*}
a_{i,j}  = \left\{ \begin{array}{lcl}
               \pi_{j-i} + (i-1)b_{j-i}  & \mbox{  if } & i \geq 1 \, \, \text{and} \, \, j >i, \\
                \pi_{j-i} & \mbox{  if } & i = 0 \, \, \text{and} \, \, j >i, \\
              d + c (i-1)  & \mbox{  if } & i \geq 1 \, \, \text{and} \, \, j = i-1, \\
              -(d + \rho + (b+c) (i-1) ) & \mbox{  if } & i \geq 1 \, \, \text{and} \, \, j =i,  \\
               - \rho  & \mbox{  if } & i = 0 \, \, \text{and} \, \, j =i,  \\
              0 & \mbox{  if } & \text{otherwise},  \\
              \end{array}
    \right.
\end{eqnarray*}

\noindent where the parameters $d, c, \rho, b, (\pi_{i})_{i \geq 1}$ and $ (b_{i})_{i \geq 1}$ are defined in (\ref{eq7}). We denote by $\mathbf{P}_{z}$ the law of $X$ when issued from $z \in \mathbb{N}_{0}$. The classification of states of the process $X$ is given in Table \ref{Table2} (recall that we always assume that $c >0$ by (\ref{main})). Let $\sigma_\infty$ be the (first) explosion time of the process $X$, i.e., $\sigma_\infty \coloneqq \inf \{t  \geq 0: X_{t} = \infty \}$, with the usual convention that $\inf \varnothing = \infty$, and observe that $X$ does not explode. Indeed, since we have assumed $\pi = \sum_{i \geq 1} \pi_{i} < \infty$, it follows from (\ref{main}) and \cite[Theorem 2.7.1]{Norris1998} that $\mathbf{P}_{z}(\sigma_{\infty}<\infty)=0$, for $z \in \mathbb{N}_{0}$.

%In this work, we shall consider that the process $X$ is sent to a cemetery point, let us  say $\infty$, after explosion and remains there forever (i.e.\ $X_{t} = \infty$ if $t \geq \sigma_{\infty}$); see e.g., \cite{Norris1998}. Indeed, since we have assumed $\pi = \sum_{i \geq 1} \pi_{i} < \infty$, it follows from (\ref{main}) and \cite[Theorem 2.7.1]{Norris1998} that $\mathbf{P}_{z}(\sigma_{\infty}<\infty)=0$, for $z \in \mathbb{N}_{0}$ (i.e., $X$ does not explode).

\begin{table}[h!] \label{Table2}
\centering
\begin{tabular}{ |c|c|c|c|c|c|} 
 \hline
  $\{0 \}$ & $\{1\}$ & If $\rho >0$ & If $\rho >0$ & If $\rho = 0$ and $b>0$  \\ 
 is absorbing & is absorbing & $X$ is irreducible & $X$ is irreducible   &  $X$ is irreducible \\ 
 if $\rho = 0$  & when   & on $\mathbb{N}_{0}$ when & on $\mathbb{N}$ when &  on $\mathbb{N} \setminus \{1\}$  \\  
  & $\rho=0$  & $d>0$ & $d=0$ & \\
 \hline
\end{tabular}
\caption{Classification of the states of the process $X$.}
\label{Table2}
\end{table}

Informally, the dynamics of the process $X$ are similar to that of a branching process with migration (emigration/immigration). This type of processes has appeared in \cite{Lamb2008} (see pp.\ 97-98) in connection with LB-processes (i.e., BPI-processes with no cooperation). 

We equip the space $\mathbb{N}_{0}$ with the discrete topology. Let $\mathbb{N}_{0,\infty} = \mathbb{N}_{0} \cup \{\infty\}$ be the one-point compactification of $\mathbb{N}_{0}$. Let $B(\mathbb{N}_{0,\infty}, \mathbb{R})$ be the set of bounded functions from $\mathbb{N}_{0,\infty}$ to $\mathbb{R}$. The generator of the Markov process X is the linear operator $(\mathscr{D}(\mathscr{L}^{{\rm mp}}), \mathscr{L}^{{\rm mp}})$, where
\begin{align} \label{eq10}
\mathscr{L}^{{\rm mp}} f(z) & \coloneqq  \sum_{i \geq 1} \pi_{i} \left( f \left(z+ i \right) - f(z) \right) \mathbf{1}_{\{0 \leq z < \infty\}}  +  d \left( f \left(z-1  \right) - f(z) \right) \mathbf{1}_{\{1 \leq z < \infty \}}  \nonumber \\
& \hspace*{10mm} +  c (z-1)  \left( f \left(z-1 \right) - f(z) \right) \mathbf{1}_{\{1 \leq z < \infty \}} + (z-1)\sum_{i \geq 1} b_{i}(f(z+i)-f(z))\mathbf{1}_{\{1 \leq z < \infty\}} \\
& = \sum_{i = -1}^{\infty} a_{z,z+i}(f(z+i) - f(z))\mathbf{1}_{\{0 \leq z < \infty\}}, \nonumber
\end{align}
\noindent for $f \in \mathscr{D}(\mathscr{L}^{{\rm mp}})$ and $ z \in \mathbb{N}_{0, \infty}$ (with the convention $\infty\cdot 0 = 0$) such that
\begin{eqnarray*}
\mathscr{D}(\mathscr{L}^{{\rm mp}}) \coloneqq \left \{ f \in B(\mathbb{N}_{0, \infty}, \mathbb{R}): \sup_{z \in \mathbb{N}_{0}}\sum_{i = -1}^{\infty} |a_{z,z+i}||f(z+i) - f(z)|< \infty \right\};
\end{eqnarray*}
\noindent see for instance Section 4.2 of \cite{Et1986} (or the result in Problem 15, Section 4.11 of \cite{Et1986}).  

We define the probability generating function $G_{z}(u, t)$ of $X$ started at $z \in \mathbb{N}$, as follows
\begin{eqnarray*}
G_{z}(u, t) \coloneqq \mathbf{E}_{z}[u^{ X_{t}}], \qquad \text{for} \qquad u \in [0,1] \quad \text{ and } \quad t \geq 0.
\end{eqnarray*}
Our next goal is to  compute $G_{z}(u, t)$. First, observe that our main assumption (\ref{main}) implies that 
\[ 
\int_{0}^{u} \frac{{\rm d} w}{\Phi_{c,b}(w)} < \infty,\qquad \textrm{for}\quad  u \in [0,1),
\]
 where $\Phi_{c,b}$ is the interaction  mechanism defined in (\ref{eq15}). Thus, we define the mapping $\Upsilon: [0,1) \rightarrow \mathbb{R}_{+}$ by
\begin{eqnarray*}
\Upsilon(u) \coloneqq \int_{0}^{u} \frac{{\rm d} w}{\Phi_{c,b}(w)}, \qquad \text{for} \quad  u \in [0,1),
\end{eqnarray*}

\noindent which is clearly continuous and bijective. Note also that we can continuously extend $\Upsilon$ on $[0,1]$ by taking $\Upsilon(1) \coloneqq  \lim_{u \uparrow 1} \Upsilon(u)= \infty$. We denote such extension by $\Upsilon$ with a slight abuse of notation (observe that now $\Upsilon: [0,1] \rightarrow \mathbb{R}_{+} \cup \{\infty\}$) and we write $\tilde{\Upsilon}$ for its inverse mapping.

For instance,  when $b=0$ (no cooperation), we have that $\Upsilon(u) = - c^{-1} \ln (1-u)$ for  $u \in [0,1]$, and thus, $\tilde{\Upsilon}(w) = 1-e^{-cw}$, for  $w \in\mathbb{R}_{+} \cup \{\infty\}$. Next, we take $c >0$, $b_{1} = b \leq c$ and $b_{i} = 0$ for all $i \geq 2$. In this case,  we have for $u \in [0,1]$ and $w \in \mathbb{R}_{+} \cup \{\infty\}$,  that 
\begin{eqnarray*}
\Upsilon(u)  = \left\{ \begin{array}{lcl}
               \frac{1}{c-b}\ln \left(\frac{c-bu}{c(1-u)} \right)  & \mbox{  if } & b < c, \\
               \frac{1}{c} \frac{u}{1-u} & \mbox{  if } & b = c, \\
              \end{array}
\right.
\, \, \qquad\text{and} \qquad\, \,
\tilde{\Upsilon}(w)  = \left\{ \begin{array}{lcl}
               \frac{ce^{(c-b)w}-c}{ce^{(c-b)w}-b}  & \mbox{  if } & b < c, \\
               \frac{cw}{1+cw} & \mbox{  if } & b = c. \\
              \end{array}
    \right.
\end{eqnarray*}

\noindent As we will show in Section \ref{sec3}, the associated BPI-process of the previous example is in a duality relationship with the so-called Wright-Fisher diffusion with efficiency studied recently in \cite{PaMi2019}. 

\begin{proposition} \label{Pro1}
Suppose that either $d=0$ or $c=d$ and $b=0$. Then 
\begin{eqnarray*} 
G_{z}(u, t)   =  \Big(\tilde{\Upsilon}(t + \Upsilon(u))  \Big)^{z} \exp \left(\int_{u}^{\tilde{\Upsilon}(t + \Upsilon(u)) } \frac{\Psi_{d,\rho}(w) - \Phi_{c,b}(w)}{w\Phi_{c,b}(w)} {\rm d} w \right),  
\end{eqnarray*}
for $ z \in \mathbb{N}$, $ u \in [0,1]$ and $t \geq 0.$
\end{proposition}

\begin{proof}
For $u \in [0,1]$ and $z \in \mathbb{N}$, define the function $H(z,u) \coloneqq u^{z}$ and  note that $H(\cdot, u) \in  \mathscr{D}(\mathscr{L}^{{\rm mp}})$. Hence
\begin{eqnarray*}
\mathscr{L}^{{\rm mp}} H(\cdot, u)(z) = \Big( \Psi_{d,\rho}(u) + \Phi_{c,b}(u)(z-1) \Big) u^{z-1}\mathbf{1}_{\{z  \geq 0 \}} +\frac{1}{u} \Big( \Phi_{c,b}(u)-d(1-u)\Big)\mathbf{1}_{\{z  = 0 \}}.
\end{eqnarray*}

\noindent Furthermore, Dynkin's theorem (Proposition 4.1.7 in \cite{Et1986}) implies that the process 
\[
\left( H(X_{t},u) -  H(z,u) - \int_{0}^{t}  \mathscr{L}^{{\rm mp}} H(\cdot, u)(X_{s})  {\rm d}s, t \geq 0 \right),
\] is a martingale. Then, 
\begin{align*}
 G_{z}(u, t) - G_{z}(u, 0) & - \int_{0}^{t} \mathbb{E}_{z} \left[ \Big( \Psi_{d,\rho}(u)  + \Phi_{c,b}(u)  (X_{s}-1) \Big)u^{X_{s}-1} \right]{\rm d}s \\
& \hspace*{20mm} + \frac{1}{u}\Big(\Phi_{c,b}(u)-d(1-u)\Big) \int_{0}^{t} \mathbf{P}_{z}(X_{s}  = 0)  {\rm d}s = 0.
\end{align*}

\noindent On the one hand, $c=d$ and $b=0$ imply that 
$\Phi_{c,b}(u)-d(1-u) =0$. On the other hand, $d=0$ implies that $\mathbf{P}_{z}(X_{t}  = 0)=0$, for $t \geq 0$. Thus, we deduce that $G_{z}(u, t)$ satisfies the following equation
\begin{eqnarray*}
- \Phi_{c,b}(u) \frac{\partial}{\partial u} G_{z}(u, t)  + 
\frac{\partial}{\partial t} G_{z}(u, t)  = \left( \Psi_{d,\rho}(u) - \Phi_{c,b}(u)\right)u^{-1} G_{z}(u, t), 
\end{eqnarray*}

\noindent for $t \geq 0$, $z \in \mathbb{N}$ and $u \in (0,1]$, with initial condition $G_{z}(u,0) = u^{z}$. Then, note that the function defined in Proposition \ref{Pro1} satisfies the above equation. Indeed, one could also use the method of characteristics to solve the above equation and note that the resulting function can be continuously extended on $[0,1]$. 
\end{proof}

In the supercritical cooperative regime (i.e., when $\varsigma>0$), the integral $ \int_{0}^{u} \frac{{\rm d} w}{\Phi_{c,b}(w)}$, for $u \in [0,1)$, is not always finite. This is because there may exist $u^{\ast} \in (0,1)$ such that $\Phi_{c,b}(u^{\ast}) = 0$, which prevents us from defining the function $\Upsilon$ on the whole interval $[0,1]$. Therefore, it seems complicated to deduce a similar result as in Proposition \ref{Pro1} in the supercritical cooperative regime.

%From Proposition \ref{Pro1}, it is not difficult to deduce the following result which provides a necessary and sufficient condition for the existence of a stationary distribution for the process $X$ in the cases when either $d=0$ or $c=d$ and $b=0$. 

%\begin{corollary} \label{Pro4}
%Suppose that either $d=0$ or $c=d$ and $b=0$. We have the following:
%\begin{itemize}
%\item[(i)] $X$ has a stationary distribution if and only if $\mathcal{Q}_{d,\rho}^{c,b}(\theta; 1)  \in (- \infty, \infty)$ for some $\theta \in (0,1)$. 
%\item[(ii)] If $\mathcal{Q}_{d,\rho}^{c,b}(\theta; 1)  = - \infty$, then for all $z \in \mathbb{N}$ and $a \in \mathbb{R}_{+}$ we have that $\lim_{t \rightarrow \infty} \mathbf{P}_{z}(X_{t}  \leq a) = 0$. 
%\end{itemize}
%\end{corollary}

%\begin{proof}
%By Proposition \ref{Pro1},
%\begin{eqnarray*}
%\lim_{t \rightarrow \infty} G_{z}(u, t) = \exp \left( \int_{u}^{1}  \frac{\Psi_{d, \rho}(w) - \Phi_{c,b}(w)}{w\Phi_{c,b}(w)} {\rm d} w \right), 
%\end{eqnarray*}
%\noindent  for $z \in \mathbb{N}$ and every $u \in [0,1]$. Furthermore, the function 
%\begin{eqnarray*}
%u \mapsto \int_{u}^{1}  \frac{\Psi_{d, \rho}(w) - \Phi_{c,b}(w)}{w\Phi_{c,b}(w)} {\rm d} w
%\end{eqnarray*}
%\noindent is continuous if and only if $\mathcal{Q}_{d,\rho}^{c,b}(\theta; 1) = \int_{\theta}^{1} \frac{\Psi_{d, \rho}(u)}{u \Phi_{c,b}(u)} {\rm d} u \in (- \infty, \infty)$  for some $\theta \in (0,1]$, which clearly implies point (i). On the other hand, if  $\mathcal{Q}_{d,\rho}^{c,b}(\theta; 1)  = - \infty$ then $\lim_{t \rightarrow \infty} G_{z}(u, t) = 0$ which shows (ii). 
%\end{proof}

We explicitly compute the Laplace transform of the hitting time $\sigma_{a} := \inf \{t  \geq 0:  X_{t} \leq a\}$, for $a \in \mathbb{N}_{0}$, in the subcritical cooperative regime. This result will be of special interest in the forthcoming analysis.

\begin{proposition} \label{lemma1}
Assume that $d > 0$ and $\varsigma <0$. For any  $z \geq a\ge 0$ and $\mu > 0$,
\begin{align*}
\mathbf{E}_{z}[e^{-\mu \sigma_{a}}] = \frac{f_{\mu}(z)}{f_{\mu}(a)},
\end{align*}
\noindent where $f_{\mu}(y) \coloneqq \int_{0}^{\infty}e^{-yx}g_{\mu}(x){\rm d} x$, for $y \in \mathbb{N}_{0,\infty}$, and 
\begin{align*}
g_{\mu}(x) \coloneqq \frac{1}{\Phi_{c,b}(e^{-x})} \exp\left( \int_{\theta}^{e^{-x}} \frac{\Psi_{d, \rho}(w)-\mu w}{w \Phi_{c,b}(w)} {\rm d}w \right), \hspace*{5mm} x \geq 0 \hspace*{3mm} \text{and} \hspace*{3mm} \theta \in (0,1).
\end{align*}
\end{proposition}

Before we prove the above result, we require the following lemma that will also play an important role in the proofs of Theorems \ref{Theo1} and \ref{Theo3}. Its proof is deferred to the Appendix to simplify the reading.

\begin{lemma} \label{lemma12}
Suppose that $\varsigma < 0$. Then, for $\theta \in (0,1)$, we have that $\mathcal{Q}_{d, \rho}^{c,b}(\theta;1) \in [-\infty, \infty)$.
\end{lemma}

\begin{proof}[Proof of Proposition \ref{lemma1}]
We first check that $ \int_{0}^{\infty} g_{\mu}(x) {\rm d}x < \infty$. Recall from \eqref{eq02}, that 
\[
-\varsigma(1-w) \leq \Phi_{c,b}(w) \leq c(1-w),\qquad \textrm{for }\quad w \in [0,1], 
\]
provided that $\varsigma < 0$.  In other words, there is a constant $C >0$ such that  
\begin{eqnarray} \label{eq5}
\exp\left(- \int_{\theta}^{u} \frac{\mu}{\Phi_{c,b}(w)} {\rm d}w \right) \leq  \exp\left( -\int_{\theta}^{u} \frac{\mu}{c(1-w)} {\rm d}w \right) \leq  C (1-u)^{\frac{\mu}{c}},
\end{eqnarray}

\noindent for $u \in [\theta, 1)$. Similarly, one can check that there is another constant $C_1 >0$ satisfying
\begin{eqnarray} \label{eq6}
\exp\left(- \int_{\theta}^{u} \frac{\mu}{\Phi_{c,b}(w)} {\rm d}w \right)  \leq  \exp\left( -\int_{\theta}^{u} \frac{\mu}{-\varsigma(1-w)} {\rm d}w \right) \leq  C_1 (1-u)^{-\frac{\mu}{\varsigma}}, 
\end{eqnarray}

\noindent for $u \in (0, \theta]$. Since $\varsigma <0$, Lemma \ref{lemma12} implies that either 
$\mathcal{Q}_{d, \rho}^{c,b}(\theta;1) =  - \infty$ or $\mathcal{Q}_{d, \rho}^{c,b}(\theta;1) \in (-\infty, \infty)$. In both cases, we can find $\theta^{\prime} \in [\theta, 1)$ and a constant $C_2 > 0$ such that
\begin{eqnarray*} 
\int_{0}^{- \ln \theta^{\prime}} g_{\mu}(x) {\rm d}x \leq C_2 \int_{0}^{- \ln \theta^{\prime}} \frac{1}{\Phi_{c,b}(e^{-x})} \exp\left( -\int_{\theta}^{e^{-x}} \frac{\mu}{\Phi_{c,b}(w)} {\rm d}w \right)  {\rm d}x < \infty. 
\end{eqnarray*}
The latter can be obtained by using the estimates in (\ref{eq02}) and (\ref{eq5}) taking into account that $\mu > 0$.

Since $d >0$, there exists $\theta^{\prime \prime} \in (0, \theta]$ such that for all $w  \in [0, \theta^{\prime \prime}]$, $\Psi_{d,\rho}(w) \geq \Psi_{d, \rho}(\theta^{\prime \prime}) > 0$. Then, for some constant $C_3 > 0$ (whose value may change from line to line),
\begin{align*}
\int_{- \ln \theta^{\prime \prime}}^{\infty} g_{\mu}(x) {\rm d}x & \leq  C_3 \int_{- \ln \theta^{\prime \prime}}^{\infty} (1-e^{-x})^{-1} \exp\left(- \int_{e^{-x}}^{\theta^{\prime \prime}} \frac{\Psi_{d, \rho}(w)-\mu w}{w\Phi_{c,b}(w)} {\rm d}w \right) {\rm d}x
\nonumber \\
& \leq  C_3 \int_{- \ln \theta^{\prime \prime}}^{\infty} (1-e^{-x})^{-\frac{\mu}{\varsigma}-1} \exp\left( - \int_{e^{-x}}^{\theta^{\prime \prime}} \frac{\Psi_{d,\rho}(\theta^{\prime \prime})}{cw(1-w)} {\rm d}w \right) {\rm d}x \nonumber \\
& \leq   C_3\int_{- \ln \theta^{\prime \prime}}^{\infty} (1-e^{-x})^{-\frac{\mu}{\varsigma} - \frac{\Psi_{d,\rho}(\theta^{\prime \prime})}{c}-1} e^{-\frac{\Psi_{d,\rho}(\theta^{\prime \prime})}{c}x}  {\rm d}x < \infty,
\end{align*}

\noindent where we used (\ref{eq02}) to obtain the first inequality and (\ref{eq6}) to obtain the second one. The above computations imply our claim, i.e.\ $ \int_{0}^{\infty} g_{\mu}(x) {\rm d}x < \infty$.

Now we set $e_{x}(z) \coloneqq e^{-zx}$   for $z \in  \mathbb{N}_{0,\infty}$ and $x \geq 0$. It is not difficult to see that $e_{x}(\cdot), f_{\mu}(\cdot) \in \mathscr{D}(\mathscr{L}^{{\rm mp}})$. Observe from (\ref{eq10}) that 
\begin{align*}
\mathscr{L}^{{\rm mp}} e_{x}(z) = ( \Psi_{d, \rho}(e^{-x}) + \Phi_{c,b}(e^{-x}) (z-1) ) e^{-x(z-1)},\quad \textrm{for $z \in \mathbb{N}$}.
\end{align*}
We now verify that $f_{\mu}$ is an eigenfunction of the generator $\mathscr{L}^{{\rm mp}}$. For $z \in \mathbb{N}$, 
\begin{align*}
\mathscr{L}^{{\rm mp}} f_{\mu}(z) - \mu f_{\mu}(z) & =  \int_{0}^{\infty} g_{\mu}(x) \left( \mathscr{L}^{{\rm mp}} e_{x}(z) - \mu e_{x}(z) \right) {\rm d}x \\
& =  \int_{0}^{\infty} \Big( e^{x} \Psi_{d,\rho}(e^{-x}) - \mu - e^{x}\Phi_{c,b}(e^{-x}) \Big) g_{\mu}(x)e^{-zx}  {\rm d}x  + \int_{0}^{\infty} e^{x}\Phi_{c,b}(e^{-x}) z e^{-x z} g_{\mu}(x)  {\rm d}x \\
& =  \int_{0}^{\infty} \Big( e^{x} \Psi_{d,\rho}(e^{-x}) - \mu - e^{x}\Phi_{c,b}(e^{-x}) \Big) g_{\mu}(x)e^{-zx}  {\rm d}x - e^{-x(z-1)}\Phi_{c,b}(e^{-x}) g_{\mu}(x) \big |_{x=0}^{x=\infty}\\
& \hspace*{5mm} + \int_{0}^{\infty} \Big( e^{x} \Phi_{c,b}(e^{-x})g_{\mu}^{\prime}(x) +(e^{x} \Phi_{c,b}(e^{-x})- \Phi_{c,b}^{\prime}(e^{-x})) g_{\mu}(x) \Big) e^{-zx}  {\rm d}x \\
& =  \int_{0}^{\infty} \Big(\left( e^{x} \Psi_{d, \rho}(e^{-x}) - \mu - \Phi^{\prime}_{c,b}(e^{-x}) \right) g_{\mu}(x) + e^{x}\Phi_{c,b}(e^{-x}) g_{\mu}^{\prime}(x) \Big) e^{-zx}  {\rm d}x = 0,
\end{align*} 
where in the third identity we have used integration by parts and in the last identity we have used that $g_{\mu}$ solves the following differential equation
\begin{align*}
\left( e^{x} \Psi_{d,\rho}(e^{-x}) - \mu - \Phi^{\prime}_{c,b}(e^{-x}) \right) g_{\mu}(x) + e^{x} \Phi_{c,b}(e^{-x})g_{\mu}^{\prime}(x) = 0, \hspace*{5mm} \text{for all} \hspace*{4mm} x \geq 0.
\end{align*}
\noindent An application of the integration by parts  formula for right-continuous local martingales implies that the process $(e^{-\mu t} f_{\mu}(X_{t}), t\geq 0)$ is a local martingale, see for e.g., \cite[Proposition 2.3.2]{Et1986}. Since the function $f_{\mu}$ is decreasing, one has for any $t \leq \sigma_{a}$, $X_{t} \geq a$ and $f_{\mu}(X_{t}) \leq f_{\mu}(a)< \infty,$ $\mathbf{P}_{z}$-a.s., for all $z \geq a$. Therefore, $(e^{-\mu (t \wedge \sigma_{a})} f_{\mu}(X_{t \wedge \sigma_{a}}), t\geq 0)$ is a bounded martingale and our claim follows by the optional stopping theorem.
\end{proof}

Next, we provide a criterion for the process $X$ to be recurrent or transient whenever $c=d$ and $b=0$,  $d=0$ or $d>0$ and $\rho=0$. In general, we could analyse the  classification of states for $X$ via a coupling argument as is illustrated in the proof of Proposition \ref{Pro2} (iv) below. However, we decided not to provide a general state classification to keep the document at a reasonable size and since the cases here considered are enough for our purposes. 

Recall that for $z \in \mathbb{N}_{0}$, the state $z$ is recurrent if $\mathbf{P}_{z} \left( \{  t \geq 0: X_{t} = z \} \, \, \text{is unbounded} \right) = 1$, and that $z$ is transient if $\mathbf{P}_{z} \left( \{  t \geq 0: X_{t} = z \} \, \, \text{is unbounded} \right) = 0$; see \cite[Section 3.4, page 114]{Norris1998}. We also recall  that recurrence and transience are class properties (see for e.g., \cite[Theorem 3.4.1]{Norris1998}), i.e., if a state $ z_{1} $ is recurrent (or transient) and communicates with another state, say $ z_{2} $, then $ z_{2} $ is also recurrent (or transient). 

\begin{proposition} \label{Pro2}
We have the following classification of states for $X$.
\begin{itemize}
\item[(i)] If $d=0$ and $\rho >0$, then the process $X$ is recurrent or transient on $\mathbb{N}$ according to whether $\mathcal{R}^{c,b}_{0, \rho}(\theta;1) = \infty$ or  $\mathcal{R}^{c,b}_{0, \rho}(\theta;1)  <\infty$, for some $\theta \in (0,1)$. 

\item[(ii)] If $d=c$, $\rho >0$ and $b=0$, then the process $X$ is recurrent or transient on $\mathbb{N}_{0}$ according to whether $\mathcal{R}^{c,0}_{c, \rho}(\theta;1) = \infty$ or  $\mathcal{R}^{c,0}_{c, \rho}(\theta;1)  <\infty$, for some $\theta \in (0,1)$. 

\item[(iii)] If $d= \rho=0$, then the process $X$ is transient on $\mathbb{N} \setminus \{1\}$.

\item[(iv)] If $c \geq d >0$ and $\rho =0$, then the process $X$ is transient on $\mathbb{N}$. 
\end{itemize}
\end{proposition}

\begin{proof}   
We start with the proof of (i). Let us assume that $d=0$ and $\rho >0$. Proposition \ref{Pro1} implies that
\begin{eqnarray*} 
G_{1}(u, t) = u \exp \left( \int_{u}^{\tilde{\Upsilon}(t + \Upsilon(u)) } \frac{\Psi_{0,\rho}(w)}{w \Phi_{c,b}(w)} {\rm d} w \right), \hspace*{3mm} \text{for} \hspace*{2mm}  t \geq 0 \hspace*{2mm}  \text{and} \hspace*{2mm}  u \in [0,1]. 
\end{eqnarray*}

\noindent We deduce from the folklore of  probability generating functions that
\begin{eqnarray*}
\mathbf{P}_{1}\left( X_{t} = 1 \right) = \lim_{u \downarrow 0} \frac{\partial G_{1}(u,t)}{\partial u} =  \exp \left( \int_{0}^{\tilde{\Upsilon}(t) } \frac{\Psi_{0,\rho}(w)}{w \Phi_{c,b}(w)} {\rm d} w \right). 
\end{eqnarray*}

\noindent Hence 
\begin{eqnarray} \label{eq20}
\int_{0}^{\infty} \mathbf{P}_{1}\left( X_{t} = 1 \right) {\rm d}t  & = & \int_{0}^{\infty} \exp \left( \int_{0}^{\tilde{\Upsilon}(t) } \frac{\Psi_{0, \rho}(w)}{w \Phi_{c,b}(w)} {\rm d} w \right) {\rm d}t \nonumber \\
 & = &\int_{0}^{1} \frac{1}{\Phi_{c,b}(x)} \exp \left( \int_{0}^{x} \frac{\Psi_{0, \rho}(w)}{w \Phi_{c,b}(w)} {\rm d} w \right) {\rm d}x.
\end{eqnarray}

\noindent By \cite[Theorem 3.4.2]{Norris1998}, we known that the state $1$ is recurrent or transient whenever the previous integral is infinite or finite. Moreover that  the state $1$ is recurrent or transient will imply that $X$ is recurrent or transient in $\mathbb{N}$ because $X$ is irreducible on $\mathbb{N}$ when $\rho > 0$ and $c >0$; see for e.g., \cite[Theorem 3.4.1]{Norris1998}. Since the  integrand of  (\ref{eq20}) is integrable for values of $x$ close to $0$, (i) follows from our integral conditions. 

Next, we prove (ii). Suppose that  $d =c$, $\rho >0$ and $b=0$. Proposition \ref{Pro1} implies that
\begin{eqnarray*} 
G_{0}(u, t) = \frac{u}{\tilde{\Upsilon}(t + \Upsilon(u)) } \exp \left( \int_{u}^{\tilde{\Upsilon}(t + \Upsilon(u)) } \frac{\Psi_{c,\rho}(w)}{w \Phi_{c,0}(w)} {\rm d} w \right) =\exp \left( \int_{u}^{\tilde{\Upsilon}(t + \Upsilon(u)) } \frac{\Psi_{0,\rho}(w)}{w \Phi_{c,0}(w)} {\rm d} w \right),
\end{eqnarray*}

\noindent for $t \geq 0$ and $u \in [0,1]$. Then,
\begin{eqnarray*}
\mathbf{P}_{0}\left( X_{t} = 0 \right) = \lim_{u \downarrow 0} G_{0}(u,t) =  \exp \left( \int_{0}^{\tilde{\Upsilon}(t) } \frac{\Psi_{0,\rho}(w)}{w \Phi_{c,0}(w)} {\rm d} w \right). 
\end{eqnarray*}

\noindent Hence 
\begin{eqnarray} \label{eq22}
\int_{0}^{\infty} \mathbf{P}_{0}\left( X_{t} = 0 \right) {\rm d}t  & = & \int_{0}^{\infty} \exp \left( \int_{0}^{\tilde{\Upsilon}(t) } \frac{\Psi_{0, \rho}(w)}{w \Phi_{c,0}(w)} {\rm d} w \right) {\rm d}t \nonumber \\
 & = &\int_{0}^{1} \frac{1}{\Phi_{c,0}(x)} \exp \left( \int_{0}^{x} \frac{\Psi_{0, \rho}(w)}{w \Phi_{c,0}(w)} {\rm d} w \right) {\rm d}x.
\end{eqnarray}

\noindent By \cite[Theorem 3.4.2]{Norris1998} again, we known that the state $0$ is recurrent or transient whenever the previous integral is infinite or finite. This would imply that $X$ is recurrent in $\mathbb{N}_{0}$ because $X$ is irreducible on $\mathbb{N}_{0}$ when $d = c >0$ and $\rho >0$. Since the integrand of (\ref{eq22}) is integrable for values of $x$ close to $0$, it is enough to show that 
\begin{eqnarray*} 
\int_{\theta}^{1} \frac{1}{\Phi_{c,0}(x)} \exp \left( \int_{\theta}^{x} \frac{\Psi_{0, \rho}(w)}{w \Phi_{c,0}(w)} {\rm d} w \right) {\rm d}x & = & \int_{\theta}^{1} \frac{1}{\Phi_{c,0}(x)} \exp \left( \int_{\theta}^{x} \frac{\Psi_{c, \rho}(w) - \Psi_{c, 0}(w)}{w \Phi_{c,0}(w)} {\rm d} w \right) {\rm d}x \\
& = & \int_{\theta}^{1} \frac{\theta}{ x\Phi_{c,0}(x)} \exp \left( \int_{\theta}^{x} \frac{\Psi_{c, \rho}(w)}{w \Phi_{c,0}(w)} {\rm d} w \right) {\rm d}x
\end{eqnarray*}

\noindent is infinite or finite whenever $\mathcal{R}^{c,0}_{c, \rho}(\theta;1) = \infty$ or $\mathcal{R}^{c,0}_{c, \rho}(\theta;1) < \infty$, for some $\theta \in (0,1)$. This concludes the proof of (ii). 

Now, we continue with the proof of (iii). Let us  assume that  $d = \rho = 0$. In this case, $X$ is irreducible on $\mathbb{N} \setminus \{1\}$. Proposition \ref{Pro1} implies that $G_{2}(u,t) = u \tilde{\Upsilon}(t + \Upsilon(u))$, for $t \geq 0$ and $u \in [0,1]$. We deduce from properties of the probability generating functions that 
\[
\mathbf{P}_{2}\left( X_{t} = 2 \right) = \frac{1}{2} \lim_{u \downarrow 0} \frac{\partial^{2}G_{2}(u,t)}{\partial u^{2}} =  c^{-1} \Phi_{c,b}(\tilde{\Upsilon}(t)).
\] Hence it is not difficult to see that $ \int_{0}^{\infty} \mathbf{P}_{2}\left( X_{t} = 2 \right) {\rm d}t < \infty$ by our main assumption (\ref{main}). Therefore, $X$ is transient on $\mathbb{N} \setminus \{1\}$ by \cite[Theorem 3.4.1 and Theorem 3.4.2]{Norris1998}. This completes the proof of (iii). 

Finally, we prove (iv). Suppose that $d = c$ and $\rho = 0$. Observe,  in the particular case when $d > 0$ and $\rho = b= 0$, that the process $X$ only jumps downwards and it should be plain that $X$ is transient on $\mathbb{N}$ since $d=c >0$. We now consider the general case $c \geq d > 0$ and $\rho = 0$. Let $X^{0,0}=(X^{0, 0}_{t}, t \geq 0)$ be a modified branching  process with migration with mechanisms $\Psi_{0,0}$ and $\Phi_{c,b}$. We also let $X^{c, 0}=(X^{c,0}_{t}, t \geq 0)$ be a modified branching  process with migration with mechanisms $\Psi_{c,0}$ and $\Phi_{c,0}$. A classical comparison argument for continuous-time Markov processes on $\mathbb{N}_{0, \infty}$ (see for instance \cite[Chapters IV and V]{Lindvall1992}) shows that $X$ can be coupled with $X^{0, 0}$ and $X^{c,0}$ with $X^{0, 0}_{0} = X^{c, 0}_{0}= X_{0} = z \in \mathbb{N}$ and such that $X^{c, 0}_{t} \leq X_{t} \leq  X^{0, 0}_{t}$ for all $t \geq 0$. Therefore, (iv) follows from the coupling together with (iii) and the case $d = c$ and $\rho = b= 0$ proved before.
\end{proof}

%%%%%%%%%%%%%%%%%%%%%%%%%%%%%%%%%%%%%%%%%%%%%%%%%%%%%%%%

Finally, we show that BPI-processes can be constructed from modified branching processes with migration via a Lamperti type transformation. Recall that $Q= (q_{i,j})_{i,j \in \mathbb{N}_{0}}$ defined in (\ref{BPIQmatrix}) denotes the infinitesimal generator matrix of the BPI-process $Z$. Then, the generator of the BPI-process is the linear operator $(\mathscr{D}(\mathscr{L}^{{\rm BPI}}), \mathscr{L}^{{\rm BPI}})$, where 
\begin{align} \label{eq8}
\mathscr{L}^{{\rm BPI}} f(z) & \coloneqq  z\sum_{i \geq 1} \pi_{i} \left( f \left(z+ i \right) - f(z) \right) \mathds{1}_{\{1 \leq z < \infty \}} + d z\left( f \left(z-1  \right) - f(z) \right) \mathds{1}_{\{1 \leq z < \infty \}}  \nonumber \\
& \hspace*{5mm} + cz (z-1)  \left( f \left(z-1 \right) - f(z) \right) \mathds{1}_{\{1 \leq z < \infty\}}  + z(z-1)\sum_{i \geq 1} b_{i}(f(z+i)-f(z)) \mathds{1}_{\{1 \leq z < \infty \}} \\
& = \sum_{i = -1}^{\infty} q_{z,z+i}(f(z+i) - f(z))\mathbf{1}_{\{0 \leq z < \infty\}}, \nonumber
\end{align}
\noindent for $f \in \mathscr{D}(\mathscr{L}^{{\rm BPI}})$ and $z \in \mathbb{N}_{0, \infty}$ (with the convention $\infty \cdot 0 = 0$) such that 
\begin{eqnarray*}
\mathscr{D}(\mathscr{L}^{{\rm BPI}})  \coloneqq \left \{ f \in B(\mathbb{N}_{0, \infty}, \mathbb{R}):  \sup_{z \in \mathbb{N}_{0}} \sum_{i = -1}^{\infty} |q_{z,z+i}||f(z+i) - f(z)|<\infty  \right\}.
\end{eqnarray*}

Define $\sigma_{0}:= \inf\{t \geq 0: X_{t} =0 \}$ and $\sigma_{1}:= \inf\{t \geq 0: X_{t} =1 \}$; and, for $t\ge 0$, the random clock
\begin{eqnarray*}
\theta_{t} \coloneqq \left\{ \begin{array}{lcl}
             \displaystyle \int_{0}^{t \wedge \sigma_{1}} \frac{{\rm d} s}{X_{s}} & \mbox{  if } & d = \rho = 0, \\[5mm]
              \displaystyle\int_{0}^{t \wedge \sigma_{0}} \frac{{\rm d} s}{X_{s}}  & \mbox{  if } & \text{otherwise}.  \\ 
              \end{array}
    \right.
\end{eqnarray*}
Let us introduce, its right-continuous inverse as follows $t \mapsto C_{t} \coloneqq \inf \{u \geq 0: \theta_{u} > t \} \in [0,\infty]$, with the usual convention $\inf \varnothing = \infty$. The Lamperti time-change of the process $X$ is the process $Z^{{\rm min}} = (Z^{{\rm min}}_{t}, t \geq 0)$ defined by
\begin{eqnarray*}
Z^{{\rm min}}_t  = \left\{ \begin{array}{lcl}
               X_{C_{t}} & \mbox{  if } & 0 \leq t < \theta_{\infty}, \\
              1  & \mbox{  if } & t \geq \theta_{\infty} \, \, \text{and} \, \, \sigma_{1} < \infty, \\
              \infty & \mbox{  if } & t \geq \theta_{\infty} \, \, \text{and} \, \, \sigma_{1} = \infty,  \\ 
              \end{array}
    \right.
\end{eqnarray*}

\noindent if $d = \rho = 0$, and otherwise, 
\begin{eqnarray*}
Z^{{\rm min}}_t  = \left\{ \begin{array}{lcl}
               X_{C_{t}} & \mbox{  if } & 0 \leq t < \theta_{\infty}, \\
              0  & \mbox{  if } & t \geq \theta_{\infty} \, \, \text{and} \, \, \sigma_{0} < \infty, \\
              \infty & \mbox{  if } & t \geq \theta_{\infty} \, \, \text{and} \, \, \sigma_{0} = \infty.  \\ 
              \end{array}
    \right.
\end{eqnarray*}

\noindent If $d>0$ or $\rho >0$, we observe that  $Z^{{\rm min}}$  hits its boundaries $\{0 \}$ or $\{\infty\}$ if and only if $\theta_{\infty} < \infty$. Otherwise, if $d=\rho =0$, we observe that  $Z^{{\rm min}}$ hits  $\{1 \}$ or $\{\infty\}$ if and only if $\theta_{\infty} < \infty$.  In particular, $\{1\}$ is an absorbing state whenever $d = \rho =0$. Let $\zeta^{\rm min}_{0} \coloneqq \inf\{t \geq 0: Z^{{\rm min}}_{t}=0 \}$, $\zeta^{\rm min}_{1} \coloneqq \inf\{t \geq 0: Z^{{\rm min}}_{t}=1\}$ and $\zeta^{\rm min}_{\infty} \coloneqq \inf\{t \geq 0: Z^{{\rm min}}_{t}=\infty \}$. On the one hand, if $d=\rho = 0$ and $\sigma_{1}< \infty$, then $\zeta^{\rm min}_{\infty}  = \infty$ and $\zeta^{\rm min}_{1} = \theta_{\infty} = \int_{0}^{\sigma_{1}} \frac{{\rm d} s}{X_{s}}$. On the other hand, if $\sigma_{1}=\infty$ then $\zeta^{\rm min}_{1} = \infty$ and $\zeta^{\rm min}_{\infty} = \theta_{\infty} = \int_{0}^{\infty} \frac{{\rm d} s}{X_{s}}$. Similarly, if $d > 0$ or $\rho > 0$ and $\sigma_{0}< \infty$, then $\zeta^{\rm min}_{\infty}  = \infty$ and $\zeta^{\rm min}_{0} = \theta_{\infty} = \int_{0}^{\sigma_{0}} \frac{{\rm d} s}{X_{s}}$. If $\sigma_{0}=\infty$ then $\zeta^{\rm min}_{0} = \infty$ and $\zeta^{\rm min}_{\infty} = \theta_{\infty} = \int_{0}^{\infty} \frac{{\rm d} s}{X_{s}}$.

\begin{lemma} \label{lemma2}
For $z \in \mathbb{N}_{0}$, the process $Z^{{\rm min}}$ under $\mathbf{P}_{z}$ is a BPI-process. 
\end{lemma}

\begin{proof}
Observe that
\begin{eqnarray}  \label{eqLamp}
\mathscr{L}^{{\rm BPI}} f(z) =z \mathscr{L}^{{\rm mp}}f(z), \hspace*{5mm} z \in \mathbb{N}_{0, \infty} \hspace*{5mm} \text{and} \hspace*{5mm} f \in B(\mathbb{N}_{0, \infty}, \mathbb{R}),
\end{eqnarray}
\noindent Then, it follows that for $f \in \mathscr{D}(\mathscr{L}^{{\rm BPI}})$, we have that $f \in \mathscr{D}(\mathscr{L}^{{\rm mp}})$. On the one hand, by \cite[Proposition 4.1.7]{Et1986}, the process $(f(X_{t}) - \int_{0}^{t} \mathscr{L}^{{\rm mp}}f(X_{s}) {\rm d}s, t \geq 0 )$ is a martingale, for any $f \in \mathscr{D}(\mathscr{L}^{{\rm BPI}})$. On the other hand, by definition of the time-change, for any $t \in [0, \theta_{\infty})$,  we have the following identities $\int_{0}^{C_{t}} \frac{{\rm d}s}{X_{s}} = t$ and  $C_{t} = \int_{0}^{t} Z_{s}^{{\rm min}} {\rm d}s$. Since we have defined $Z_{t}^{{\rm min}}$ even for $t \geq \theta_{\infty}$, the reasoning of Proposition V.1.5 in \cite{Revuz1999} applies, i.e., a (continuous) time-change of a local martingale remains a local martingale. Hence, for $f \in \mathscr{D}(\mathscr{L}^{{\rm BPI}})$ the process, $M_{t}^{f} = (M_{t}^{f}, t \geq 0)$ given by
\begin{align*}
M_{t}^{f} = f(X_{C_{t}}) - \int_{0}^{C_{t}}\mathscr{L}^{{\rm mp}}(X_{s})  {\rm d}s = f(Z_{t}^{{\rm min}}) - \int_{0}^{t} Z_{s}^{{\rm min}} \mathscr{L}^{{\rm mp}}(Z_{s}^{{\rm min}}) {\rm d}s,
\end{align*}
\noindent for $0 \leq t  < \theta_{\infty}$, and 
\begin{align*}
M_{t}^{f} = f(Z_{t}^{{\rm min}}) - \int_{0}^{C_{\theta_{\infty}-}}\mathscr{L}^{{\rm mp}}(X_{s})  {\rm d}s = f(Z_{t}^{{\rm min}}) - \int_{0}^{\theta_{\infty}} Z_{s}^{{\rm min}} \mathscr{L}^{{\rm mp}}(Z_{s}^{{\rm min}}) {\rm d}s,
\end{align*}
\noindent for $t \geq \theta_{\infty}$, is a local martingale. Since $f \in \mathscr{D}(\mathscr{L}^{{\rm BPI}})$, it follows from (\ref{eqLamp}) that $M_{t}^{f}$ has paths which
are bounded on time-intervals $[0, t]$. Therefore, $M_{t}^{f}$ is a true martingale and we conclude that $Z^{{\rm min}}$ solves the martingale problem associated to the BPI-process.
\end{proof}

%%%%%%%%%%%%%%%%%%%%%%%%%%%%%%%%%%%%%%%%%%%%%%%%%%%%%%%%%%
\section{Explosion: Proofs of Theorem \ref{Theo1}, Theorem \ref{Theo3}, Proposition \ref{Pro5} and Proposition \ref{lemma3}} \label{sec2}
%%%%%%%%%%%%%%%%%%%%%%%%%%%%%%%%%%%%%%%%%%%%%%%%%%%%%%%%

In this section, we consider the representation of  BPI-processes as a  time-change of  modified branching  processes with migration established in the previous section (see Lemma \ref{lemma2}). For simplicity, we write $Z = (Z_{t}, t \geq 0)$ instead of $Z^{{\rm min}} = (Z_{t}^{{\rm min}}, t \geq 0)$, and we work under the law $\mathbb{P}_{z}$ instead of $\mathbf{P}_{z}$, for $z \in \mathbb{N}_{0}$. We start by stating some preliminary results that are  important ingredients in the proofs of Proposition \ref{Pro5} and Theorems \ref{Theo1} and \ref{Theo3}. For the ease of reading their proofs are referred to the Appendix.

\begin{lemma} \label{lemma13}
If $\mathcal{Q}_{d, \rho}^{c,b}(\theta;1) \in (-\infty, \infty)$, for some $\theta \in (0,1)$, then  $\mathcal{R}_{d, \rho}^{c,b}(\theta; 1) = \infty$ and $\mathcal{E}_{d, \rho}^{c,b}(\theta; 1) = \infty$. 
\end{lemma}

\begin{lemma} \label{NEWlemma1}
Assume that $\varsigma < 0$. Then $\mathcal{R}_{d, \rho}^{c,b}(\theta; 1) < \infty$, for some $\theta \in (0,1)$, implies that $\mathcal{E}_{d, \rho}^{c,b}(\theta; 1) < \infty$. 
\end{lemma}

\begin{remark} \label{remark8}
If $\mathcal{R}_{d, \rho}^{c, b}(\theta; 1) = \infty$, for some $\theta\in(0,1)$, then $\mathcal{E}_{d, \rho}^{c, b}(\theta; 1) = \infty$. To see this, note that
\[
\mathcal{E}_{d, \rho}^{c, b}(\theta; x) =  \int_{\theta}^{x} (\mathcal{R}_{d, \rho}^{c, b}(\theta; x)-\mathcal{R}_{d, \rho}^{c, b}(\theta; u)) \mathcal{S}_{d, \rho}^{c, b}(\theta; {\rm d} u),
\] 
\noindent for $x \in [\theta, 1]$, where $\mathcal{S}_{d, \rho}^{c, b}(\theta; {\rm d} u)$ is a strictly positive measure on $(\theta,1)$. 
\end{remark}

\begin{lemma} \label{Pro3}
Suppose that $\varsigma < 0$. Then, $\mathcal{E}_{0, \rho}^{c,b}(\theta; 1)< \infty$ (resp.\ $\mathcal{R}_{0, \rho}^{c,b}(\theta; 1) <\infty$), for some $\theta \in (0,1)$, if and only if 
$\mathcal{E}_{d, \rho}^{c,b}(\theta; 1) < \infty$ (resp.\ $\mathcal{R}_{d, \rho}^{c,b}(\theta; 1) < \infty)$, for some $\theta \in (0,1)$. If in addition $b=0$, then  $\mathcal{E}_{c, \rho}^{c,0}(\theta; 1)< \infty$ (resp.\ $\mathcal{R}_{c, \rho}^{c,0}(\theta; 1)< \infty$), for some $\theta \in (0,1)$, if and only if $\mathcal{E}_{d, \rho}^{c,0}(\theta; 1) < \infty$ (resp.\ $\mathcal{R}_{d, \rho}^{c,0}(\theta; 1) < \infty$), for some $\theta \in (0,1)$.
\end{lemma}

\begin{proposition} \label{Pro6}
For $z \in \mathbb{N}$, we have the following:
\begin{itemize}
\item[(i)] Suppose that $d=0$ and $\rho >0$. If $\mathcal{R}_{0,\rho}^{c,b}(\theta; 1) = \infty$ , for some $\theta \in (0,1)$, then $\mathbb{P}_{z}(\zeta_{\infty} < \infty) = 0$. On the other hand, if $\mathcal{R}_{0,\rho}^{c,b}(\theta; 1) < \infty$ and $\mathcal{E}_{0,\rho}^{c,b}(\theta; 1) < \infty$, for some $\theta \in (0,1)$, then $\mathbb{P}_{z}(\zeta_{\infty} < \infty) =1$.

\item[(ii)] If $d=c$, $\rho >0$ and $b=0$, then $\mathbb{P}_{z}(\zeta_{\infty} < \infty) > 0$ if and only if $\mathcal{R}_{c,\rho}^{c,0}(\theta; 1) < \infty$ for some $\theta \in (0,1)$.

\item[(iii)] If $d=c$ and $\rho = b =0$, then $\mathbb{P}_{z}(\zeta_{\infty} < \infty) = 0$ and $\mathbb{P}_{z}(\zeta_{0} < \infty) =1$.

\item[(iv)] If $d = \rho =0$, then $\mathbb{P}_{z}(\zeta_{\infty} < \infty) = 0$ and $\mathbb{P}_{z}(\zeta_{1} < \infty) =1$.
\end{itemize}
\end{proposition} 

\begin{proof} 
Assume that $\rho >0$. By the time-change $\sigma_{0} < \infty$ implies that $\zeta_{\infty} = \infty$ and $\zeta_{0} = \theta_{\infty} = \int_{0}^{ \sigma_{0}} \frac{{\rm d} s}{X_{s}}$, while $\sigma_{0}= \infty$ implies that $\zeta_{0} = \infty$ and $\zeta_{\infty} = \theta_{\infty} = \int_{0}^{ \infty} \frac{{\rm d} s}{X_{s}}$. In particular, on the event $\{ \sigma_{0} < \infty \}$ the process $Z$ converges towards $0$ a.s.\ and thus never explodes. We then concentrate on the event $\{ \sigma_{0} = \infty \}$. The time-change implies that $Z$ explodes if and only if $\int_{0}^{\infty} \frac{1}{X_{s}} {\rm d}s < \infty$, a.s.\ on $\{ \sigma_{0} = \infty \}$.

We start by proving (i). We first suppose that $\mathcal{R}_{0,\rho}^{c,b}(\theta; 1) = \infty$ and show that $Z$ does not explode. Proposition \ref{Pro2} (i) implies that $X$ is recurrent on $\mathbb{N}$. Consider the successive excursions of $X$ under the level $a >0$. Then, there is an infinite number of such excursions by the recurrence of $X$. Define $h_{0} \coloneqq 0$ and for any $n \geq 1$, $a_{n} \coloneqq \inf \{t > h_{n-1}: X_{t} \leq a\}$, $h_{n} \coloneqq \inf \{t > a_{n}: X_{t} > a\}$. Hence
\begin{eqnarray*}
\int_{0}^{\infty} \frac{{\rm d}s}{X_{s}} \geq \sum_{n \geq 1} \int_{a_{n}}^{h_{n}} \frac{{\rm d}s}{X_{s}} \geq \sum_{n \geq 1} \frac{h_{n}-a_{n}}{a}. 
\end{eqnarray*}

\noindent Since the excursions of $X$ are i.i.d.\ and they have positive lengths, we have that
\begin{eqnarray*}
\mathbb{E}_{z} \left[e^{-\int_{0}^{\infty} \frac{1}{X_{s}} {\rm d}s }, \sigma_{0} = \infty \right] \leq \mathbb{E}_{z} \left[e^{- \frac{1}{a}\sum_{n \geq 1} (h_{n}-a_{n} )}\right] = 0, \hspace*{5mm} z \in \mathbb{N}.
\end{eqnarray*}

\noindent Thus, $\zeta_{\infty} = \theta_{\infty} = \infty$ which implies that $Z_{t} < \infty$ for all $t \geq 0$ and that $Z$ does not explode. 

Now, we suppose that $\mathcal{R}_{0,\rho}^{c,b}(\theta; 1) < \infty$ and $\mathcal{E}_{0,\rho}^{c,b}(\theta; 1) < \infty$. Proposition \ref{Pro2} (i) implies that $X$ is transient on $\mathbb{N}$. Note that the event $\{ \sigma_{0} = \infty \}$ has positive probability; indeed $\mathbb{P}_{z}(\sigma_{0} = \infty) = 1$ since $d=0$. 
Tonelli's theorem and Proposition \ref{Pro1} imply that
\begin{align} \label{additeq1}
\int_{0}^{\infty} \mathbb{E}_{z} \left[ \frac{1}{X_{s}}, \sigma_{0} = \infty \right] {\rm d} s & =   \int_{0}^{1} \int_{0}^{\infty} \mathbb{E}_{z} \left[ u^{ X_{s}-1}  \right] {\rm d}s  {\rm d} u \\
& =   \int_{0}^{1} \int_{0}^{\infty} (\tilde{\Upsilon}(s+ \Upsilon(u)))^{z-1} \exp \left( \int_{u}^{\tilde{\Upsilon}(s+ \Upsilon(u))} \frac{\Psi_{0,\rho}(w)}{w\Phi_{c,b}(w)} {\rm d} w \right) {\rm d}s {\rm d}u\nonumber.
\end{align}
By the change of variable $x = \tilde{\Upsilon}(s+ \Upsilon(u))$, we get that 
\begin{eqnarray} \label{eq17}
\int_{0}^{\infty} \mathbb{E}_{z} \left[ \frac{1}{X_{s}}, \sigma_{0} = \infty \right] {\rm d} s  =  \int_{0}^{1} \int_{u}^{1}  \frac{x^{z}}{\Phi_{c,b}(x)} \exp \left( \int_{u}^{x} \frac{\Psi_{0,\rho}(w)}{w\Phi_{c,b}(w)} {\rm d} w \right) {\rm d}x {\rm d}u.
\end{eqnarray}
Observe that the  first integral at the right-hand side of \eqref{eq17} is always finite for $u$ close to $0$ and since  $\mathcal{E}_{0,\rho}^{c,b}(\theta; 1) < \infty$,  the double integral is finite. Therefore, 
\begin{eqnarray*}
\mathbb{P}_{z}\left( \int_{0}^{\infty}  \frac{1}{X_{s}} {\rm ds} < \infty \big| \sigma_{0} = \infty    \right) = 1, \hspace*{5mm} z \in \mathbb{N},
\end{eqnarray*}
\noindent and $Z$ explodes in finite time almost surely since $\mathbb{P}_{z}\left( \sigma_{0} = \infty    \right) = 1$. 

We now prove (ii). First, suppose that $\mathcal{R}_{c,\rho}^{c,0}(\theta; 1) = \infty$. By Proposition \ref{Pro2} (ii), we know that $X$ is recurrent on $\mathbb{N}_{0}$. Hence, a similar argument as the one used in the proof of part (i) will imply that $Z$ does not explode.

Next, we assume that $\mathcal{R}_{c,\rho}^{c,0}(\theta; 1) < \infty$. By Proposition \ref{Pro2} (ii), we know that $X$ is transient on $\mathbb{N}_{0}$. On the other hand, note that Lemma \ref{NEWlemma1} implies that $\mathcal{E}_{c,\rho}^{c,0}(\theta; 1) < \infty$. Again a similar argument as the one used in the proof of part (i) will imply that $Z$ explodes in finite time. Indeed, we only require to show that 
 the event $\{ \sigma_{0} = \infty \}$ has positive probability and that (\ref{additeq1}) still holds but now as an inequality, i.e. 
 \[
\int_{0}^{\infty} \mathbb{E}_{z} \left[ \frac{1}{X_{s}}, \sigma_{0} = \infty \right] {\rm d} s  \le   \int_{0}^{1} \int_{u}^{1}  \frac{x^{z}}{\Phi_{c,b}(x)} \exp \left( \int_{u}^{x} \frac{\Psi_{0,\rho}(w)}{w\Phi_{c,b}(w)} {\rm d} w \right) {\rm d}x {\rm d}u.
\]
Then the proof proceeds  in the same way. The former  is a consequence of Proposition \ref{lemma1} by taking $a=0$ and letting $\mu \rightarrow 0$,  that is
\begin{eqnarray*}
\mathbb{P}_{z}(\sigma_{0} < \infty) = \lim_{\mu \rightarrow 0}\mathbf{E}_{z}[e^{-\mu \sigma_{0}}] = \frac{\displaystyle\int_{0}^{\infty} \frac{1}{\Phi_{c,0}(e^{-x})} \exp\left( -zx + \int_{\theta}^{e^{-x}} \frac{\Psi_{c, \rho}(w)}{w \Phi_{c,0}(w)} {\rm d}w \right) {\rm d}x }{ \displaystyle\int_{0}^{\infty} \frac{1}{\Phi_{c,0}(e^{-x})} \exp\left(  \int_{\theta}^{e^{-x}} \frac{\Psi_{c, \rho}(w)}{w \Phi_{c,0}(w)} {\rm d}w \right) {\rm d}x} \in (0,1).
\end{eqnarray*}

We now consider the cases when $\rho = 0$, that is, parts (iii) and (iv). We start by showing (iii). Since $d = c$, $\rho = b=0$, the BPI-process only jumps downwards. Therefore, $\mathbb{P}_{z}(\zeta_{\infty} < \infty)=0$ and $\mathbb{P}_{z}(\zeta_{0} < \infty)=1$, for $z \in \mathbb{N}$, which implies (iii).  

Finally, we prove (iv). By the time-change, recall that $d=\rho =0$ implies that $Z$, issued from $z \in \mathbb{N}$, hits  $\{1 \}$ or $\{\infty\}$ if and only if $\theta_{\infty} < \infty$. Recall that $\{1\}$ is an absorbing state in this case. Our goal is to prove that $\sigma_{1} < \infty$, a.s.,\ which would imply that $\zeta_{\infty} = \infty$ and $\mathbb{P}_{z}(\zeta_{1} < \infty)= 1$, i.e., $Z$ does not explode. By Proposition \ref{Pro1}, $G_{z}(u,t) = u (\tilde{\Upsilon}(t + \Upsilon(u)))^{z-1}$, for $t \geq 0$, $u \in [0,1]$ and $z \in \mathbb{N}$. Hence $\mathbf{P}_{z}(\sigma_{1} \leq t) = (\tilde{\Upsilon}(t))^{z-1}$. Therefore, letting $t \rightarrow \infty$ implies $\mathbb{P}_{z}(\zeta_{\infty} < \infty)=0$ and $\mathbb{P}_{z}(\zeta_{1} < \infty)=1$, for $z \in \mathbb{N}$, which proves (iv). 
\end{proof}

To prove Theorem \ref{Theo1}, Theorem \ref{Theo3} and Proposition \ref{Pro5}, we introduce as last ingredient the following coupling or stochastic bounds for  BPI-processes. Let $Z^{0, \rho}=(Z^{0, \rho}_{t}, t \geq 0)$ be a BPI-process with mechanisms $\Psi_{0, \rho}$ and $\Phi_{c,b}$. Let $\tilde{c}=c\lor d$ and $Z^{\tilde{c}, \rho}=(Z^{\tilde{c}, \rho}_{t}, t \geq 0)$ be a BPI-process with mechanisms $\Psi_{\tilde{c}, \rho}$ and $\Phi_{\tilde{c},0}$. 
A classical comparison argument for continuous-time Markov process on $\mathbb{N}_{0,\infty}$ (see for e.g., \cite[Chapters IV and V]{Lindvall1992}) shows that, $Z$ can be coupled with $Z^{0, \rho}$ and $Z^{\tilde{c},\rho}$ with $Z^{0, \rho}_{0} = Z^{\tilde{c}, \rho}_{0}= Z_{0} = z \in \mathbb{N}$ such that $Z^{\tilde{c}, \rho}_{t} \leq Z_{t} \leq  Z^{0, \rho}_{t}$ for all $t \geq 0$. 

% NEWlemma1

\begin{proof}[Proof of Theorem \ref{Theo1}] 
Theorem \ref{Theo1} (i) follows immediately from Proposition \ref{Pro6} (i) and Lemma \ref{NEWlemma1}. Theorem \ref{Theo1} (ii) follows from Proposition \ref{Pro6} (ii) and by using the coupling $Z^{\tilde{c}, \rho}_{t} \leq Z_{t}$ for all $t \geq 0$. Finally, Theorem \ref{Theo1} (iii) follows Proposition \ref{Pro6} (i), Lemma \ref{NEWlemma1}, Lemma \ref{Pro3} and by using the coupling $Z_{t} \leq Z^{0, \rho}_{t}$ for all $t \geq 0$.
\end{proof}

\begin{proof}[Proof of Theorem \ref{Theo3}] 
Theorem \ref{Theo3} (i) is nothing more than Proposition \ref{Pro6} (i). Theorem \ref{Theo3} (ii) follows from Proposition \ref{Pro6} (ii) and by using the coupling $Z_{t} \geq Z_{t}^{\tilde{c}, \rho}$, for all $t \geq 0$. % Theorem \ref{Theo3} (iii) follows from Proposition \ref{Pro6} (ii) since  $Z_{t} \geq Z_{t}^{d, \rho}$, for all $t \geq 0$. 
Finally, Theorem \ref{Theo3} (iii) follows from Proposition \ref{Pro6} (i) since $Z_{t} \leq Z_{t}^{0, \rho}$, for all $t \geq 0$, again by the coupling argument from above.
\end{proof}

\begin{proof}[Proof of Proposition \ref{Pro5}]
Proposition \ref{Pro5} (i) follows from Proposition \ref{Pro6} (iv) since $Z_{t} \leq  Z^{0, 0}_{t}$, for all $t \geq 0$, by the coupling argument from above. In particular, this implies under our assumptions that $\mathbb{P}_{z} ( \zeta_{1} < \infty) = 1$, for $z \in \mathbb{N}$, which allows us to conclude that $\mathbb{P}_{z} ( \zeta_{0} < \infty) = 1$ because $d>0$ and $\rho =0$. Finally, Proposition \ref{Pro5} (ii) is Proposition \ref{Pro6} (iv).
\end{proof}

\begin{proof}[Proof of Proposition \ref{lemma3}]
Proposition \ref{lemma3} (i) follows from Proposition \ref{Pro6} (i) and Lemma \ref{NEWlemma1}. 

We now prove Proposition \ref{lemma3} (ii). If $d=0$, then (again) Proposition \ref{Pro6} (i) and Lemma \ref{NEWlemma1} imply that  $\mathbb{P}_{z}(\zeta_{\infty} < \infty) =1$. Now, suppose that $z \in \mathbb{N}$ and $d >0$. Since $X$ is irreducible in $\mathbb{N}$ (if $\rho >0$) and there is a positive probability for $X$ to jump from the state $\{ 1\}$ to $\{ 0 \}$, we deduce that $\mathbb{P}_{z}(\sigma_{0} = \infty) < 1$. On the other hand, recall from the discussion at the beginning of the proof of Proposition \ref{Pro6}, $\{ \sigma_{0} < \infty \} \subseteq \{ \zeta_{\infty} = \infty \}$. The above implies that $\mathbb{P}(\zeta_{\infty} < \infty) < 1$, which concludes our proof.
\end{proof}

%%%%%%%%%%%%%%%%%%%%%%%%%%%%%%%%%%%%%%%%%%%%%%%%%%%
\section{Duality: Generalisations of Wright-Fisher diffusions} \label{sec3}
%%%%%%%%%%%%%%%%%%%%%%%%%%%%%%%%%%%%%%%%%%%%%%%%%%%

The aim of this section is to show that  (sub)critical cooperative BPI-processes are in a duality relationship with  a class of diffusions that we call as generalised Wright-Fisher diffusions. First, we introduce some notation. Let $C([0,1], \mathbb{R})$ be the set of continuous functions from $[0,1]$ to $\mathbb{R}$. Let $C^{2}((0,1), \mathbb{R})$ be the set of continuous functions from $(0,1)$ to $\mathbb{R}$ which together with their first and second derivatives are continuous. Let $C([0,\infty), [0,1])$ be the set of continuous functions from $[0,\infty)$ to $[0,1]$. Let $\mathbb{C}([0,\infty), [0,1])$ denote the space of continuous paths from $[0,\infty)$ to $[0,1]$ furnished with the uniform topology. We also denote by $\mathbb{D}([0,\infty), [0,1])$ the space of c\`adl\`ag paths from $[0,\infty)$ to $[0,1]$ equipped with the Skorokhod topology. 

%For a measurable function $f: E \rightarrow \mathbb{R}$, we write $\Vert f \Vert_{\infty} = \sup_{x \in E} |f(x)|$, i.e., the supremum norm of $f$. For $E \subseteq \mathbb{R}$, let $C(E, \mathbb{R})$ be the set of continuous functions from $E$ to $\mathbb{R}$. For $n \geq 1$, we let $C^{n}(E, \mathbb{R})$ be the set of functions in $C(E, \mathbb{R})$ which together with their derivatives, up to the $n$-th order, are continuous. We also use $C_{0}(E, \mathbb{R})$ to denote the set of  continuous functions from $E$ to $\mathbb{R}$  which vanish at infinity. 
%For $n \geq 1$, we let $C_{0}^{n}(E, \mathbb{R})$ be the set of functions in $C^n_{0}(E, \mathbb{R})$ which  vanish at infinity.

Consider the linear operator $\mathscr{L}^{{\rm dual}}$ given by
\begin{eqnarray*}
\mathscr{L}^{{\rm dual}}f(u) = \Psi_{d, \rho}(u) f^{\prime}(u) + u \Phi_{c,b}(u) f^{\prime \prime}(u), \hspace*{5mm} u \in (0,1)  \hspace*{5mm} \text{and} \hspace*{5mm} f \in  C([0,1], \mathbb{R}) \cap  C^{2}((0,1), \mathbb{R}).
\end{eqnarray*}

\noindent We shall show that there exists a restriction of the domain 
\begin{align*}
\mathscr{D}( \mathscr{L}^{{\rm dual}}) \coloneqq \Big\{f \in  C([0,1], \mathbb{R}) \cap  C^{2}((0,1), \mathbb{R}):  \mathscr{L}^{{\rm dual}}f \in C([0,1], \mathbb{R})\Big\},
\end{align*}
\noindent such that the operator $\mathscr{L}^{{\rm dual}}$ generates a Feller semi-group on $C([0,1], \mathbb{R})$ that one can associate to a unique (in distribution) Markov process with sample paths in $\mathbb{C}([0,\infty), [0,1])$, and prove that such Markov process is in duality with the (sub)critical cooperative BPI-process. Indeed, there is not a unique semi-group associated with $\mathscr{L}^{{\rm dual}}$ as several boundary conditions are possible. Therefore, one needs to be careful when showing the announced duality. To specify the appropriate restrictions of the domain of the operator $\mathscr{L}^{{\rm dual}}$, we will use Feller's boundary classification for which we refer to Section 8.1 of \cite{Et1986} (see also Chapter 23 of \cite{Kalla2002}). Consider a process taking values on $[\mathfrak{a},\mathfrak{b}]$ with $-\infty \leq \mathfrak{a} < \mathfrak{b} \leq \infty$,
\begin{itemize}
\item[-] the boundary $\mathfrak{a}$ (resp.\ $\mathfrak{b}$) is said to be accessible if with positive probability it will be reached in finite time, i.e., the process can enter into $\mathfrak{a}$ (resp.\ $\mathfrak{b}$). If $\mathfrak{a}$ (resp.\ $\mathfrak{b}$)  is accesible, then either the process cannot get out from $\mathfrak{a}$ (resp.\ $\mathfrak{b}$) and the boundary $\mathfrak{a}$ (resp.\ $\mathfrak{b}$) is said to be an exit or the process can get out from $\mathfrak{a}$ (resp.\ $\mathfrak{b}$) and the boundary $\mathfrak{a}$ (resp.\ $\mathfrak{b}$) is called a regular boundary,

\item[-] the boundary $\mathfrak{a}$ (resp.\ $\mathfrak{b}$) is inaccesible if it cannot be reached  in finite time from the interior of $[\mathfrak{a},\mathfrak{b}]$. If the boundary $\mathfrak{a}$ (resp.\ $\mathfrak{b}$) is inaccesible, then either the process cannot get out from $\mathfrak{a}$ (resp.\ $\mathfrak{b}$), and the boundary $\mathfrak{a}$ (resp.\ $\mathfrak{b}$) is said to be natural or the process can get out from $\mathfrak{a}$ (resp.\ $\mathfrak{b}$) and the boundary $\mathfrak{a}$ (resp.\ $\mathfrak{b}$) is said to be an entrance.
\end{itemize} 

The following two lemmas classify the boundaries $0$ and $1$ of the operator $\mathscr{L}^{{\rm dual}}$. The proofs are rather computational and can be found in  Appendix \ref{Apendice}.
\begin{lemma} \label{lemma14}
The boundary $0$ is classified as follows:
\begin{itemize}
\item[(i)] If $d>c$, then $0$ is an entrance boundary. 
\item[(ii)] If $0<d<c$, then $0$ is a regular reflecting boundary (i.e., regular boundary where the process is reflected, see a proper definition in Appendix \ref{Apendice}). If $d = 0$, then $0$ is an exit boundary. 

\item[(iii)] Suppose that $d =c$. If $\mathcal{S}_{c, \rho}^{c, b}(\theta; 0) = - \infty$, for some $\theta \in (0,1)$, then $0$ is  an entrance boundary. If $|\mathcal{S}_{c, \rho}^{c, b}(\theta; 0)| <  \infty$, for some $\theta \in (0,1)$, then $0$ is a regular reflecting boundary. 
\end{itemize}
\end{lemma}

\begin{lemma}  \label{lemma15}
The boundary $1$ is classified as follows:
\begin{itemize}
\item[(i)] Suppose that $\varsigma <0$. Then $1$ is an exit boundary if $\mathcal{E}_{d,\rho}^{c,b}(\theta;1) =  \infty$, for some $\theta \in (0,1)$, or a regular boundary if $\mathcal{E}_{d,\rho}^{c,b}(\theta;1) <  \infty$.

\item[(ii)] Suppose that $\varsigma =0$. If $\mathcal{I}_{d, \rho}^{c, b}(\theta; 1) < \infty$, for some $\theta \in (0,1)$, then $1$ is an exit or a regular boundary accordingly as $\mathcal{E}_{d,\rho}^{c,b}(\theta;1) =  \infty$ or $\mathcal{E}_{d,\rho}^{c,b}(\theta;1) <  \infty$.  When $\mathcal{I}_{d, \rho}^{c, b}(\theta; 1) = \infty$, for some $\theta \in (0,1)$, then $1$ is a natural boundary or an entrance boundary according as $\mathcal{E}_{d,\rho}^{c,b}(\theta;1) =  \infty$ or  $\mathcal{E}_{d,\rho}^{c,b}(\theta;1) <  \infty$.
\end{itemize}
\end{lemma}

\begin{remark} \label{remark9}
A more precise classification of the boundary $1$ can be given when it is regular, i.e. when $\mathcal{I}_{d,\rho}^{c,b}(\theta;1) <  \infty$ and $\mathcal{E}_{d,\rho}^{c,b}(\theta;1) <  \infty$. However, in what follows we assume that assumption \eqref{mainE1} is always satisfied which guarantee that $\mathcal{E}_{d,\rho}^{c,b}(\theta;1) = \infty$. The above is a consequence of the definition of $\mathcal{R}_{d,\rho}^{c,b}(\theta;1)$ and Remark \ref{remark8}. 

On the other hand, since \eqref{mainE1} implies that $\mathcal{E}_{d,\rho}^{c,b}(\theta;1) = \infty$, Lemma \ref{lemma15} ensures that $1$ is a natural boundary when $\mathcal{I}_{d,\rho}^{c,b}(\theta;1) =  \infty$.
\end{remark}

Now, we establish the duality between (sub)critical cooperative BPI-processes $Z = (Z_{t}, t \geq 0)$ and some generalisations of Wright-Fisher diffusions under assumption \eqref{mainE1}. Recall that $\mathbb{P}_{z}$ denotes the law of $Z$ starting from $z \in \mathbb{N}_{0}$. 
\begin{theorem} \label{Theo2}
Suppose that assumption \eqref{mainE1} is satisfied.  Fix $\theta \in (0,1)$, and for $i =0,1$, define 
\begin{eqnarray*}
\mathscr{D}_{i}  (\mathscr{L}^{{\rm dual}} ) \coloneqq \left\{ \begin{array}{lcl}
             \mathscr{D}(\mathscr{L}^{{\rm dual}} )    & \mbox{  if } & i \, \, \text{is inaccessible}, \\
              \{ f \in  \mathscr{D}(\mathscr{L}^{{\rm dual}} ) :   \lim_{u \rightarrow i} \mathscr{L}^{{\rm dual}} f(u) = 0 \}& \mbox{  if } & i \, \, \text{is an exit}, \\
               \{ f \in  \mathscr{D}(\mathscr{L}^{{\rm dual}} )  :   \lim_{u \rightarrow 0} \mathcal{J}^{c,b}_{d, \rho}(\theta; u) f^{\prime}(u) = 0 \}& \mbox{  if } & i=0 \, \, \text{is regular reflecting}.
              \end{array}
    \right.
\end{eqnarray*}

\noindent There exists a unique (in distribution) Markov process $U = (U_{t}, t \geq 0)$ with sample paths in $\mathbb{C}([0,\infty), [0,1])$ whose transition semi-group is generated by 
\begin{eqnarray*}
\mathscr{A}^{\rm dual} \coloneqq \Big\{ (f,\mathscr{L}^{{\rm dual}}f) : f \in \mathscr{D}_{0}(\mathscr{L}^{{\rm dual}} ) \cap \mathscr{D}_{1}(\mathscr{L}^{{\rm dual}} ) \Big\}.
\end{eqnarray*}
\noindent Moreover, let $\mathbb{Q}_{u}$ denote the law of $U$ issued from $u \in [0,1]$, then
\begin{eqnarray} \label{eq11}
\mathbb{E}_{z}\left[u^{Z_{t}} \right] = \mathbb{E}_{\mathbb{Q}_{u}} \left[U_{t}^{z} \right], \hspace*{4mm} \text{for all} \hspace*{2mm} t \geq 0, \hspace*{2mm} z \in \mathbb{N} \hspace*{2mm} \text{and} \hspace*{2mm} u \in [0,1].
\end{eqnarray}
\end{theorem}

%The dual process $U$ is well-defined as long as $\Phi_{c,b}(u) \geq 0$, for all $u \in [0,1]$, which is satisfied due to our main assumption (\ref{main}). We observe that in the supercritical cooperative regime (i.e.\ $\varsigma >0$), the non-negativity of $\Phi_{c,b}$ cannot be guaranteed and implicitly the duality approach cannot be used.

The identity in  (\ref{eq11}) is known as {\sl moment duality} (see e.g., \cite{Jan2014}) and has appeared between many interesting branching processes with interactions and frequency processes in population genetics. The most simple example corresponds to the case $d=\rho=b =0$ and thus, $U$ is the classical Wright-Fisher diffusion. Furthermore, by choosing $d=b=0$, $c=1$, $\pi_{1} = \rho$ and $\pi_{i} = 0$, for all $i \geq 2$, one sees that $U$ is the  so-called Wright-Fisher diffusion with selection. A formal construction via diffusion approximations of the previous two stochastic processes can be found in Section 10.2 of   \cite{Et1986}. Recently, Gonz\'alez-Casanova et al.\ \cite{PaMi2019} studied a model that they called Wright-Fisher diffusion with efficiency which corresponds to the case $d=0$, $c=1$, $\pi_{1} = \rho$, $b_{1} = b \in [0,1]$ and $\pi_{i}=b_{i} = 0$ for all $i \geq 2$. The previous examples emphasise the relevance of this dual relationship. We refer to the works of Alkemper and Hutzenthaler \cite{Alk2007}, Gonz\'alez-Casanova et al.\ \cite{Pa2017} and references therein for further examples.

\begin{proof}[Proof of Theorem \ref{Theo2}]
From Theorem 8.1.1 in \cite{Et1986}, we deduce that $\mathscr{A}^{\rm dual}$ generates a Feller semi-group on $C([0,\infty), [0,1])$. Then, by Theorem 4.2.7 and Proposition 4.1.6 in \cite{Et1986}, there exists a unique (in distribution) Markov process $U = (U_{t}, t \geq 0 )$ with sample paths in $\mathbb{D}([0,\infty), [0,1])$ whose transition semi-group is generated by $\mathscr{A}^{\rm dual}$. We claim that  $U$ admits a modification with sample paths in $\mathbb{C}([0,\infty), [0,1])$. In order to do so, we will show that $U$ satisfies the conditions of the Kolmogorov's continuity criterion (see \cite[Theorem I.1.8]{Revuz1999}). Consider the functions $g_{x}^{(1)}(u) = (u-x)^{2}$ and $g_{x}^{(2)}(u) = (u-x)^{4}$, for $u,x \in [0,1]$. Note that $g_{x}^{(1)}(\cdot), g_{x}^{(2)}(\cdot) \in  \mathscr{A}^{\rm dual}$, for all $x \in [0,1]$. To see the latter, we  only need to be careful in the case when $0$ is regular reflecting but such case follows from the claim that if  $d>0$ then $\mathcal{J}_{d, \rho}^{c, b}(\theta; 0) = 0$, for some $\theta \in (0,1)$. Indeed, if $d>0$ we 
  note that $\lim_{u \rightarrow 0} \Psi_{d, \rho}(u)/\Phi_{c, b}(u) =d/c \in (0,\infty)$ which implies that  for all $\varepsilon >0$, there exists $\theta \in (0,1)$ such that 
\begin{eqnarray}  \label{eq35}
(d/c-\varepsilon) \ln (\theta / x) \leq \Big|\mathcal{Q}_{d,\rho}^{c,b}(\theta;x)\Big| \leq (d/c+\varepsilon) \ln (\theta / x), \qquad \text{for} \qquad x \in (0, \theta).
\end{eqnarray}
From the latter, we clearly deduce that  $\mathcal{Q}_{d, \rho}^{c, b}(\theta; 0) =- \infty$ and therefore $\mathcal{J}_{d, \rho}^{c, b}(\theta; 0) = 0$, for some $\theta \in (0,1)$ as claimed.

On the other hand, there exists a (finite) constant $C >0$ such that $|\mathscr{L}^{{\rm dual}}g_{x}^{(1)}(u)| \leq C$ and  $|\mathscr{L}^{{\rm dual}}g_{x}^{(2)}(u)| \leq Cg_{x}^{(1)}(u)$, for $u,x \in [0,1]$. Then, it follows from Proposition 1.1.5 in \cite{Et1986} that, for $t \geq 0$ and $x \in [0,1]$,
\begin{align*}
\mathbb{E}_{\mathbb{Q}_{x}}\Big[(U_{t}-x)^{2}\Big] \leq \int_{0}^{t} \mathbb{E}_{\mathbb{Q}_{x}}\Big[|\mathscr{L}^{{\rm dual}}g_{x}^{(1)}(U_{s})|\Big] {\rm d} s \leq C t. 
\end{align*}
\noindent Another application of Proposition 1.1.5 in \cite{Et1986} and the previous inequality show that, for $t \geq 0$ and $x \in [0,1]$,
\begin{align*}
\mathbb{E}_{\mathbb{Q}_{x}}\Big[(U_{t}-x)^{4}\Big] \leq \int_{0}^{t} \mathbb{E}_{\mathbb{Q}_{x}}\Big[|\mathscr{L}^{{\rm dual}}g_{x}^{(2)}(U_{s})|\Big] {\rm d} s \leq C \int_{0}^{t} \mathbb{E}_{\mathbb{Q}_{x}}\Big[(U_{s}-x)^{2}\Big] {\rm d} s \leq C^{2}t^{2}. 
\end{align*}
\noindent Then, the Markov property and the previous inequality imply that, for $0 \leq s < t$ and $u \in [0,1]$,
\begin{align*}
\mathbb{E}_{\mathbb{Q}_{u}}\Big[(U_{t}-U_{s})^{4}\Big] = \mathbb{E}_{\mathbb{Q}_{u}}\left[ \mathbb{E}_{\mathbb{Q}_{U_{s}}}\Big[ (U_{t-s}-U_{0})^{4}\Big] \right] \leq C^{2}(t-s)^{2}.
\end{align*}
\noindent Finally, the above shows that the conditions \cite[Theorem I.1.8]{Revuz1999} are satisfied with $\alpha = 4$ and $\beta =1$. This concludes with the proof of the first part of this Theorem.

We now proceed with the proof of (\ref{eq11}). Consider the function $H(z,u) = u^{z}$, for $u \in [0,1]$ and $z \in \mathbb{N}_{0}$. Note that $H(\cdot, u) \in \mathscr{D}(\mathscr{L}^{{\rm BPI}})$ and that $H(z, \cdot) \in \mathscr{D}_{0}(\mathscr{L}^{{\rm dual}}) \cap \mathscr{D}_{1}(\mathscr{L}^{{\rm dual}})$.   Again to see the latter, we need to be careful in the case when $0$ is regular reflecting but in this case recall that $\mathcal{J}_{d, \rho}^{c, b}(\theta; 0) = 0$, for some $\theta \in (0,1)$, which guarantees that $H(z, \cdot) \in \mathscr{D}_{0}(\mathscr{L}^{{\rm dual}}) \cap \mathscr{D}_{1}(\mathscr{L}^{{\rm dual}})$. Then, it follows from \cite[Proposition 4.1.7]{Et1986} that the processes
\begin{eqnarray*}
\left(u^{Z_{t}} - \int_{0}^{t} \mathscr{L}^{{\rm BPI}}H(\cdot, u)(Z_{s}) {\rm d} s, t \geq 0 \right) \hspace*{2mm} \text{and} \hspace*{2mm} \left(U_{t}^{z} - \int_{0}^{t} \mathscr{L}^{{\rm dual}}H(z, \cdot)(U_{s}) {\rm d} s, t \geq 0 \right) 
\end{eqnarray*}

\noindent are martingales, for $u \in [0,1]$ and $z \in \mathbb{N}_{0}$. Note that
\begin{align} \label{eq13}
\mathscr{L}^{{\rm BPI}}H(\cdot, u)(z) & =  \sum_{i \geq 1} z \pi_{i} \left( u^{z+ i} - u^{z} \right) + d z\left( u^{z-1} - u^{z} \right) \nonumber  \\
& \hspace*{10mm} +
c z(z-1)  \left( u^{z-1} - u^{z} \right) + z(z-1)\sum_{i \geq 1} b_{i}(u^{z+i} - u^{z}) \nonumber \\
& =  \Psi_{d, \rho}(u) z u^{z-1} + u\Phi_{c,b}(u)  z (z-1) u^{z-2} =  \mathscr{L}^{{\rm dual}}H(z, \cdot)(u), 
\end{align}
\noindent for $z \in \mathbb{N}$ and $u \in [0,1]$. We also have that $\mathscr{L}^{{\rm BPI}}H(\cdot, u)(0) = 0 = \mathscr{L}^{{\rm dual}}H(0, \cdot)(u)$, for $u \in [0,1]$. Set $h(z,u) = \Psi_{d, \rho}(u) z u^{z-1} + u\Phi_{c,b}(u)  z (z-1) u^{z-2}$, for $u \in [0,1]$ and $z \in \mathbb{N}_{0}$. We apply Ethier-Kurtz's duality result \cite[Corollary 4.4.15]{Et1986} (with $\alpha = \beta=0$, $\tau= \infty$, $\sigma = \tau_{w}$) to the (sub)critical cooperative BPI-process $Z$ and the stopped process $(U_{t \wedge \tau_{w}},t \geq 0)$ where $\tau_{w} \coloneqq \inf\{t \geq 0: U_{t} = w\}$, for $w\in (0,1)$. To apply \cite[Corollary 4.4.15]{Et1986}, an integrability condition \cite[(4.50)]{Et1986} is required. We will verify the latter below and proceed with the proof of  (\ref{eq11}). Following the framework of \cite[Corollary 4.4.15]{Et1986}, we consider that the processes $Z$ and $U$ are independent and we can verify similarly that,  for $u \in [0, w]$ and $z \in \mathbb{N}$, the following identity is satisfied
\begin{eqnarray*}
\mathbb{E}_{z}[u^{Z_{t}}] - \mathbb{E}_{\mathbb{Q}_{u}}[U_{t \wedge \tau_{w}}^{z}] = \int_{0}^{t} \mathbb{E}_{z,u}[\mathbf{1}_{\{t-s > \tau_{w}\}} h(Z_{s}, U_{(t-s) \wedge \tau_{w}})] {\rm d} s = \mathbb{E}_{z,u} \left[ \int_{0}^{t-t\wedge\tau_{w}}  h(Z_{s}, w) {\rm d} s \right],
\end{eqnarray*}

\noindent where the expectation $\mathbb{E}_{z,u}$ is with respect the product measure $\mathbb{P}_{z}\times\mathbb{Q}_{u}$. Since the process $(w^{Z_{t}}- \int_{0}^{t} h(Z_{s}, w) {\rm d}s, t \geq 0)$ is a martingale and $\tau_{w}$ is independent of $Z$, we deduce that
\begin{eqnarray*}
\mathbb{E}_{\mathbb{Q}_{u}}[U_{t \wedge \tau_{w}}^{z}] - \mathbb{E}_{z}[u^{Z_{t}}] = w^{z} - \mathbb{E}_{z}[w^{Z_{t-t\wedge\tau_{w}}}].
\end{eqnarray*}

\noindent On the one hand, by letting $w \rightarrow 1$, $\tau_{w} \rightarrow \tau_{1}$ a.s. On the other hand, by Theorems \ref{Theo1} and \ref{Theo3}, we know that $Z$ does not explode in finite time under  condition (\ref{mainE1}). Therefore,
\begin{eqnarray*}
\lim_{w \rightarrow 1} \Big(\mathbb{E}_{\mathbb{Q}_{u}}[U_{t \wedge \tau_{w}}^{z}] - \mathbb{E}_{z}[u^{Z_{t}}] \Big)= 1 - \mathbb{P}_{z}(Z_{t-t\wedge\tau_{1}} < \infty) = 0, 
\end{eqnarray*}

\noindent i.e., $\mathbb{E}_{z}[u^{Z_{t}}] = \mathbb{E}_{\mathbb{Q}_{u}}[U_{t \wedge \tau_{1}}^{z}]$, for all $u \in [0,1]$ and $z \in \mathbb{N}$. Finally, (\ref{eq11}) follows from Lemma \ref{lemma15} and Remark \ref{remark9}  since $1$ is either an exit or a natural boundary for $U$, under assumption (\ref{mainE1}). 

It remains to verify the technical condition \cite[(4.50)]{Et1986}. Namely, for any $T > 0$ and $w \in (0,1)$ fixed, we need to show that the random variables
\begin{eqnarray*}
\sup_{s,t \in [0,T]} H(Z_{t}, U_{s \wedge \tau_{w}}) \qquad \text{ and } \qquad \sup_{s,t \in [0,T]} |h(Z_{t}, U_{s \wedge \tau_{w}})|
\end{eqnarray*}

\noindent are integrable with respect to $\mathbb{P}_{z}\times\mathbb{Q}_{u}$. Since $\sup_{s,t \in [0,T]} H(Z_{t}, U_{s \wedge \tau_{w}})$ is bounded by $1$, we only need to focus on $\sup_{s,t \in [0,T]} |h(Z_{t}, U_{s \wedge \tau_{w}})|$. Note that $|\Psi_{d,\rho}(u)| \leq d+ \rho$ and $|\Phi_{c,b}(u)| \leq c+ b$, for $u \in [0,1]$. Recall the following two inequalities: 
\[
zu^{z-1} \leq \frac{z}{(1-z)\ln u}\qquad \textrm{and} \qquad z(z-1)u^{z-1} \leq \frac{z^2}{(z-1)^2(\ln u)^{2}}, \qquad \textrm{for } z \geq 2, u \in [0,1].
\]
 Then, for $z \in \mathbb{N}_{0}$, $u \in [0, w]$ and $w \in (0,1)$, there exists a constant $C_{w} >0$ (depending only on $w$) such that
\begin{align*}
|h(z, u)| \leq |\Psi_{d,\rho}(u)| \mathbf{1}_{\{ z =1\}} + |\Psi_{d, \rho}(u)| z  u^{z-1} \mathbf{1}_{\{ z \geq 2\}} + |\Phi_{c,b}(u)| z (z-1) u^{z-1} \mathbf{1}_{\{ z \geq 2\}} \leq C_{w}.
\end{align*}
\noindent Since for all $s \geq 0$, $U_{s \wedge \tau_{w}} \leq w$,  $\mathbb{Q}_{u}$-a.s., for $u \in [0,w]$, the previous inequality implies that the random variable $\sup_{s,t \in [0,T]} |h(Z_{t}, U_{s \wedge \tau_{w}})|$ is integrable. 
\end{proof}

For the remainder of this section, we always assume that (\ref{mainE1}) is satisfied. We then study the behaviour of the dual process $U$. Recall that  $\tau_{w} = \inf \{t \geq 0: U_{t} = w \}$, for $w \in [0,1]$. 

\begin{lemma}  \label{lemma7}
Fix an arbitrary $\theta \in (0,1)$ and $0 < \mathfrak{a} < \mathfrak{b} < 1$. For $u \in [\mathfrak{a},\mathfrak{b}]$, 
\begin{itemize}
\item[(i)] the stopping time $\tau_{\mathfrak{a}} \wedge \tau_{\mathfrak{b}}$ is $\mathbb{Q}_{u}$-a.s.\ finite, and

\item[(ii)] $\mathbb{Q}_{u}(\tau_{\mathfrak{a}} > \tau_{\mathfrak{b}}) = \frac{\mathcal{S}_{d, \rho}^{c, b}(\theta; u) - \mathcal{S}_{d, \rho}^{c, b}(\theta; \mathfrak{a})}{\mathcal{S}_{d, \rho}^{c, b}(\theta; \mathfrak{b})-\mathcal{S}_{d, \rho}^{c, b}(\theta; \mathfrak{a})}$. 
\end{itemize}
\end{lemma}

The function $\mathcal{S}_{d, \rho}^{c, b}$ is well-defined, that is to say $|\mathcal{S}_{d, \rho}^{c, b}(\theta; u)| < \infty$, for all $u \in (0,1)$. Moreover, the quantity $\mathbb{Q}_{u}(\tau_{\mathfrak{a}} > \tau_{\mathfrak{b}}) $ does not depend upon $\theta$.
 
\begin{proof}[Proof of Lemma  \ref{lemma7}]
For $u \in [0,1]$, define the functions
\begin{align*}
f(u) \coloneqq   - \frac{\Psi_{d, \rho}(u)}{u\Phi_{c,b}(u)} \mathbf{1}_{\{u \in [\mathfrak{a},\mathfrak{b}] \}} - \frac{\Psi_{d, \rho}(\mathfrak{a})}{\mathfrak{a}\Phi_{c,b}(\mathfrak{a})} \mathbf{1}_{\{u \in [0,\mathfrak{a}) \}} - \frac{\Psi_{d, \rho}(\mathfrak{b})}{\mathfrak{b}\Phi_{c,b}(\mathfrak{b})} \mathbf{1}_{\{u \in (\mathfrak{b},1] \}},
\end{align*}
\begin{align*}
h(u)  \coloneqq  \frac{1}{u\Phi_{c,b}(u)} \mathbf{1}_{\{u \in [\mathfrak{a},\mathfrak{b}]\}} + \frac{1}{\mathfrak{a}\Phi_{c,b}(\mathfrak{a})} \mathbf{1}_{\{u \in [0,\mathfrak{a}) \}} + \frac{1}{\mathfrak{b}\Phi_{c,b}(\mathfrak{b})} \mathbf{1}_{\{u \in (\mathfrak{b},1] \}},
\end{align*}
\begin{align*}
S(u) \coloneqq \int_{\theta}^{u} \exp \left( \int_{\theta}^{x} f(w) \text{d}w \right) \text{d} x,
\end{align*}
\noindent and
\begin{align*}
T(u) & \coloneqq \frac{S(u)-S(\mathfrak{a})}{S(\mathfrak{b})-S(\mathfrak{a})}  \int_{u}^{\mathfrak{b}} \Big(S(\mathfrak{b})- S(x)\Big) h(x) \exp \left( - \int_{\theta}^{x} f(w) \text{d}w \right) \text{d}x \\
& \hspace*{5mm} + \frac{S(\mathfrak{b})-S(u)}{S(\mathfrak{b})-S(\mathfrak{a})}  \int_{\mathfrak{a}}^{u} (S(x)- S(\mathfrak{a})) h(x) \exp \left( - \int_{\theta}^{x} f(w) \text{d}w \right) \text{d}x.
\end{align*}

\noindent Note that $S(\cdot) \in \mathscr{D}_{0}(\mathscr{L}^{{\rm dual}} ) \cap \mathscr{D}_{1}(\mathscr{L}^{{\rm dual}} )$ and $T(\cdot) \in \mathscr{D}_{0}(\mathscr{L}^{{\rm dual}} ) \cap \mathscr{D}_{1}(\mathscr{L}^{{\rm dual}} )$. In particular, 
\begin{align}  \label{eq24}
\mathscr{L}^{{\rm dual}} S(u) = 0 \hspace*{4mm} \text{and} \hspace*{4mm} \mathscr{L}^{{\rm dual}} T(u) = -1, \hspace*{4mm} \text{for} \hspace*{2mm} u \in [\mathfrak{a}, \mathfrak{b}]. 
\end{align}

By Proposition 4.1.7 in \cite{Et1986}, we know that the process $( T(U_{t})- \int_{0}^{t}  \mathscr{L}^{{\rm dual}} T(U_{s}) {\rm d} s,  t \geq 0)$ is a martingale. Therefore, the optional stopping theorem and (\ref{eq24}) imply that
\begin{eqnarray*}
\mathbb{E}_{\mathbb{Q}_{u}} \left[ t \wedge \tau_{\mathfrak{a}} \wedge \tau_{\mathfrak{b}} \right] = -\mathbb{E}_{\mathbb{Q}_{u}} \left[ \int_{0}^{t \wedge \tau_{\mathfrak{a}} \wedge \tau_{\mathfrak{b}}} \mathscr{L}^{{\rm dual}} T(U_{s}) {\rm d} s \right] = T(u) - \mathbb{E}_{\mathbb{Q}_{u}} \left[ T(U_{t \wedge \tau_{\mathfrak{a}} \wedge \tau_{\mathfrak{b}}}) \right],
\end{eqnarray*}
for $u \in [\mathfrak{a},\mathfrak{b}]$ and $t \geq 0$. By letting $t \rightarrow \infty$, we see that $\mathbb{E}_{\mathbb{Q}_{u}} \left[ \tau_{\mathfrak{a}} \wedge \tau_{\mathfrak{b}}  \right] \leq T(u) < \infty$, i.e. $\tau_{\mathfrak{a}} \wedge \tau_{\mathfrak{b}} < \infty$ a.s.\ which shows (i). 

Again it follows from Proposition 4.1.7 in \cite{Et1986} that the process $(S(U_{t})-\int_{0}^{t}  \mathscr{L}^{{\rm dual}} S(U_{s}) {\rm d} s,  t \geq 0)$ is a martingale. Then, by the optional stopping theorem and (\ref{eq24}), we conclude that 
\begin{eqnarray*}
S(u) = \mathbb{E}_{\mathbb{Q}_{u}} \left[ S(U_{\tau_{\mathfrak{a}} \wedge \tau_{\mathfrak{b}} \wedge t}) \right] - \mathbb{E}_{\mathbb{Q}_{u}} \left[ \int_{0}^{\tau_{\mathfrak{a}} \wedge \tau_{\mathfrak{b}}\wedge t} \mathscr{L}^{{\rm dual}} S(U_{s}) {\rm d} s \right]  = \mathbb{E}_{\mathbb{Q}_{u}} \left[ S(U_{\tau_{\mathfrak{a}} \wedge \tau_{\mathfrak{b}} \wedge t}) \right],
\end{eqnarray*}
for $u \in [\mathfrak{a},\mathfrak{b}]$ and $t \geq 0$. By letting $t \rightarrow \infty$, the dominated convergence theorem and (i) imply that 
\begin{eqnarray*}
S(u) =  \mathbb{E}_{\mathbb{Q}_{u}} \left[ S(U_{\tau_{\mathfrak{a}} \wedge \tau_{\mathfrak{b}}}) \right] =  \mathbb{E}_{\mathbb{Q}_{u}} \left[ S(U_{\tau_{\mathfrak{a}} \wedge \tau_{\mathfrak{b}}}) \mathbf{1}_{\{ \tau_{\mathfrak{a}} < \tau_{\mathfrak{b}} \}}  + S(U_{\tau_{\mathfrak{a}} \wedge \tau_{\mathfrak{b}}}) \mathbf{1}_{\{ \tau_{\mathfrak{a}} > \tau_{\mathfrak{b}} \}} \right], 
\end{eqnarray*}

\noindent for $u \in [\mathfrak{a},\mathfrak{b}]$. In other words, (ii) follows by noticing that $S(u) = \mathcal{S}_{d,\rho}^{c, b}(\theta,u)$, for $u \in [\mathfrak{a}, \mathfrak{b}]$.
\end{proof}

Next, we verify that the dual process $U$ is regular in $(0,1)$, i.e., there is a positive probability that $U$ reach any point in $(0,1)$ from any starting point in $(0,1)$. 

\begin{corollary} \label{corollary2}
Suppose that $u \in (0,1)$. Then, for any $w \in (0,1)$, $\mathbb{Q}_{u}(\tau_{w} < \infty)>0$.
\end{corollary}

\begin{proof}
 First, we assume  that $0<u < w<1$ and observe that  we can always find $\mathfrak{a} \in (0,u)$ and $\mathfrak{b} \in (w,1)$. Thus $\tau_{w} < \tau_{\mathfrak{b}}$, under $\mathbb{Q}_{u}$,  and Lemma \ref{lemma7} (ii) implies that $\mathbb{Q}_{u}(\tau_{w} < \infty) > \mathbb{Q}_{u}(\tau_{\mathfrak{b}} < \infty) > \mathbb{Q}_{u}(\tau_{\mathfrak{b}} < \tau_{\mathfrak{a}}) >0$.
\end{proof}

The function $\mathcal{S}_{d,\rho}^{c, b}(\theta; u)$, for $u \in (0,1)$, is the so-called scale function of $U$. In particular, it is continuous and strictly increasing; see e.g., \cite[Chapter 14, Section 6, page 226-227]{Taylor1981}, \cite[Proposition VII.3.2]{Revuz1999} or \cite[Theorem 23.7]{Kalla2002}. 

Let us  study the long term behaviour of $U$ under the assumption that $d=0$. The next result generalises  Corollaries 1 and  2 in \cite{Pa2017} and Corollary 1 in \cite{PaMi2019}. Recall  that $|\mathcal{S}_{0, \rho}^{c,b}(\theta; u)| < \infty$, for all $u \in [0,1)$ (see the comments before Theorem \ref{lemma6} for an explanation of this fact). 

\begin{proposition} \label{lemma5}
Suppose  $u \in (0,1)$ and that $d=0$. Then, there exists a $\{0,1\}$-valued random variable $U_{\infty}$ such that $\lim_{t \rightarrow \infty} U_{t} = U_{\infty}$ a.s.  More precisely,
\begin{itemize}
\item[(i)] if $\mathcal{S}_{0,\rho}^{c,b}(\theta; 1)  < \infty$, for some $\theta \in (0,1)$, then  
\begin{align*}
\mathbb{Q}_{u}(U_{\infty} = 1) = 1- \mathbb{Q}_{u}(U_{\infty} = 0) = \frac{\mathcal{S}_{0,\rho}^{c,b}(0; u)} {\mathcal{S}_{0,\rho}^{c,b}(0; 1)}.
\end{align*}
\item[(ii)] Otherwise, $\mathbb{Q}_{u} \left( U_{\infty} = 0 \right) = 1$. 
\end{itemize}
\end{proposition}

\begin{proof} 
It is well-known that the process $(\mathcal{S}_{0, \rho}^{c,b}(\theta; U_{t}), t \geq 0)$ is a regular diffusion on a natural scale; see e.g., \cite[Proposition VII.3.4]{Revuz1999}. Then, our claim follows from \cite[Theorem 23.15]{Kalla2002}.
\end{proof}

Finally, we study the case $d >0$.  

\begin{proposition} \label{Pro7}
Suppose that $u \in (0,1)$ and that $d>0$.
\begin{itemize}
\item[(i)] If $\mathcal{S}_{d,\rho}^{c,b}(\theta; 1) < \infty$, for some $\theta \in (0,1)$, then $\lim_{t \rightarrow \infty} U_{t} = 1$ almost surely. 

\item[(ii)] Otherwise, $U$ is null-recurrent.
\end{itemize}
\end{proposition}

\begin{proof}
Recall that $(\mathcal{S}_{d, \rho}^{c,b}(\theta; U_{t}), t \geq 0)$ is a regular diffusion on a natural scale; see e.g., \cite[Proposition VII.3.4]{Revuz1999}. Then, our claim follows again from \cite[Theorem 23.15]{Kalla2002}. We only need to be careful when $\mathcal{S}_{d,\rho}^{c,b}(\theta; 1) = \infty$, for some $\theta \in (0,1)$, that is, part (ii). In this case, we have that $\mathcal{I}_{d,\rho}^{c,b}(\theta; 1) = \infty$. To see this, note that 
\begin{align*}
\mathcal{I}_{d, \rho}^{c, b}(\theta; x) =  \int_{\theta}^{x} (\mathcal{S}_{d, \rho}^{c, b}(\theta; x)-\mathcal{S}_{d, \rho}^{c, b}(\theta; u)) \mathcal{R}_{d, \rho}^{c, b}(\theta; {\rm d} u),\qquad \textrm{ for} \quad x \in [\theta, 1], 
\end{align*} 
\noindent where $\mathcal{R}_{d, \rho}^{c, b}(\theta; {\rm d} u)$ has  a strictly positive density with respect to the Lebesgue measure  on $(\theta,1)$. Therefore the fact that $\mathcal{I}_{d,\rho}^{c,b}(\theta; 1) = \infty$ implies  that $\varsigma =0$ and $1$ is a natural boundary by Lemma \ref{lemma15} and Remark \ref{remark9}.

On the other hand, by our assumption (\ref{mainE1}), we always have that the so-called speed measure (see \eqref{SpeedM} for a definition) of $U$ (or equivalently of $(\mathcal{S}_{d, \rho}^{c,b}(\theta; U_{t}), t \geq 0)$) is unbounded. Therefore, \cite[Theorem 23.15]{Kalla2002} implies that $U$ is null-recurrent whenever $\mathcal{S}_{d,\rho}^{c,b}(\theta; 1) = \infty$. 
\end{proof}

%%%%%%%%%%%%%%%%%%%%%%%%%%%%%%%%%%%%%%%%%%%%%%%%%%%
\section{Extinction and stationarity: Proof of Theorem \ref{lemma4}, Theorem \ref{lemma6}, Lemma \ref{lemma11} and Theorem \ref{Lamtype}} \label{sec4}
%%%%%%%%%%%%%%%%%%%%%%%%%%%%%%%%%%%%%%%%%%%%%%%%%%%

The aim of this section is to prove Theorems \ref{lemma4} and \ref{lemma6} that are devoted to the event of extinction and the existence of a stationary distribution for (sub)critical cooperative BPI-processes. The proofs are based on the duality relationship presented in Section \ref{sec3}.  We also provide some details of the proofs of Lemma \ref{lemma11} and Theorem \ref{Lamtype} which are  natural extensions of the results of Lambert \cite{La2005}. 

Recall that $U = (U_{t}, t  \geq 0)$ is the dual process of the BPI-process defined in Theorem \ref{Theo2}. 
\begin{proof}[Proof of Theorem \ref{lemma4}]
The duality relationship (\ref{eq11}) in Theorem  \ref{Theo2} and the Dominated Convergence Theorem imply that
\begin{eqnarray} \label{eq37}
\lim_{t \rightarrow \infty} \mathbb{E}_{z}[u^{Z_{t}}] =  \mathbb{E}_{z}[u^{\lim_{t \rightarrow \infty} Z_{t}}] = \mathbb{E}_{\mathbb{Q}_{u}}\left[ \left(\lim_{t \rightarrow \infty} U_{t} \right)^{z} \right],  \hspace*{5mm} \text{for} \hspace*{3mm} z \in \mathbb{N} \hspace*{3mm} \text{and} \hspace*{3mm} u \in [0,1].
\end{eqnarray}
\noindent Since $d >0$, $\mathcal{S}_{d,\rho}^{c,b}(\theta; 1) < \infty$, for some $\theta \in (0,1)$, and (\ref{mainE1}) is satisfied, we conclude from Proposition \ref{Pro7} (i) that $\mathbb{P}_{z} \left(\lim_{t \rightarrow \infty} Z_{t} = 0 \right)= 1$. Moreover, $\mathbb{P}_{z}(\lim_{t \rightarrow \infty} Z_{t} = 0 ) = \mathbb{P}_{z}(\zeta_{0} < \infty)$ which implies Theorem \ref{lemma4} (i). Theorem \ref{lemma4} (ii) follows from (\ref{eq37}) and Proposition \ref{Pro7} (ii). 
\end{proof}

Next, we prove Theorem \ref{lemma6} which deals with the stationarity of the BPI-process.
\begin{proof}[Proof of Theorem \ref{lemma6}]
Since $d=0$, our claim follows from \eqref{eq37} and Proposition \ref{lemma5}.
\end{proof}

Finally, we provide some details of the proofs of Lemma \ref{lemma11} and Theorem \ref{Lamtype} just to convince the reader that everything can be carried out as in \cite{La2005}.

\begin{proof}[Proof of Lemma \ref{lemma11}]
We start by proving (i). The fact that the function $H_{q,z} \in C^{2}([0,1), \mathbb{R})$ follows easily by the Dominated Convergence Theorem. Moreover, the function $f(u) \coloneqq u^{z}$, for $z \in \mathbb{N}_{0}$ and $u \in [0,1]$, belongs to $\mathscr{D}(\mathscr{L}^{{\rm BPI}})$. Then the forward Kolmogorov equation implies that $q H_{q,z}(u) = u^{z} + \int_{0}^{\infty} e^{-qt} \mathbb{E}_{z}[\mathscr{L}^{{\rm BPI}} f(Z_{t})] {\rm d} t$, where $\mathscr{L}^{{\rm BPI}}$ denotes the generator of the BPI-process defined in (\ref{eq8}). Therefore, the second claim in (i) follows from (\ref{eq13}).

Now, we prove (ii). Since $f_{q}$ solves the homogeneous equation $(E_{q}^{h})$ associated with (\ref{eq1}), we have that $- u \Phi_{c,b}(u) f_{q}^{\prime \prime}(u) - \Psi_{d,\rho}(u)f_{q}^{\prime}(u) + q f_{q}(u) = 0$, for $u \in I$. Note from the definition of $g_{q}$ in (ii) that $f^{\prime \prime}_{q}(u) / f_{q}(u) = g_{q}^{\prime}(u) + g_{q}^{2}(u)$, for $u \in I_{0}$. By taking $u \in I_{0}$ such that $f_{q}(u) \neq 0$, the claim in (ii) follows by dividing the previous differential equation by $-u \Phi_{c,b}(u) f_{q}(u)$. 

Next, we show (iii). From the definition of the function $h_{q}$ in (iii), we see that
\begin{eqnarray*}
h_{q}^{\prime}(u) = e^{-2m(\varphi(u))} \left( m^{\prime}(\varphi(u)) g_{q}(\varphi(u))- g^{\prime}_{q}(\varphi(u)) \right) = e^{-2m(\varphi(u))} \left(  g_{q}^{2}(\varphi(u))- \frac{q}{\varphi(u) \Phi_{c,b}(\varphi(u))} \right), 
\end{eqnarray*}

\noindent for $u \in \tau(J)$, where we have used that $g_{g}$ satisfies (\ref{eq3}) and $m^{\prime}(u) = - \Psi_{d,\rho}(u)/(u \Phi_{c,b}(u))$ to obtain the last equality. By recalling that $\varphi^{\prime}(u) = - e^{-m(\varphi(u))}$, we obtain that 
\begin{eqnarray*}
h_{q}^{\prime}(u) = h_{q}^{2}(u) - \frac{q \varphi^{\prime}(u)^{2}}{\varphi(u) \Phi_{c,b}(\varphi(u))}, 
\end{eqnarray*}

\noindent which implies the claim in (iii). 

Finally, (iv) follows from a simple adaptation of the argument of \cite[Proof of Lemma 2.1]{La2005}.
\end{proof}

\begin{proof}[Proof of Theorem \ref{Lamtype}]
The proof follows along the lines of  \cite[Theorem 2.3 and Theorem 3.9]{La2005}. We only need to note that $\mathcal{Q}_{d, \rho}^{c, b}(\theta; 1) \in (-\infty, \infty)$, for some $\theta \in (0,1)$, implies that $\mathcal{S}_{d, \rho}^{c, b}(\theta; 1) < \infty$ and therefore, $\mathbb{P}_{z}(\zeta_{0} < \infty) = 1$, for $z \in \mathbb{N}$, by Theorem \ref{lemma4} (i). 
\end{proof}

%%%%%%%%%%%%%%%%%%%%%%%%%%%%%%%%%%%%%%%%%%%%%%%%
\section{Coming-down from infinity: Proof of Theorem \ref{lemma9}} \label{sec5}
%%%%%%%%%%%%%%%%%%%%%%%%%%%%%%%%%%%%%%%%%%%%%%%%

In this section, we prove Theorem \ref{lemma9} which deals with  the property of coming-down from infinity. We also show how to construct a version of the (sub)critical cooperative BPI-process coming-down from infinity, i.e., a BPI-process with an entrance law from $\{ \infty\}$. Our approach is based on \cite[Lemmas 6.2 and  6.3]{Cle2017}. Let $U = (U_{t}, t  \geq 0)$ be the dual process of the (sub)critical cooperative BPI-process defined in Theorem \ref{Theo2}. 

\begin{lemma}
Suppose that $\mathcal{I}_{d, \rho}^{c, b}(\theta; 1) < \infty$, for some $\theta \in (0,1)$, and that (\ref{mainE1}) is satisfied. For any $t >0$, $u \mapsto \mathbb{Q}_{u}(U_{t} = 1)$, for $u \in [0,1]$, is the Laplace transform of certain probability measure $\eta_{t}$ over $\mathbb{N}_{0}$. Moreover, $\eta_{t} \rightarrow \eta_{0} \coloneqq \delta_{\{ \infty\}}$ weakly, as $t \rightarrow 0$. 
\end{lemma}

\begin{proof}
The duality relationship (\ref{eq11}) in Theorem  \ref{Theo2} implies that
\begin{eqnarray*}
\lim_{z \rightarrow \infty} \mathbb{E}_{z}[u^{Z_{t}}]  = \lim_{z \rightarrow \infty} \mathbb{E}_{\mathbb{Q}_{u}}[U_{t}^{z}] = \mathbb{Q}_{u}(U_{t} = 1), \hspace*{4mm} \text{for} \hspace*{4mm} u \in [0,1] \hspace*{2mm} \text{and} \hspace*{2mm} t \geq 0. 
\end{eqnarray*}

\noindent Since $\mathcal{I}_{d, \rho}^{c, b}(\theta; 1) < \infty$, for some $\theta \in (0,1)$, and (\ref{mainE1}) is satisfied, Lemma \ref{lemma15} and Remark \ref{remark9} imply that $1$ is an exit boundary for the dual process $U$. Furthermore, \cite[Chapter 15, Theorem 7.1]{Taylor1981} shows that $\lim_{u \rightarrow 1} \mathbb{E}_{\mathbb{Q}_{u}}[e^{-\lambda \tau_{1}}] = 1$, for $\lambda >0$ and where $\tau_{1} = \inf \{ t \geq 0: U_{t} =1\}$. So, for $u$ close enough to $1$, we get for any $t >0$ that $\mathbb{Q}_{u} (U_{t} = 1) = \mathbb{Q}_{u} (\tau_{1} \leq t) >0$. Moreover, $\lim_{u \rightarrow 1} \mathbb{Q}_{u} (\tau_{1} \leq t) =1$.  Therefore,  L\'evy's continuity theorem (see e.g., \cite[Theorem 5.3]{Kalla2002}) implies that $u \mapsto \mathbb{Q}_{u}(U_{t} = 1)$ is the Laplace transform of certain finite measure $\eta_{t}$ on $\mathbb{N}_{0}$, which is the weak limit of the law of $Z_{t}$ under $\mathbb{P}_{z}$, as $z \rightarrow \infty$. Indeed, since $U$ has continuous paths, if $u \in [0,1)$, then $\lim_{t \rightarrow 0} \mathbb{Q}_{u}(U_{t} = 1)= 0$, and if $u =1$ then $\lim_{t \rightarrow 0} \mathbb{Q}_{1}(U_{t} = 1)= 1$. This entails that $\lim_{t \rightarrow 0} \eta_{t} = \delta_{\{ \infty \}}$ weakly. 
\end{proof}

\begin{lemma} \label{lemma20}
Suppose that $\mathcal{I}_{d, \rho}^{c, b}(\theta; 1) < \infty$, for some $\theta \in (0,1)$, and that (\ref{mainE1}) is satisfied. For any function $f \in B(\mathbb{N}_{0,\infty}, \mathbb{R})$ and any $t \geq 0$, set $P_{t}f(z) \coloneqq \mathbb{E}_{z}[f(Z_{t})]$, for $z \in \mathbb{N}_{0}$, and $P_{t}f(\infty) \coloneqq \sum_{z=0}^{\infty} f(z) \eta_{t}(\{z\})$. This defines a Feller semigroup $(P_{t}, t \geq 0)$ over $\mathbb{N}_{0,\infty}$. Furthermore, if $(Z_{t}, t \geq 0)$ is a c\`adl\`ag Markov process with semigroup $(P_{t}, t \geq 0)$, and $\zeta  \coloneqq \inf \{ t \geq 0: Z_{t} < \infty \}$, then $\mathbb{P}_{\infty}(\zeta = 0)=1$. 
\end{lemma}

\begin{proof}
This result is the discrete-space counterpart of \cite[Lemma 6.3]{Cle2017} and can be proved following exactly the same argument. 
\end{proof}

\begin{proof}[Proof of Theorem \ref{lemma9}]
It follows from Lemma \ref{lemma20} since we have that $\mathbb{P}_{\infty}(\zeta = 0) = \mathbb{P}_{\infty} ( \, \forall \, \, t >0, Z_{t} < \infty ) = 1$; see \cite[Definition 2.1 and Proposition 2.2]{Ban2016}.
\end{proof}

\appendix
%%%%%%%%%%%%%%%%%%%%%%%%%%%%%%%%%%%%%%%%%%%%%%%%%%%
\section{Appendix} \label{Apendice}
%%%%%%%%%%%%%%%%%%%%%%%%%%%%%%%%%%%%%%%%%%%%%%%%%%%
Here, we provide the proofs of the technical Corollary \ref{corollary4}, Lemmas \ref{lemma1Red}, \ref{lemma2Red}, \ref{lemma12}, \ref{lemma13}, \ref{Pro3}, \ref{lemma14} and \ref{lemma15}. 
 
%We start with the proofs of Lemmas \ref{lemma12}, \ref{lemma13} and \ref{Pro3}.

\begin{proof}[Proof of Corollary \ref{corollary4}]
Suppose first that $\varsigma < 0$. Note (\ref{eq9}) implies that, for $\theta \in (0,1)$ and $u \in (\theta, 1)$,
\begin{align*}
\mathcal{Q}_{d, \rho}^{c,b}(\theta; u) = \int_{\theta}^{u} \frac{d(1-w)}{w \Phi_{c,b}(w)} {\rm d}w + \mathcal{Q}_{0, \rho}^{c,b}(\theta; u).
\end{align*}
\noindent It follows from (\ref{eq02}) that $\Phi_{c,b}(u) \geq - \varsigma (1-u)$, for $u \in [0,1]$. Then, for $\theta \in (0,1)$ and $u \in (\theta, 1)$,
\begin{align*}
\mathcal{Q}_{0, \rho}^{c,b}(\theta; u) \leq \mathcal{Q}_{d, \rho}^{c,b}(\theta; u) \leq - \frac{d (\theta -u)}{\varsigma \theta}  + \mathcal{Q}_{0, \rho}^{c,b}(\theta; u).
\end{align*}
\noindent The above implies that $\mathcal{Q}_{d, \rho}^{c,b}(\theta; 1)$ and $\mathcal{Q}_{0, \rho}^{c,b}(\theta; 1)$ have the same nature whenever $\varsigma < 0$. Then, Corollary \ref{corollary4} (i) follows from Lemma \ref{lemma13}, Theorem \ref{Theo1} and Theorem \ref{Theo3}. 

If $\varsigma = 0$, Corollary \ref{corollary4} (ii)  follows from Lemma \ref{lemma13} and Theorem \ref{Theo3}. 
\end{proof}

\begin{proof}[Proof of Lemma \ref{lemma12}]
Note that $\Psi_{d, \rho}(u) \leq (1-u)d$, for $u \in [0,1]$; see (\ref{eq9}). Then, (\ref{eq02}) implies that $\mathcal{Q}_{d, \rho}^{c,b}(\theta;x) \leq  - d(x-\theta) (\theta \varsigma)^{-1}$, for  $x \in (\theta, 1)$, and our claim follows.
%Recall that  $(1-u^{i}) = (1-u)\sum_{k = 0}^{i-1} u^{k}$, for $i \in \mathbb{N}$, thus we can rewrite the branching and cooperation mechanisms as follows
%\begin{eqnarray*}
%\Psi_{d, \rho}(u) = (1-u)\left(d - u\sum_{i \geq 1} \pi_{i} \sum_{k=0}^{i-1} u^{k} \right) \qquad \text{and} \qquad \Phi_{c, b}(u) = (1-u)\left(c - u\sum_{i \geq 1} b_{i} \sum_{k=0}^{i-1} u^{k} \right),
%\end{eqnarray*}
%for $u \in [0,1]$. Observe that $\Psi_{d, \rho}(u) \leq (1-u)d$ and  $\Phi_{c, b}(u) \geq -(1-u)\varsigma$ since $ \varsigma  < 0$. Therefore, for $a \in (\theta, 1)$,
%\begin{eqnarray*}
%\mathcal{Q}_{d, \rho}^{c,b}(\theta;a) = \int_{\theta}^{a} \frac{\Psi_{d, \rho}(u)}{u\Phi_{c, b}(u) }{\rm d}u\le -\frac{d(a-\theta)}{\theta \varsigma}.
%\end{eqnarray*}
%Since $\lim_{a \rightarrow 1}\mathcal{Q}_{d, \rho}^{c,b}(\theta;a) = \mathcal{Q}_{d, \rho}^{c,b}(\theta;1)$, our claim follows from the previous inequality.  
\end{proof}

\begin{proof}[Proof of Lemma \ref{lemma13}]
Since $\mathcal{Q}_{d, \rho}^{c,b}(\theta;1) \in (-\infty, \infty)$, there exists a constant $C >0$ such that $\mathcal{J}_{d, \rho}^{c,b}(\theta;x) \geq C$, for $x \in [\theta, 1]$ (one could choose $\theta$ closer to $1$ if necessary). From \eqref{eq02}, we know that  $\Phi_{c,b}(u) \leq c(1-u)$, for $u \in [0,1]$. Then, there is a constant $C^{\prime} >0$ such that $\mathcal{R}_{d, \rho}^{c,b}(\theta; 1) \geq C^{\prime} \int_{\theta}^{1} (1-x)^{-1} {\rm d}x= \infty$. Moreover, since $\mathcal{E}_{d, \rho}^{c,b}(\theta; 1) = \int_{\theta}^{1} \mathcal{R}_{d, \rho}^{c,b}(u; 1) {\rm d} u$, it follows that $\mathcal{E}_{d, \rho}^{c,b}(\theta; 1) = \infty$.
%\Gabriel{Hay problema con la prueba de la segunda parte, no esta bien. No creo q sea verdad}
%\textcolor{red}{Next, we consider that $\varsigma < 0$ and suppose that $\mathcal{Q}_{d, \rho}^{c,b}(\theta;1) = -\infty$, for some  $\theta \in (0,1)$. Then, from the lower bound in \eqref{eq02}, we see that $\Phi_{c,b}(u) \geq 0$, for $u \in [0,1]$. Hence there exists $\theta \in (0,1)$ such that $\Psi_{d, \rho}(u) < 0$, for $u \in (\theta, 1)$, and $\Psi_{d, \rho}$ is increasing on $(\theta, 1]$. So, by the mean value theorem there exists $\theta^{\prime} \in (\theta, 1)$ such that
%\textcolor{blue}{\begin{eqnarray*}
%\mathcal{R}_{d, \rho}^{c,b}(\theta; 1) =  \int_{\theta}^{1} \frac{\Psi_{d, \rho}(x)}{x\Phi_{c,b}(x) \Psi_{d, \rho}(x)} e^{ \mathcal{Q}_{d, \rho}^{c,b}(\theta;x)} {\rm d}x = \frac{1}{\Psi_{d,\rho}(\theta^{\prime})} \int_{\theta}^{1} \frac{\Psi_{d, \rho}(x)}{x\Phi_{c,b}(x)} e^{ \mathcal{Q}_{d, \rho}^{c,b}(\theta;x)} {\rm d}x = - \frac{1}{\Psi_{d,\rho}(\theta^{\prime})}.
%\end{eqnarray*}}
%\noindent Thus, $\mathcal{R}_{d, \rho}^{c,b}(\theta;1) < \infty$ and $\mathcal{E}_{d, \rho}^{c,b}(\theta; 1) < \infty$. Therefore, our second claim follows from Lemma \ref{lemma12}.}
\end{proof}

\begin{proof}[Proof of Lemma \ref{NEWlemma1}]
Let $\mathbf{m} \coloneqq - d + \sum_{i \geq 1} i \pi_{i}$. We can and will assume that $\mathbf{m} > 0$. To see this, note that \eqref{cebolla} and \eqref{eq02} imply that $|\mathcal{Q}_{d, \rho}^{c,b}(\theta; 1)| < \infty$ whenever $\mathbf{m} \leq 0$ and in particular, Lemma \ref{lemma13} implies that $\mathcal{R}_{d, \rho}^{c,b}(\theta; 1) = \infty$. We henceforth assume that  $\mathbf{m} > 0$. Note that, for $u \in (\theta, 1)$, $\mathcal{E}_{d, \rho}^{c,b}(\theta; u) \leq  \mathcal{R}_{d, \rho}^{c,b}(\theta; 1)  \mathcal{S}_{d, \rho}^{c,b}(\theta; u)$. Therefore, we only need to verify that $\mathcal{S}_{d, \rho}^{c,b}(\theta; 1)< \infty$. Since $\mathbf{m} > 0$, we consider $\theta \in (0,1)$ such that $\Psi_{d, \rho}$ is increasing and negative on $[\theta, 1)$. Note that \eqref{eq02} implies that
\begin{align}
\mathcal{Q}_{d, \rho}^{c,b}(\theta; x) \geq - \frac{\Psi_{d, \rho}(\theta)}{\varsigma} \int_{\theta}^{x} \frac{{\rm d} w}{w (1-w)} =  - \frac{\Psi_{d, \rho}(\theta)}{\varsigma} \ln \left( \frac{x(1-\theta)}{(1-x)\theta}\right),
\end{align}
\noindent and thus, by choosing $\theta \in (0,1)$ such that $\Psi_{d, \rho}(\theta) / \varsigma < 1$,
\begin{align}
\mathcal{S}_{d, \rho}^{c,b}(\theta; 1) \leq \int_{\theta}^{1} \left( \frac{x(1-\theta)}{(1-x)\theta}\right)^{\frac{\Psi_{d, \rho}(\theta)}{\varsigma}} {\rm d} x < \infty.
\end{align}
\end{proof}

\begin{proof}[Proof of Lemma \ref{Pro3}]
Since $\Psi_{0, \rho}(w)  = \Psi_{d,\rho}(w) - d(1-w)$, for $w \in [0,1]$, we have 
\begin{eqnarray} \label{meq1}
\mathcal{Q}_{0, \rho}^{c,b}(u;x) = \int_{u}^{x} \frac{\Psi_{0, \rho}(w)}{w \Phi_{c,b}(w)} {\rm d}w = \mathcal{Q}_{d, \rho}^{c,b}(u;x) - \int_{u}^{x} \frac{d(1-w)}{w \Phi_{c,b}(w)} {\rm d}w,
\end{eqnarray}
\noindent for $u \in (0,1)$  and  $x \in [u,1)$. Since $\varsigma < 0$, (\ref{eq02}) implies
\begin{eqnarray} \label{meq1b}
\frac{d}{c} \ln \left(\frac{x}{u} \right)\leq \int_{u}^{x} \frac{d(1-w)}{w \Phi_{c,b}(w)} {\rm d}w  \leq -\frac{d}{\varsigma} \ln \left(\frac{x}{u} \right), \hspace*{5mm} \text{for} \hspace*{3mm} u \in (0,1)  \hspace*{2mm} \text{and} \hspace*{2mm} x \in [u,1).
\end{eqnarray}
\noindent From the above inequality and (\ref{meq1}), one concludes the first claim. 

Next, we assume that $b=0$ and observe that $\Psi_{c, \rho}(w)  = \Psi_{d,\rho}(w) + (c-d)(1-w)$, for $w \in [0,1]$. Hence 
\begin{eqnarray*} 
\mathcal{Q}_{c, \rho}^{c,0}(u;x) = \int_{u}^{x} \frac{\Psi_{c, \rho}(w)}{w \Phi_{c,0}(w)} {\rm d}w = \int_{u}^{x} \frac{\Psi_{d,\rho}(w)}{w \Phi_{c,0}(w)} {\rm d}w + \frac{c-d}{c}\ln \left(\frac{x}{u} \right), \hspace*{4mm} \text{for} \hspace*{3mm} u \in (0,1)  \hspace*{2mm} \text{and} \hspace*{2mm} x \in [u,1),
\end{eqnarray*}
which clearly implies the second claim.  
\end{proof}

%If $\varsigma < 0$, our claim follows from Lemma \ref{lemma13} and Theorem \ref{Theo1}. If $\varsigma = 0$ and $d=0$, our claim follows from Lemma \ref{lemma13} and Theorem \ref{Theo3} (i). Suppose that $\varsigma = 0$ and $d>0$, it follows from (\ref{eq02}) that, for $\theta \in (0,1)$ and $u \in (\theta, 1)$,
%\begin{eqnarray*}
%\int_{\theta}^{u} \frac{\Psi_{0,\rho}(w)}{w \Phi_{d,\rho}(w)} {\rm d}w \leq \mathcal{Q}_{d, \rho}^{c,b}(\theta; u) = \int_{\theta}^{u} \frac{d(1-w)}{w \Phi_{d,\rho}(w)} {\rm d}w + \int_{\theta}^{u} \frac{\Psi_{0,\rho}(w)}{w \Phi_{d,\rho}(w)} {\rm d}w. 
%\end{eqnarray*}
%\end{proof}

\begin{proof}[Proof of Lemma \ref{lemma1Red}]
It follows from \eqref{meq1} that, $0 < \theta < x < 1$,
\begin{align} \label{Ident11}
\mathcal{Q}_{d, \rho}^{c, b}(\theta; x) = \mathcal{Q}_{0, \rho}^{c, b}(\theta; x) + \mathcal{Q}_{d, 0}^{c, b}(\theta; x).
\end{align}
\noindent  On the other hand, \eqref{cebolla} and \eqref{eq02}, we have that
\begin{align} \label{Ident12}
 \widetilde{\mathcal{Q}}_{\rho}^{c, b}(\theta; x) \leq \mathcal{Q}_{0, \rho}^{c, b}(\theta; x) \leq  -\frac{\varsigma}{c}   \widetilde{\mathcal{Q}}_{\rho}^{c, b}(\theta; x),
\end{align}
\noindent Since, by \eqref{meq1b},
\begin{align} \label{Ident13}
\lim_{x \uparrow 1} \mathcal{Q}_{d,0}^{c, b}(\theta; x) = 0,
\end{align}
\noindent our first claim follows from \eqref{Ident11} and \eqref{Ident12}. 
Note that, by \eqref{cebolla}, $\Phi_{c,b}(w)\sim-\varsigma(1-w)$ as $w\uparrow 1$. Hence again from \eqref{cebolla}, For all $\varepsilon >0$, there exists $\theta^{\prime} \in (0,1)$ such that for $\theta^{\prime} < \theta < x < 1$ such that
\begin{align*}
\frac{1}{1-\varepsilon} \widetilde{\mathcal{Q}}_{\rho}^{c, b}(\theta; x) \leq \mathcal{Q}_{0, \rho}^{c, b}(\theta; x)  \leq \frac{1}{1+\varepsilon}\widetilde{\mathcal{Q}}_{\rho}^{c, b}(\theta; x),
\end{align*}
\noindent and our second claim follows from \eqref{Ident11} and \eqref{Ident13}. 
\end{proof}

\begin{proof}[Proof of Lemma \ref{lemma2Red}]
It follows from \eqref{Ident11} (with $b=0$) and \eqref{cebolla} that, for $0 < \theta < x < 1$,
\begin{align*}
\mathcal{R}_{d, \rho}^{c, 0}(\theta; x)  & = \int_{\theta}^{x} \frac{1}{c(1-u)} e^{\mathcal{Q}_{0, \rho}^{c, 0}(\theta; u) + \mathcal{Q}_{d,0}^{c, 0}(\theta; u) } {\rm d} u \nonumber \\
& = \exp\left(\frac{1}{c}\sum_{i\ge 1}\frac{\overline{\pi}_i}{i}\theta^i\right) \int_{\theta}^{x} \frac{u^{\frac{d}{c}}}{\theta^{\frac{d}{c}}c(1-u)} \exp\left(-\frac{1}{c}\sum_{i\ge 1}\frac{\overline{\pi}_i u^i}{i}\right) {\rm d} u, 
\end{align*}
\noindent which implies our claim. 
\end{proof}

Finally, we study the boundaries $0$ and $1$ of the operator $\mathscr{L}^{{\rm dual}}$ and prove Lemmas \ref{lemma14} and \ref{lemma15}. We will use Feller's boundary classification for which we refer to \cite[Chapter 8, Section 1]{Et1986} (or \cite[Chapter 23]{Kalla2002}). Following the presentation of \cite[Chapter 8, Section 1]{Et1986}, for $x \in [0,1]$, note, after some simple interchanges of order integration, that
\begin{eqnarray*}
\mathcal{I}_{d, \rho}^{c,b}(\theta; x) = \int_{\theta}^{x} \mathcal{R}_{d, \rho}^{c,b}(\theta; u)\mathcal{S}_{d, \rho}^{c,b}(\theta;  {\rm d} u) \qquad \text{and} \qquad \mathcal{E}_{d, \rho}^{c,b}(\theta; x) = \int_{\theta}^{x} \mathcal{S}_{d, \rho}^{c,b}(\theta; u) \mathcal{R}_{d, \rho}^{c,b}(\theta; {\rm d} u).
\end{eqnarray*}
\noindent Then, for $i=0,1$, the boundary $i$ is said to be
\begin{table}[!htbp]
\centering
\begin{tabular}{lcccc}
regular  & if  & $\mathcal{I}_{d, \rho}^{c,b}(\theta; i) < \infty$  & and & $\mathcal{E}_{d, \rho}^{c,b}(\theta; i) < \infty$,  \\
exit     & if  & $\mathcal{I}_{d, \rho}^{c,b}(\theta; i) < \infty$  & and & $\mathcal{E}_{d, \rho}^{c,b}(\theta; i) = \infty$, \\
entrance & if  & $\mathcal{I}_{d, \rho}^{c,b}(\theta; i) = \infty$  & and & $\mathcal{E}_{d, \rho}^{c,b}(\theta; i) < \infty$, \\
natural  & if  & $\mathcal{I}_{d, \rho}^{c,b}(\theta; i) = \infty$  & and & $\mathcal{E}_{d, \rho}^{c,b}(\theta; i) = \infty$.
\end{tabular}
\end{table}

\noindent Regular and exit boundaries are said to be accesible; entrance and natural boundaries are said to be inaccesible. We also introduce the so-called speed measure $M_{d, \rho}^{c,b}$ by letting
\begin{eqnarray} \label{SpeedM}
M_{d, \rho}^{c,b}([u,v]) = \int_{u}^{v} \frac{1}{2x\Phi_{c,b}(x)} \exp \left( \int_{\theta}^{x} \frac{\Psi_{d, \rho}(y)}{y\Phi_{c,b}(y)} {\rm d}y \right) {\rm d} x, \hspace*{4mm} \text{for} \hspace*{2mm} [u,v] \subset (0,1),
\end{eqnarray}
\noindent where $\theta \in (0,1)$ is an arbitrary point. If $i$ is regular, the boundary behaviour depends on the value of $M_{d, \rho}^{c,b}(\{i\})$; see \cite[Theorem 23.12]{Kalla2002} or \cite[Definition VII.3.11]{Revuz1999}. Thus, in this case, the boundary $i$ is 
\begin{table}[!htbp]
\centering
\begin{tabular}{lcl}
regular absorbing  & if  & $M_{d, \rho}^{c,b}(\{ i\}) = \infty$,  \\
regular slowly reflecting & if  & $M_{d, \rho}^{c,b}(\{ i\}) \in (0, \infty)$, \\
regular (instantaneously) reflecting & if  & $M_{d, \rho}^{c,b}(\{ i\}) = 0$.
\end{tabular}
\end{table}
~\\
\begin{proof}[Proof of Lemma \ref{lemma14}]
By condition (\ref{main}), $\lim_{u \rightarrow 0} \Phi_{c, b}(u) = c$ and in particular, $\Phi_{c, b}$ is bounded away from zero. 

We start with the proof of (i) and suppose that $d > c$. Since $\Phi_{c, b}$ is bounded away from zero, for any $0< \varepsilon_{1} < d/c -1$, there exist $\theta_{1} \in (0,\theta)$ and a constant $C_{1}>0$ such that (\ref{eq35}) implies that
\begin{eqnarray*}
\mathcal{I}_{d, \rho}^{c,b}(\theta_{1}; 0) \geq \int_{0}^{\theta_{1}} \int_{0}^{x} \frac{1}{x\Phi_{c,b}(x)} \left(\frac{x}{y} \right)^{d/c-\varepsilon_{1}} {\rm d}y {\rm d}x \geq C_{1} \int_{0}^{\theta_{1}} \int_{0}^{x} \frac{1}{x}  \left(\frac{x}{y} \right)^{d/c-\varepsilon_{1}} {\rm d}y {\rm d}x = \infty.
\end{eqnarray*} 

\noindent On the other hand, since $d >0$ and $\Phi_{c, b}$ is bounded away from zero, for any $0 < \varepsilon_{2} < d/c$, there exist $\theta_{2} \in (0,\theta)$ and a constant $C_{2}>0$ such that (\ref{eq35}) implies that
\begin{eqnarray} \label{eq28}
\mathcal{E}_{d, \rho}^{c,b}(\theta_{2}; 0) \leq \int_{0}^{\theta_{2}} \int_{0}^{x}\frac{1}{y\Phi_{c,b}(y)} \left(\frac{x}{y} \right)^{-d/c+\varepsilon_{2}}  {\rm d}y {\rm d}x \leq C_{2} \int_{0}^{\theta_{3}} \int_{0}^{x}\frac{1}{y} \left(\frac{x}{y} \right)^{-d/c+\varepsilon_{2}}  {\rm d}y {\rm d}x < \infty.
\end{eqnarray}

\noindent This shows that $0$ is an entrance boundary if $d > c$ which concludes the proof of (i). 

Next, we prove (ii) and suppose that $0 \leq d < c$. Since $\Phi_{c, b}$ is bounded away from zero, for any $0 < \varepsilon_{3} < 1-d/c$, there exist $\theta_{3} \in (0,\theta)$ and a constant $C_{3} >0$ such that (\ref{eq35}) implies that
\begin{eqnarray*}
\mathcal{I}_{d, \rho}^{c,b}(\theta_{3}; 0) \leq \int_{0}^{\theta_{3}} \int_{0}^{x} \frac{1}{x\Phi_{c,b}(x)} \left(\frac{x}{y} \right)^{d/c+\varepsilon_{3}} {\rm d}y {\rm d}x \leq C_{3} \int_{0}^{\theta_{3}} \int_{0}^{x} \frac{1}{x}  \left(\frac{x}{y} \right)^{d/c+\varepsilon_{3}} {\rm d}y {\rm d}x < \infty,
\end{eqnarray*} 

\noindent If $0 < d < c$, (\ref{eq28}) implies that $\mathcal{E}_{d, \rho}^{c,b}(\theta; 0) < \infty $, for some $\theta \in (0,1)$. Thus, $0$ is a regular boundary if $0 < d < c$. Since $d >0$ and $\lim_{u \rightarrow 0} \Psi_{d, \rho}(u) = d$, we have that $\Psi_{d, \rho}$ is bounded away from zero. Then, for $[u,v] \subset (0,1)$, there exists a constant $C_{4}>0$ such that
\begin{eqnarray} \label{regularR}
M_{d, \rho}^{c,b}([u, v]) \leq C_{4} \int_{u}^{v} \frac{\Psi_{d, \rho}(x)}{x\Phi_{c,b}(x)} \exp \left( \int_{\theta}^{x} \frac{\Psi_{d, \rho}(y)}{y\Phi_{c,b}(y)} {\rm d}y \right) {\rm d} x = C_{4} ( e^{\mathcal{Q}_{d, \rho}^{c,b}(\theta; v)} - e^{\mathcal{Q}_{d, \rho}^{c,b}(\theta; u)}).
\end{eqnarray} 

\noindent It is not difficult to see that $\lim_{x \rightarrow 0} \mathcal{Q}_{d, \rho}^{c,b}(\theta; x) = -\infty$ and thus, $M_{d, \rho}^{c,b}(\{0\}) = 0$, i.e., $0$ is regular reflecting. If $d=0$, then $\Psi_{d,\rho}(u) \leq 0$ for all $u \in [0,1]$. Thus, since $\Phi_{c, b}$ is bounded away from zero, there is a constant $C_{5} >0$ such that $\mathcal{E}_{d, \rho}^{c,b}(\theta; 0) \geq C_{5} \int_{0}^{\theta} \int_{0}^{x}  y^{-1} {\rm d}y {\rm d}x = \infty$. Hence $0$ is an exit boundary if $d=0$ which concludes with the proof of (ii).

Finally, we prove (iii), and thus, assume that $d=c$. Note that 
\[
\mathcal{I}_{d, \rho}^{c,b}(\theta; x) = \int_{x}^{\theta} (\mathcal{S}_{d, \rho}^{c,b}(\theta; u) - \mathcal{S}_{d, \rho}^{c,b}(\theta;x)) \mathcal{R}_{d, \rho}^{c,b}(\theta;  {\rm d} u),\qquad \textrm{ for }\quad x \in [0, \theta].
\] 
\noindent If $\mathcal{S}_{c, \rho}^{c, b}(\theta; 0) = - \infty$, for some $\theta \in (0,1)$, then $\mathcal{S}_{c, \rho}^{c,b}(\theta; u)- \mathcal{S}_{c, \rho}^{c,b}(\theta; 0) = \infty$, for all $u \in (0,1)$. Hence $\mathcal{I}_{c, \rho}^{c,b}(\theta; 0) = \infty$ since $\mathcal{R}_{c, \rho}^{c,b}(\theta;  {\rm d} u)$ is a strictly positive measure on $(0, \theta)$. Since $d= c$, (\ref{eq28}) implies that $\mathcal{E}_{c, \rho}^{c,b}(\theta; 0) < \infty$, for some $\theta \in (0,1)$, and thus, $0$ is an entrance boundary which proves the first part in (iii). Suppose now that $|\mathcal{S}_{c, \rho}^{c, b}(\theta; 0)|<  \infty$. Since $d = c$ and $\Phi_{c, b}$ is bounded away from zero, for any $0 < \varepsilon_{4} < 1$, there are $\theta_{4} \in (0,\theta)$  and a constant $C_{6} >0$ such that (\ref{eq35}) implies that $|\mathcal{R}_{c, \rho}^{c,b}(\theta;0) | \leq C_{6} \int_{0}^{\theta_{4}} x^{\varepsilon_{4}} {\rm d} x < \infty$. Therefore, by noticing that $\mathcal{I}_{d, \rho}^{c,b}(\theta; u) + \mathcal{E}_{d, \rho}^{c,b}(\theta; u) = \mathcal{R}_{d, \rho}^{c,b}(\theta; u)\mathcal{S}_{d, \rho}^{c,b}(\theta; u)$, for $u \in [0,1]$, we deduce that $0$ is a regular boundary. Moreover, one can deduce from (\ref{regularR}) that $0$ is regular reflecting. This concludes with the proof of the second part in (iii). 
\end{proof}

\begin{proof}[Proof of Lemma \ref{lemma15}]
We claim that if $\varsigma < 0$, then $\mathcal{I}_{d, \rho}^{c,b}(\theta; 1) < \infty$, for some $\theta \in (0,1)$. To see this, note that $\lim_{u \rightarrow 1} \Psi_{d,\rho}(u) = 0$ and choose $\theta \in (0,1)$ such that $- \varepsilon < \Psi_{d,\rho}(u) < \varepsilon$, for $u \in [\theta,1]$. Then, (\ref{eq02}) implies that $\mathcal{Q}_{d, \rho}^{c, b}(x; y) \geq (\varepsilon/\varsigma)\ln(\frac{y}{1-y} \frac{1-x}{x})$, for all $\theta \leq x \leq y <1$. Thus, one can find a constant $C >0$ such that  $\mathcal{I}_{d, \rho}^{c, b}(\theta; 1) \leq C \int_{\theta}^{1} \int_{\theta}^{u} (1-u)^{\varepsilon/\varsigma} (1- x)^{-1-\varepsilon/\varsigma} {\rm d}x {\rm d} u < \infty$. 

Therefore, our claim in Lemma \ref{lemma15} follows from Feller's tests and the previous remark. 
\end{proof}

%%%%%%%%%%%%%%%%%%%%%%%%%%%%%%%%%%%%%%%%%%%%%%%%

\paragraph{Acknowledgements.}
GB thanks Uppsala University where most of the work was developed while he was supporting by the Knut and Alice Wallenberg Foundation, a grant from the Swedish Research Council and The Swedish Foundations' starting grant from Ragnar S\"oderbergs Foundation. GB  also acknowledges to the University of G\"ottingen and the support from the DFG-SPP Priority Programme 1590, {\sl Probabilistic Structures in Evolution}  where this work was initiated.   Part of this work was done whilst JCP was on sabbatical leave holding a David
Parkin Visiting Professorship at the University of Bath, he gratefully acknowledges the kind
hospitality of the Department and University.

%\bibliography{Gab}

\begin{thebibliography}{10}

\bibitem{Aldous1999}
D.~J. Aldous, \emph{Deterministic and stochastic models for coalescence
  (aggregation and coagulation): a review of the mean-field theory for
  probabilists}, Bernoulli \textbf{5} (1999), no.~1, 3--48. \MR{1673235}

\bibitem{Alk2007}
R.~Alkemper and M.~Hutzenthaler, \emph{Graphical representation of some duality
  relations in stochastic population models}, Electron. Comm. Probab.
  \textbf{12} (2007), 206--220. \MR{2320823}

\bibitem{Ath2004}
K.~B. Athreya and P.~E. Ney, \emph{Branching processes}, Dover Publications,
  Inc., Mineola, NY, 2004, Reprint of the 1972 original [Springer, New York;
  MR0373040]. \MR{2047480}

\bibitem{Ban2016}
V.~Bansaye, S.~M\'{e}l\'{e}ard, and M.~Richard, \emph{Speed of coming down from
  infinity for birth-and-death processes}, Adv. in Appl. Probab. \textbf{48}
  (2016), no.~4, 1183--1210. \MR{3595771}

\bibitem{Steffanson2015}
J.~E. Bj\"{o}rnberg and S.~O. Stef\'{a}nsson, \emph{Random walk on random
  infinite looptrees}, J. Stat. Phys. \textbf{158} (2015), no.~6, 1234--1261.
  \MR{3317412}

\bibitem{MAC09}
M.~E. Caballero, A.~Lambert, and G.~Uribe~Bravo, \emph{Proof(s) of the
  {L}amperti representation of continuous-state branching processes}, Probab.
  Surv. \textbf{6} (2009), 62--89. \MR{2592395}

\bibitem{Chen1997}
R.-R. Chen, \emph{An extended class of time-continuous branching processes}, J.
  Appl. Probab. \textbf{34} (1997), no.~1, 14--23. \MR{1429050}

\bibitem{Et1986}
S.~N. Ethier and T.~G. Kurtz, \emph{Markov processes}, Wiley Series in
  Probability and Mathematical Statistics: Probability and Mathematical
  Statistics, John Wiley \& Sons, Inc., New York, 1986, Characterization and
  convergence. \MR{838085}

\bibitem{Feller1966}
W.~Feller, \emph{An introduction to probability theory and its applications.
  {V}ol. {II}}, John Wiley \& Sons, Inc., New York-London-Sydney, 1966.
  \MR{0210154}

\bibitem{Cle2017}
C.~Foucart, \emph{Continuous-state branching processes with competition:
  duality and reflection at infinity}, Electron. J. Probab. \textbf{24} (2019),
  Paper No. 33, 38. \MR{3940763}

\bibitem{Pa20192}
A.~{Gonz{\'a}lez Casanova}, G.~Kersting, and J.~C. {Pardo}, \emph{{Branching
  processes with interactions: the critical cooperative regime}}, Work in
  progress (2023+).

\bibitem{Pa2017}
A.~Gonz\'{a}lez~Casanova, J.~C. Pardo, and J.~L. P\'{e}rez, \emph{Branching
  processes with interactions: subcritical cooperative regime}, Adv. in Appl.
  Probab. \textbf{53} (2021), no.~1, 251--278. \MR{4232756}

\bibitem{PaMi2019}
A.~{González Casanova}, V.~{Miró Pina}, and J.~C. Pardo, \emph{The
  wright–fisher model with efficiency}, Theoretical Population Biology
  \textbf{132} (2020), 33 -- 46.

\bibitem{Harris2002}
T.~E. Harris, \emph{The theory of branching processes}, Dover Phoenix Editions,
  Dover Publications, Inc., Mineola, NY, 2002, Corrected reprint of the 1963
  original [Springer, Berlin; MR0163361 (29 \#664)]. \MR{1991122}

\bibitem{Jagers1994}
P.~Jagers, \emph{Towards dependence in general branching processes}, Classical
  and modern branching processes ({M}inneapolis, {MN}, 1994), IMA Vol. Math.
  Appl., vol.~84, Springer, New York, 1997, pp.~127--139. \MR{1601717}

\bibitem{Jan2014}
S.~Jansen and N.~Kurt, \emph{On the notion(s) of duality for {M}arkov
  processes}, Probab. Surv. \textbf{11} (2014), 59--120. \MR{3201861}

\bibitem{Ka82}
A.~V. Kalinkin, \emph{On the probability of the extinction of branching process
  with interaction of particles}, Teor. Veroyatnost. i Primenen. \textbf{27}
  (1982), no.~1, 201--205.

\bibitem{Ka99}
A.~V. Kalinkin, \emph{Final probabilities of a branching process with interaction of
  particles, and the epidemic process}, Teor. Veroyatnost. i Primenen.
  \textbf{43} (1998), no.~4, 773--780. \MR{1692393}

\bibitem{Ka01}
A.~V. Kalinkin, \emph{Exact solutions of the {K}olmogorov equations for critical
  branching processes with two complexes of particle interaction}, Uspekhi Mat.
  Nauk \textbf{56} (2001), no.~3(339), 173--174. \MR{1859738}

\bibitem{Ka02}
A.~V. Kalinkin, \emph{On the probability of the extinction of a branching process with
  two complexes of particle interaction}, Teor. Veroyatnost. i Primenen.
  \textbf{46} (2001), no.~2, 376--381. \MR{1968693}

\bibitem{Ka02-1}
A.~V. Kalinkin, \emph{Markov branching processes with interaction}, Uspekhi Mat. Nauk
  \textbf{57} (2002), no.~2(344), 23--84. \MR{1918194}

\bibitem{Kalla2002}
O.~Kallenberg, \emph{Foundations of modern probability}, second ed.,
  Probability and its Applications (New York), Springer-Verlag, New York, 2002.
  \MR{1876169}

\bibitem{Taylor1981}
S.~Karlin and H.~M. Taylor, \emph{A second course in stochastic processes},
  Academic Press, Inc. [Harcourt Brace Jovanovich, Publishers], New
  York-London, 1981. \MR{611513}

\bibitem{Kingman1982}
J.~F.~C. Kingman, \emph{The coalescent}, Stochastic Process. Appl. \textbf{13}
  (1982), no.~3, 235--248. \MR{671034}

\bibitem{Kyprianou2017}
A.~E. Kyprianou, S.~W. Pagett, T.~Rogers, and J.~Schweinsberg, \emph{A phase
  transition in excursions from infinity of the ``fast''
  fragmentation-coalescence process}, Ann. Probab. \textbf{45} (2017), no.~6A,
  3829--3849. \MR{3729616}

\bibitem{La2005}
A.~Lambert, \emph{The branching process with logistic growth}, Ann. Appl.
  Probab. \textbf{15} (2005), no.~2, 1506--1535. \MR{2134113}

\bibitem{Lamb2008}
A.~Lambert, \emph{Population dynamics and random genealogies}, Stoch. Models
  \textbf{24} (2008), no.~suppl. 1, 45--163. \MR{2466449}
\bibitem{Leman}
H.~Leman and J.~C.~Pardo, \emph{Extinction time of logistic branching processes in a  {B}rownian environment}, ALEA Lat. Am. J. Probab. Math. Stat., \textbf{18} (2021), no.2, 1859--1890. \MR{4332223}
    
\bibitem{Lindvall1992}
T.~Lindvall, \emph{Lectures on the coupling method}, Wiley Series in
  Probability and Mathematical Statistics: Probability and Mathematical
  Statistics, John Wiley \& Sons, Inc., New York, 1992, A Wiley-Interscience
  Publication. \MR{1180522}

\bibitem{MT93}
S.~P. Meyn  and R.~L. Tweedie, \emph{Stability of Markovian processes III: Foster-Lyapunov criteria for continuous-time processes.}, Adv. in Appl.
  Probab. \textbf{25} (1993), 518--548. 

\bibitem{Norris1998}
J.~R. Norris, \emph{Markov chains}, Cambridge Series in Statistical and
  Probabilistic Mathematics, vol.~2, Cambridge University Press, Cambridge,
  1998, Reprint of 1997 original. \MR{1600720}

\bibitem{Pakes2007}
A.~G. Pakes, \emph{Extinction and explosion of nonlinear {M}arkov branching
  processes}, J. Aust. Math. Soc. \textbf{82} (2007), no.~3, 403--428.
  \MR{2318173}

\bibitem{Revuz1999}
D.~Revuz and M.~Yor, \emph{Continuous martingales and {B}rownian motion}, third
  ed., Grundlehren der Mathematischen Wissenschaften [Fundamental Principles of
  Mathematical Sciences], vol. 293, Springer-Verlag, Berlin, 1999. \MR{1725357}

\bibitem{Sturm2015}
A.~Sturm and J.~M. Swart, \emph{A particle system with cooperative branching
  and coalescence}, Ann. Appl. Probab. \textbf{25} (2015), no.~3, 1616--1649.
  \MR{3325283}

\bibitem{Temme75}
N.~M. Temme, \emph{Uniform asymptotic expansions of the incomplete gamma
  functions and the incomplete beta function}, Math. Comp. \textbf{29} (1975),
  no.~132, 1109--1114. \MR{387674}
\end{thebibliography}
%\bibliographystyle{Myplain5}

\providecommand{\bysame}{\leavevmode\hbox to3em{\hrulefill}\thinspace}
\providecommand{\MR}{\relax\ifhmode\unskip\space\fi MR }
% \MRhref is called by the amsart/book/proc definition of \MR.
\providecommand{\MRhref}[2]{%
  \href{http://www.ams.org/mathscinet-getitem?mr=#1}{#2}
}
\providecommand{\href}[2]{#2}

\end{document}